\documentclass[11pt]{amsart}

\usepackage{amsmath,amssymb,amscd,amsthm,graphicx,enumerate}
 \usepackage{graphicx}  
 \usepackage{ 
   float}
\usepackage{hyperref}
\usepackage{xypic}
\usepackage{xcolor}
\usepackage{color}
\usepackage{mathtools}
\usepackage{mathrsfs}
\usepackage{tikz-cd}
\usepackage{tikz}
\usetikzlibrary{shapes.geometric}
\hypersetup{breaklinks=true}
\hypersetup{
    colorlinks=true,
    linkcolor=red,
    filecolor=magenta,      
    urlcolor=cyan,
    citecolor=green,
}

\usetikzlibrary{calc}
\usepackage{tikz-3dplot}
\newcommand{\cylinders}{%
    \begin{tikzpicture}[scale=0.13, baseline=-0.5ex]
        \draw (0,1.5) -- (1.5,0) -- (0,-1.5) -- (-1.5,0) -- cycle;
    \end{tikzpicture}
}

\newcommand{\simplex}{%
    \begin{tikzpicture}[scale=0.13, baseline=-0.5ex]
        \draw (0,1.5) -- (1.299,-0.75) -- (-1.299,-0.75) -- cycle;
    \end{tikzpicture}
}

\newcommand{\tetra}{%
  \begin{tikzpicture}[scale=0.13, baseline=-0.5ex]
        \draw (0,1.4) -- (1.299,-0.6) -- (-1.299,-0.6) --  (0,1.4) -- (-0.5, - 1.5)--(1.299,-0.6)--(-1.299,-0.6)-- (-0.5, - 1.5);
    \end{tikzpicture}
}

\numberwithin{equation}{section}
\theoremstyle{plain}

\newcommand{\be}{\begin{equation}}
\newcommand{\bee}{\begin{equation*}}
\newcommand{\bea}{\begin{equation}\begin{aligned}}
\newcommand{\eea}{\end{aligned}\end{equation}}
\newcommand{\ee}{\end{equation}}
\newcommand{\eee}{\end{equation*}}
\newcommand{\bsp}{\begin{split}}
\newcommand{\esp}{\end{split}}

\newtheorem{theorem}{Theorem}[section]
\newtheorem{proposition}[theorem]{Proposition}

\newtheorem{corollary}[theorem]{Corollary}
\newtheorem{conjecture}[theorem]{Conjecture}
\newtheorem{claim}[theorem]{Claim}

\theoremstyle{remark}
\newtheorem{remark}[theorem]{Remark}

\theoremstyle{definition}
\newtheorem{definition}[theorem]{Definition}

\newcommand{\cC}{\mathcal{C}}

\renewcommand{\cH}{\mathcal{H}}

\newcommand{\fp}{\mathfrak{p}}

\graphicspath{{./mathdraw/}}

\begin{document}

\title[Ancient ovals in higher dimensional mean curvature flow]{Classification of ancient ovals in higher dimensional mean curvature flow}

\author[Beomjun Choi,  Wenkui Du, Ziyi Zhao]{Beomjun Choi,  Wenkui Du, Ziyi Zhao}

\begin{abstract}
We study compact non-selfsimilar ancient noncollapsed  solutions to the mean curvature flow in $\mathbb{R}^{n+1}$, called ancient ovals. Our main result is the classification of $k$-ovals: any $k$-oval (characterized by having cylindrical blow down $\mathbb{R}^k\times S^{n-k}$ and the quadratic bending asymptotics) belongs, up to space-time rigid motions and parabolic dilations, to the family of ancient ovals constructed by Haslhofer and the second author. Assuming the nonexistence of exotic ovals (recently proved by Bamler--Lai), this yields a classification of all ancient ovals and identifies the moduli space, modulo symmetries, with an open $(k-1)$-simplex modulo the symmetry of simplex. Although these conclusions are contained in the recent breakthrough of Bamler--Lai classifying all ancient asymptotically cylindrical flows and resolving the mean convex neighborhood conjecture, we give an alternative argument for the independently obtained classification of $k$-ovals in arbitrary dimensions based on a different spectral parametrization. 
\end{abstract}

\maketitle

\tableofcontents

\section{Introduction}\label{classification} 

In recent years, the analysis of singularity formation in geometric evolution equations such as the mean curvature flow and the Ricci flow has achieved significant progress through the classification of potential singularity models, namely, ancient solutions that may arise as limit flows near a singularity. This approach combines powerful tools from PDE and analysis with geometric intuition. 

Very recently, a major breakthrough was achieved by Bamler and Lai \cite{BL1,BL2}, where they established the classification of  ancient asymptotically cylindrical solutions to the mean curvature flow in \emph{all dimensions}, thereby concluding a line of research in the classification program. The remarkable results of Bamler and Lai \cite{BL1,BL2} nontrivially extends the previous results towards the complete classification of  ancient noncollapsed solutions {\cite{BC1, BC2, ADS18, DH_ovals, DH_hearing_shape, DH_no_rotation, CHH_wing, DZ_spectral_quantization, Zhu, CDDHS, CH-classification-r4, CDZ_ovals}} and the  classification of asymptotically cylindrical solutions \cite{CHH, CHHW, CHH_revisiting}. Moreover, as an application of the classification of ancient asymptotically cylindrical solutions, they established the mean-convex neighborhood theorem around cylindrical singularity in all dimensions conjectured by Ilmanen \cite{Ilmanen_lectures}. Compared to previous works, the  approach of Bamler and Lai is notable in that it treats this essentially noncompact classification problem in a manner akin to the classification of compact ancient solutions. Also they address the problem from the outset without imposing geometric hypotheses such as the noncollapsing condition. Moreover, their novel PDE-ODI principle is expected to  have strong impact that goes beyond mean curvature flow.

The aim of the present paper is to give an alternative argument for the classification of ancient ovals in arbitrary dimensions, which we have obtained independently. We note that the classification of ancient ovals is also contained in the recent breakthrough by Bamler and Lai \cite{BL1,BL2}, see Remark \ref{remark-BLandours} below. Ancient ovals have been studied intensively over the last ten years \cite{ADS1, ADS18, White_nature, HaslhoferHershkovits_ancient,  DH_ovals, CDDHS, CDZ_ovals}, so we hope our alternative argument is of interest in particular to more analytically inclined readers.
Our first goal is to prove the classification of ancient ovals, which refers to  non-selfsimilar, compact ancient noncollapsed solutions. An important subclass of ancient ovals is the family of $k$-ovals \cite[Definition 1.7]{DZ_spectral_quantization} whose solution profiles have quadratic bending in all directions, which have been expected to exhaust all ancient ovals. We show that a given $k$-oval belongs to an ancient oval constructed by  Haslhofer and the second author in \cite{DH_ovals}; see Theorem \ref{classification_theorem}. Assuming there is no ancient oval that is not a $k$-oval (see Conjecture \ref{oval-koval}, which is now a theorem by \cite{BL1,BL2}), we show that the solutions constructed in \cite{DH_ovals} give a complete list of ancient ovals. Moreover, we show that the space of ancient ovals, endowed with the topology of smooth local convergence and modulo space-time rigid motions and parabolic dilatations, is homeomorphic to a $(k-1)$-dimensional simplex modulo the symmetry of the simplex; see Theorem \ref{cor_moduli}. See Remark \ref{remark-BLandours} for how our results follows as a consequence of \cite{BL1,BL2}.

Let us recall the class of ancient ovals constructed in \cite{DH_ovals}, which vanish at space-time origin and have $\mathbb{R}^k\times S^{n-k}(\sqrt{2(n-k)|t|})$ as their tangent flow at $-\infty$. Specifically, we consider $a^j\geq 0$ and $
a^1 + a^2 + \cdots+ a^k = 1$.
Then for any such $a=(a^{1}, \dots, a^{k})$ and any $\ell<\infty$ one considers the ellipsoid
\begin{equation}
E^{\ell,a}:=\left\{x\in\mathbb{R}^{n+1}: \sum_{j=1}^{k}\frac{(a^{j})^2}{\ell^2}x_j^2 + \sum_{j=k+1}^{n+1}x_j^2 = 2(n-k) \right\}\, ,
\end{equation} and the mean curvature flow starting from the ellipsoid, namely $E^{\ell,a}_t$. One then chooses time-shifts $t^{\ell,a}$ and dilation factors $\lambda^{\ell,a}$ so that the flow
\begin{equation}
M^{\ell,a}_{t} := \lambda^{\ell,a} \cdot E^{\ell,a}_{ {(\lambda^{\ell,a}})^{-2} t+t^{\ell,a}}
\end{equation}
becomes extinct at time $0$ and satisfies 
\begin{equation}
\int_{M^{\ell,a}_{-1}}\frac{1}{(4\pi)^{n/2}}e^{-\frac{|x|^2}{4}}=\frac{1}{2}(\text{Ent}(S^{n-k}\times \mathbb{R}^{k})+\text{Ent}(S^{n-k+1}\times \mathbb{R}^{k-1})).
\end{equation}

Considering sequences $a_i\in (0, 1)$  and $\ell_i\to \infty$, a subclass of ancient ovals is then constructed by
\begin{equation}\label{Ao}
\mathcal{A}^{k, \circ}:=\left\{ \lim_{i\to \infty} M^{\ell_i,a_i}_t : \textrm{the limit along $a_i,\ell_i$ exists and is compact}\right\}\, .
\end{equation}
Naturally ancient flows in $\mathcal{A}^{k, \circ}$ have $\mathbb{Z}_2^k\times O(n-k+1)$ symmetry. Moreover \cite{DH_blowdown} shows that for every prescribed geometric width ratios in the first $k$ axes, there exists a solution in $\mathcal{A}^{k, \circ}$ which has the prescribed ratios at $t=1$. Also it was shown that solutions in $\mathcal{A}^{k, \circ}$ are noncollapsed.

\subsection{Main results}
Our main result is the classification of $k$-ovals extending the  rigidity result of $k$-ovals established in  \cite{CDZ_ovals}. Roughly speaking, an ancient oval is a $k$-oval if its tangent flow at $-\infty$ is $\mathbb{R}^k \times \mathbb{S}^{n-k}(\sqrt {2(n-k)})$  and the higher order asymptotics around the cylinder exhibit quadratic bending in every cylindrical direction. For the precise definition, see \cite[Definition 1.7]{DZ_spectral_quantization}. By the recent result of Bamler-Lai in \cite{BL1,BL2}, we now know that an ancient oval is necessarily a $k$-oval for some $k=1,\ldots, n-1$.  We now state our main theorem:

\begin{theorem}[classification of $k$-ovals]\label{classification_theorem}
Any $k$-oval in $\mathbb{R}^{n+1}$  belongs, up to space-time rigid motion and parabolic dilation, to the class  $\mathcal{A}^{k, \circ}$.
\end{theorem}
Next, let us consider the moduli space of ancient ovals\begin{multline}
\mathcal{X}=\big\{ \mathcal{M} \textrm{ is ancient oval in $\mathbb{R}^{n+1}$}
\textrm{ with blowdown $\mathbb{R}^{k}\times S^{n-k}$} \big\}/\sim,
\end{multline}
where the blowdown is the tangent flow at $-\infty$ and the topology is the one induced by locally smooth convergence, and we mod out by space-time rigid motions (which are actions by $O(n+1)$ and translations in the space-time) and parabolic dilations.

Motivated by \cite{CHH, CHH_translator, CH-classification-r4}, experts in the field expect the following conjecture, which
was recently answered affirmatively by Bamler and Lai.
\begin{conjecture}[nonexistence of exotic ovals, now theorem by \cite{BL1,BL2}]\label{oval-koval}
Any ancient oval in $\mathbb{R}^{n+1}$ whose tangent flow at $-\infty$ is $\mathbb{R}^{k}\times S^{n-k}$ is a $k$-oval.
\end{conjecture}
Assuming Conjeture \ref{oval-koval}, Theorem \ref{classification_theorem} (classification of $k$-ovals) implies the following theorem on the classification of ancient ovals and the topology of moduli space:
\begin{theorem}[classification of ancient ovals and moduli space]\label{cor_moduli}
Under assuming Conjeture \ref{oval-koval}, all ancient ovals in $\mathbb{R}^{n+1}$ belong to the families of $\mathbb{Z}_2^{k}\times O(n+1-k)$ symmetric ancient ovals for some $k=1, \dots, n-1$ constructed in \cite{DH_ovals}. 
They are $k$-ovals and the moduli space of ancient ovals $\mathcal{X}$ has the following homeomorphism
\begin{equation}\label{modulo homeomorphism0}
     \mathcal{X} \cong \Delta_{k-1} /\mathbf{S}_k, 
 \end{equation}
where \[\Delta_{k-1} = \{(a^1,\ldots,a^k)\in \mathbb{R}^k\,:\,  \sum_{i=1}^k a^i=1,\,  a^i>0\} ,\] which is an open $k-1$ dimensional simplex and $\mathbf{S}_k\subset O(k)$ is the group of permutations among $k$ coordinate axes. 
\end{theorem}

The proofs of Theorem \ref{classification_theorem} and Theorem \ref{cor_moduli} will be given in Section \ref{sec-classification and moduli}.

Let us briefly explain our method and main challenge compared to the previous classification and the theorem on the topology of moduli space for $2$-ovals considered in \cite{CDDHS} or $1$-oval in \cite{ADS18}. Our proof of classification of $k$-ovals and moduli space is based on a different parametrization and the  argument is motivated from \cite{CDDHS, CDZ_ovals}, the spectral map $\mathcal{E}(\mathcal{M})$ defined by
\begin{equation} \label{eq-Eintro}
    \mathcal{E}(\mathcal{M})=\left(\frac{\left\langle v_{\cC}^{\mathcal{M}}\left(\tau_0\right), y_1^2-2\right\rangle_{\mathcal{H}}}{\left\langle v_{\cC}^{\mathcal{M}}\left(\tau_0\right),|\mathbf{y}|^2-2 k\right\rangle_{\mathcal{H}}}, \ldots, \frac{\left\langle v_{\cC}^{\mathcal{M}}\left(\tau_0\right), y_k^2-2\right\rangle_{\mathcal{H}}}{\left\langle v_{\cC}^{\mathcal{M}}\left(\tau_0\right),|\mathbf{y}|^2-2 k\right\rangle_{\mathcal{H}}}\right),
\end{equation}
where $\tau_0$ is a sufficiently negative time, $v^{\mathcal{M}}_{\mathcal{C}}$ is  suitable truncated profile function of the $k$-oval $\mathcal{M}$ obtained by suitably time-shift and parabolically dilation such that certain recentering conditions hold (see \eqref{quadractic_mode1} and \eqref{positive_mode1} for details). 
The major challenge for classification is that it is difficult to find global continuous recentering parameters\footnote{When $k=2$, spectral map is one dimensional and as in \cite{CDDHS}   the local continuous recentering condition, intermediate value theorem and the classification of one dimensional manifolds are enough for classification and moduli space characterization, which only works for one-parameter family of $2$-ovals. Thus we need a new approach dealing with the classification of higher parameter family of $k$-ovals.} of ancient solutions
by simultaneous time-shift and parabolic dilation. To settle this, we add parabolic dilation parameter and correspondingly extend $\mathcal{A}^{k, \circ}$ to $\mathscr{A}^{k,\circ}$, which is one dimension higher than $\mathcal{A}^{k, \circ}$ (see \eqref{mathscrA} for precise definition). Then we show the quadratic bending spectral mode is monotone along parabolic scaling of $O(k)\times O(n+1-k)$
symmetric oval, which means the function
\begin{equation}
    f(\gamma)=\left\langle v_{\cC}^{\mathcal{M}^{\beta(\gamma),\gamma}}\left(\tau_0\right), y_1^2-2\right\rangle_{\mathcal{H}}
\end{equation}
is monotone, where $\mathcal{M}^{\beta(\gamma), \gamma}$ is the flow obtained by $e^{\gamma/2}$-parabolic-dilation and $\beta(\gamma)$-time-shift. Here $\beta(\tau)$ is introduced to satisfy \eqref{quadractic_mode1} and \eqref{positive_mode1} and we show $\beta(\gamma)$ is a continuous function of $\gamma$ (see Claim \ref{Rescaling monotonicity} for details). Then we inductively apply the mapping degree argument
for spectral map
\begin{align}
\mathscr{E}(\mathcal{M})=\left({\Big\langle v^{\mathcal{M}}_{\cC}(\tau_{0}),  2-|{\bf{y}}|^2_{1} \Big\rangle_\cH},\dots, {\Big\langle v^{\mathcal{M}}_{\cC}(\tau_{0}),  2-|{\bf{y}}|^2_{k} \Big\rangle_\cH}\right),
\end{align}
on space of time-shifted $\mathscr{A}^{k,\circ}$ and spectral uniqueness theorem previously shown in \cite[Theorem 1.3]{CDZ_ovals} to show it is a homeomorphism. Then we can apply this  to find simultaneous continuous time-shift parameter $\beta(\mathcal{M})$ and parabolic dilation parameter $\gamma(\mathcal{M})$ and obtain the coordinates map $\mathcal{E}(\mathcal{M}^{\beta(\mathcal{M}), \gamma(\mathcal{M})})$ of moduli space of $k$-ovals (see Corollary \ref{prescribed_eccentricity_restated'} for details).  Our argument is also expected effective for classifying noncollapsed translators of mean curvature flows and classifying ancient $\kappa$-noncollapsed   Ricci flows.

\begin{remark} \label{remark-BLandours} Here we point out and outline how Conjecture \ref{oval-koval} and Theorem \ref{cor_moduli} follow from \cite{BL1,BL2}.

According to \cite{BL1}, the ancient solutions that are asymptotic to $\mathbb{R}^k\times \mathbb{S}^{n-k}$, become extinct at the space--time origin, and are not self-similarly translating are denoted by $\mathrm{MCF}_{\mathrm{oval}}^{n,k}$; see Theorem \cite[Theorem 1.3]{BL1} and the definition of $\mathrm{MCF}_{\mathrm{oval}}^{n,k}$. The $\mathbb{Z}^k_2\times O(n+1-k)$ symmetry of ancient ovals is described in the definition of $\mathrm{MCF}_{\mathrm{oval}}^{n,k}$ and \cite[(b) in Theorem 1.4]{BL1}. Bamler--Lai defined a canonical map, called the \emph{quadratic mode at $-\infty$}
\cite[Definition~7.7]{BL2}, for asymptotically $\mathbb{R}^k \times \mathbb{S}^{n-k}$ ancient solutions
that become extinct at the space--time origin. This map yields a novel global parametrization of
$\mathrm{MCF}_{\mathrm{oval}}^{n,k}$. More precisely, the restriction
\[
\textrm{Q}\big|_{\mathrm{MCF}_{\mathrm{oval}}^{n,k}}:\mathrm{MCF}_{\mathrm{oval}}^{n,k} \to \mathbb{R}^{k\times k}_{\ge 0}
\]
is a bijective homeomorphism \cite[(a) in Theorem~1.4]{BL1}.\footnote{
In our previous approach, our spectral map $\mathcal{E}$ in \eqref{eq-Eintro} was defined via the spectral projection of truncated profile onto the
quadratic modes at a sufficiently negative time, for solutions that are sufficiently close at that time
to an $O(k)\times O(n+1-k)$-symmetric ancient oval. The map $\textrm{Q}$ can be viewed as an extension of $\mathcal{E}$
to all ancient solutions, in a canonical way that respects the ambient space-time symmetries.}
Here $\mathbb{R}^{k\times k}_{\ge 0}$ denotes the space of symmetric nonnegative definite matrices.  Moreover, if $\textrm{Q}(\mathcal{M})$ has an $l$-dimensional null space, then $\mathcal{M}\in \mathrm{MCF}_{\mathrm{oval}}^{n,k}$ splits as a product of a flow $\mathcal{M}'\in \mathrm{MCF}_{\mathrm{oval}}^{n-l,k-l}$ and an $\mathbb{R}^l$-factor in the null-space directions. Those $\mathcal{M}\in \mathrm{MCF}_{\mathrm{oval}}^{n,k}$ with positive definite $\textrm{Q}(\mathcal{M})$ correspond to the space of ancient ovals. Indeed, by the fine asymptotics \cite[(d) in Proposition 7.1]{BL2}, if $\mathcal{M}\in \mathrm{MCF}_{\mathrm{oval}}^{n,k}$ has positive definite $\textrm{Q}(\mathcal{M})$, then $\mathcal{M}$ exhibits quadratic bending in all directions in $\mathbb{R}^k$. Since solutions are noncollapsed and convex \cite[Theorem 1.2]{BL1}, this implies that $\mathcal{M}\in \mathrm{MCF}_{\mathrm{oval}}^{n,k}$ with positive definite $\textrm{Q}(\mathcal{M})$ is necessarily a $k$-oval. Conjecture \ref{oval-koval} (nonexistence of exotic ovals) holds since all other $\mathcal{M}\in \mathrm{MCF}_{\mathrm{oval}}^{n,k}$ with degenerate $\textrm{Q}(\mathcal{M})$ are non-compact solutions.

The topology of the moduli space is also a consequence of \cite[Theorem 1.4]{BL1}, together with other results in the papers. Since parabolic rescaling acts equivariantly on $\textrm{Q}(\cdot)$ \cite[Proposition 7.8]{BL2}, we may rescale and assume $\mathrm{tr}(\textrm{Q}(\mathcal{M}))=1$. If we act on $\mathcal{M}$ by an orthogonal transformation in $O(k)$ (in the $\mathbb{R}^{k}$-factor), then $\textrm{Q}(\mathcal{M})$ changes by conjugation with the corresponding orthogonal matrix \cite[Proposition 7.8]{BL2}. Thus, after a rigid motion (orthogonal transformation) in space, we may assume that $\textrm{Q}(\mathcal{M})$ is diagonal, say $\textrm{Q}(\mathcal{M})=\mathrm{diag}(a^1,\ldots, a^k)$ with $a^i>0$ and $a^1+\cdots+a^k=1$. Since we can permute the eigenvalues by another orthogonal transformation, it follows that the topology of $\mathrm{MCF}_{\mathrm{oval}}^{n,k}$ modulo the action of $O(k)$ and parabolic dilations is $\Delta_{k-1}/\mathbf{S}_k$. This proves \eqref{modulo homeomorphism0}.

\end{remark}

\textbf{Acknowledgments.}
The first author has been supported by the National Research Foundation of Korea (NRF) grant funded by the Korean government (MSIT)  NRF-2022R1C1C1013511, RS-2023-00219980 and by Samsung Science \& Technology Foundation grant SSTF-BA2302-02.
The second author has been supported by Massachusetts Institute of Technology. The third author has been supported by Institute for Theoretical Sciences at Westlake University.  The authors
appreciate the communication with Professor Panagiota Daskalopoulos,
Professor Robert Haslhofer and Professor Natasa Sesum.

\section{Preliminaries}
 In this section, we recall some definitions and results which will be used in later sections. 

 For $1\le k\le n-1$, let $\mathcal{M}=\{M_t\}$ be an ancient noncollapsed flow having the tangent flow $\mathbb{R}^k\times S^{n-k}(\sqrt{2(n-k)|t|})$ at $-\infty$. This means 
\begin{equation}\label{bubble-sheet_tangent_intro}
\lim_{\tau\to -\infty}\bar{M}_\tau=\lim_{\tau\to -\infty}e^{\frac{\tau}{2}} M_{-e^{-\tau}}=\mathbb{R}^{k}\times S^{n-k}(\sqrt{2(n-k)}).
\end{equation} 
We express the renormalized flow $\bar M_\tau$ by the profile function \( v(\mathbf{y}, \vartheta,\tau) \) defined on the domains $(\mathbf{y},\vartheta) \in \Gamma_\tau$ that exhaust the cylinder $\mathbb{R}^{k}\times S^{n-k}$ via
\begin{equation}\label{profile v def}
    \bar{M}_\tau =\{(\mathbf{y},v(\mathbf{y},\vartheta,\tau)\vartheta )\in \mathbb{R}^k \times \mathbb{R}^{n+1-k}\,:\, (\mathbf{y},\vartheta)\in \Gamma_\tau  \}.
\end{equation}
The asymptotics \eqref{bubble-sheet_tangent_intro} reads to the smooth local convergence of $v(\mathbf{y},\vartheta,\tau)$ to $\sqrt{2(n-k)}$ as $\tau$ approaches $-\infty$. The $k$-oval is defined in terms of the next-order asymptotics which describes $v(\tau)$ up to an error of order $o(|\tau|^{-1})$. 

\begin{definition}[$k$-oval, c.f.{\cite[Definition 1.7]{DZ_spectral_quantization}}]\label{def-koval} An ancient noncollapsed flow $\mathcal{M}$ in $\mathbb{R}^{n+1}$ is a $k$-oval  if it has the tangent flow $\mathbb{R}^{k}\times S^{n-k}(\sqrt{2(n-k)|t|})$ at $-\infty$ and the renormalized profile $v(\mathbf{y},\vartheta,\tau)$ in \eqref{profile v def} has the asymptotics 
\[\lim_{\tau\to -\infty} \bigg \Vert |\tau|(v(\mathbf{y},\vartheta,\tau) -\sqrt{2(n-k)})+ \frac{\sqrt{2(n-k)}}{4}(|\mathbf{y}|^2-2k)\bigg \Vert_{C^m(\{|\mathbf{y}|<R\})} =0,\] for all $m\in \mathbb{N}$ and $0<R<+\infty$.
\end{definition}

We consider Gaussian weighted $L^2$-space $\mathcal{H}$ and it admits a spectral decomposition 
\begin{equation}
\mathcal{H}=L^2\big(\mathbb{R}^k,e^{-\frac{|{\bf{y}}|^2}{4}}d{\bf{y}})= \mathcal{H}_+\oplus \mathcal{H}_0\oplus \mathcal{H}_-,
\end{equation}
with respect to Ornstein-Ulenbeck operator $\mathcal{L}=\Delta_{\bf{y}}-\frac{1}{2}{\bf{y}}\cdot  \nabla_{\bf{y}}+1$,
where the positive eigenspace and neutral eigenspace are explicitly given by
\begin{equation}\begin{aligned}\label{H0+}
 &   \mathcal{H}_+=\textrm{span}\big\{1, y_1,\dots, y_k\}, \\ 
&\mathcal{H}_0=\textrm{span}\big\{y_1^2-2, \dots, y_k^2-2, y_{i}y_{j}, 1\leq i<j\leq k\big\}.\end{aligned}
\end{equation}
Let \(\mathfrak{p}_{0}\) and \(\mathfrak{p}_{\pm}\) be the orthogonal projections onto \(\mathcal{H}_0\) and \(\mathcal{H}_\pm\), respectively.

Then we discuss $\kappa$-quadraticity and strong $\kappa$-quadraticity.
\begin{definition}[$\kappa$-quadraticity at $\tau_0$]\label{k_tau00_intro}
A $k$-oval in $\mathbb{R}^{n+1}$ 
is called \emph{$\kappa$-quadratic at time $\tau_{0}$}, where $\kappa>0$ and $\tau_0<0$, if its truncated renormalized profile function $v_{\cC}$ satisfies 
\begin{equation}\label{condition1intro}
    \left\| v_{\cC}(\mathbf{y}, \tau_{0})-\sqrt{2(n-k)}+\frac{\sqrt{2(n-k)}}{4|\tau_{0}|}(|\mathbf{y}|^2-2k)\right\|_{\mathcal{H}}\leq \frac{\kappa}{|\tau_{0}|},
\end{equation}
and moreover graphical radius condition hold:
\begin{equation}\label{condition3}
        \sup_{\tau\in [2\tau_{0}, \tau_{0}]} |\tau|^{\frac{1}{50}} \|v(\cdot ,\tau)-\sqrt{2(n-k)}\|_{C^{4}(B(0, 2|\tau|^{{1}/{100}}))}\leq 1.
\end{equation}
\end{definition}
\begin{definition}[{strong  $\kappa$-quadraticity, c.f. \cite[Definition 3.7]{CHH_translator}}]\label{strong}
We say that an ancient noncollapsed flow $\mathcal{M}$ in $\mathbb{R}^{n+1}$ (with coordinates and tangent flow $-\infty$ as above) is \emph{strongly $\kappa$-quadratic from time $\tau_{0}$},  if 
\begin{enumerate}[(i)]
\item $\rho(\tau)=|\tau|^{1/10}$ is an admissible graphical radius for $\tau\leq \tau_{0}$, and
\item  the truncated graphical function $\hat{u}(\cdot,  \tau)=u(\cdot,  \tau)\chi\left(\frac{|\cdot|}{\rho(\tau)}\right)$ satisfies 
\begin{equation}
    \left\| \hat{u}({\bf{y}},  \tau)+\frac{\sqrt{2(n-k)}}{4|\tau|}{(|{\bf{y}}|^2-2k)} \right\|_{\mathcal{H}}\leq \frac{\kappa}{|\tau|}\quad \text{for}\,\,\tau\leq \tau_{0}.
\end{equation}
\end{enumerate}
\end{definition}
The relation between $\kappa$-quadraticity and strong $\kappa$-quadraticity is stated below.
\begin{theorem}[$\kappa$-quadraticity at one time  
 implies strong $\kappa$-quadraticity {c.f. \cite[Theorem 2.4]{CDZ_ovals}}]\label{point_strong}
    For every $\kappa>0$, there exist $\kappa'>0$ and $\tau_{*}>-\infty$, such that if a $k$-oval in $\mathbb{R}^{n+1}$ is $ \kappa'$-quadratic at some time $\tau_{0}\leq \tau_{*}$, then it is strongly $\kappa$-quadratic from time $\tau_{0}$.
\end{theorem}

The nonuniform sharp asymptotics in \cite[Theorem 1.8]{DZ_spectral_quantization} implies that, given any $\kappa>0$, every $k$-oval $\mathcal{M}$ in $\mathbb{R}^{n+1}$ is, after suitable recentering,\footnote{For the detailed recentering argument, please see Proposition \ref{prop_orthogonality} in Section \ref{classification}.} $\kappa$-quadratic at some $\tau_0=\tau_0(\mathcal{M},\kappa)\ll 0$. With above preparation, we state  spectral uniqueness theorem below, which is an important ingredient for classification.

\begin{theorem}[spectral uniqueness {c.f. \cite[Theorem 1.3]{CDZ_ovals}}]\label{spectral uniqueness restated} There exist $\kappa>0$ and $\tau_{\ast}>-\infty$ with the following significance:
If $\mathcal{M}^1$ and $\mathcal{M}^2$ are $k$-ovals in $\mathbb{R}^{n+1}$ that are $\kappa$-quadratic at time $\tau_0$, where $\tau_0 \leq \tau_{\ast}$, and if their trancated cylindrical profile functions $v^1_{\cC}$ and $v^2_{\cC}$ satisfy
\begin{equation}\label{spec_ecc_intro}
\fp_{0}(v^1_{\cC}(\tau_0))=\fp_{0}(v^2_{\cC}(\tau_0)),
\end{equation}
and
\begin{equation}
      \mathfrak{p}_{+}v^{1}_{\cC}(\tau_{0})=  \mathfrak{p}_{+}v^{2}_{\cC}(\tau_{0}),
    \end{equation}
then
\begin{equation}
\mathcal{M}^1=\mathcal{M}^2.
\end{equation}
\end{theorem}
Below, we also reformulate the construction of class of ancient ovals \cite{DH_ovals} including the scaling effect in details. Specifically, to construct this family, we say $a=(a^{1}, \dots, a^{k})\in \Delta_\mu^{k-1}$ if $a^j\geq 0$ and $
a^1 + a^2 + \cdots+ a^k = \mu\in (0, \infty)$.
Then for any $a=(a^{1}, \dots, a^{k})\in \Delta_\mu^{k-1}$ and any $\ell<\infty$ one considers the ellipsoid
\begin{equation}
E^{\ell,a}_\mu:=\left\{x\in\mathbb{R}^{n+1}: \sum_{j=1}^{k}\frac{(a^{j})^2}{\ell^2}x_j^2 + \sum_{j=k+1}^{n+1}\mu^2x_j^2 = 2(n-k) \right\}\, .
\end{equation}
One then chooses time-shifts $t_\mu^{\ell,a}$ and dilation factors $\lambda_\mu^{\ell,a}$ so that the flow
\begin{equation}\label{m_ell_a}
M^{\ell,a}_{\mu,t} := \lambda_\mu^{\ell,a} \cdot E^{\ell,a}_{\mu, {(\lambda_\mu^{\ell,a}})^{-2} t+t_\mu^{\ell,a}}
\end{equation}
becomes extinct at time $0$ and satisfies 
\begin{equation}\label{bubblesheetcond}
\int_{M_{\mu,-\mu^{-2}}^{\ell,a}}\frac{1}{(4\pi)^{n/2}}e^{-\frac{|x|^2}{4}}=\frac{1}{2}(\text{Ent}(S^{n-k}\times \mathbb{R}^{k})+\text{Ent}(S^{n-k+1}\times \mathbb{R}^{k-1})).
\end{equation}

For each $a_i=(a_i^j)_{j=1}^k$ with $\mu=\sum_{j=1}^k a_i^j$ and $\ell_i\in (0,\infty)$ we define the corresponding ellipsoidal flow $\mathcal{M}_{a_i, \ell_i}:=\{{M}^{a_i, \ell_i}_{\mu, t}\}$.

Then we define
\begin{align}\label{mathscrA}
\mathscr{A}^{k, \circ}&= \bigcup_{\mu\in \mathbb{R}_{+}}\mathcal{A}^{k, \circ}_\mu,
\end{align}
where $\mathcal{A}^{k, \circ}_{\mu}$ is the set of all possible limits $\lim_{i\to \infty} M^{\ell_i,a_i}_{\mu, t}$ for which the limit exists and is compact, with $\ell_i\to \infty$ and $a_i\ge 0$ satisfying $a^1_i+\dots +a^k_i=\mu$ for a fixed constant $\mu>0$. On the other hand, from the construction above, we know that $\mathcal{A}^{k, \circ}_{\mu}$ is obtained by parabolically rescaling $\mathcal{A}^{k, \circ}=\mathcal{A}^{k, \circ}_{1}$.
Let us denote $\mathcal{D}_{\gamma}$ by $e^{\frac{\gamma}{2}}$-parabolic rescaling. By the construction above, we can also represent
\begin{align}
\mathscr{A}^{k, \circ}&=\bigcup_{\mu\in \mathbb{R}_{+}}\mathcal{A}^{k, \circ}_\mu= \bigcup_{\gamma\in \mathbb{R}}\mathcal{D}_{\gamma}(\mathcal{A}^{k, \circ}_1).
\end{align}
Note that the transformation $\mu=e^{\frac{\gamma}{2}}$ gives a homeomorphism between the two spaces $\mathcal{A}^{k,\circ}_{\mu}$ and $\mathcal{A}^{k,\circ}_{1}$, where the topology on these spaces is the  local smooth convergence. It is also convenient to define $\mathcal{A}^{k}_{\mu}$ and $\mathscr{A}^k$ as closure of $\mathcal{A}^{k, \circ}_{\mu}$ and $\mathscr{A}^{k, \circ}$ respectively in local smooth convergence topology. By \cite{DH_ovals, Brendle_inscribed, HK17}, the elements in $\mathcal{A}^{k}_{\mu}$ and $\mathscr{A}^k$ are ancient noncollapsed and the spaces $\mathcal{A}^{k}_{\mu}$ and $\mathscr{A}^k$ are sequentially compact. 
For ease of notation, we also denote these spaces by $\mathcal{A}^{\circ}_{\mu}, \mathscr{A}^{\circ}, \mathcal{A}_{\mu}, \mathscr{A}$ when dimensions $k, n$ are clear.

\section{Jacobian estimates and transformation map}
Given any $k$-oval $\mathcal{M}=\{M_t\}$ in $\mathbb{R}^{n+1}$ (with coordinates chosen as usual) and  parameters $\alpha\in \mathbb{R}^{k}, \beta\in \mathbb{R}, \gamma\in \mathbb{R}$ and $R\in \mathrm{SO}(k)$, we set
\begin{equation}
\mathcal{M}^{\alpha, \beta, \gamma, R}:=\{e^{\gamma/2} R(M_{e^{-\gamma}(t-\beta)}-\alpha)\},
\end{equation}
We first have the following proposition.
\begin{proposition}[orthogonality {c.f. \cite[Proposition 5.1]{CDZ_ovals}}]\label{prop_orthogonality}
For any $k$-oval $\mathcal{M}$ in $\mathbb{R}^{n+1}$ and any $\kappa>0$, there exists  constants $\tau_{\ast}={\tau}_\ast(\mathcal{M},\kappa)>-\infty$  with the following significance. For every $\tau_0\leq {\tau}_{\ast}$ and every $\eta \in [-\frac{1}{1000}, \frac{1}{1000}]$,  there exist  $\alpha\in \mathbb{R}^{k}, \beta\in \mathbb{R}, \gamma\in \mathbb{R}$ and $R\in \mathrm{SO}(k)$ depending on $\mathcal{M}, \kappa, \tau_0$  such that the truncated renormalized profile function $v^{\alpha, \beta, \gamma, R}_{\cC}$ of the transformed flow $\mathcal{M}^{\alpha, \beta, \gamma, R}$  is $\kappa$-quadratic at time $\tau_0$ and satisfies
\begin{equation}\label{cross}
\qquad \Big\langle v^{\alpha, \beta, \gamma, R}_{\cC}(\tau_0), y_{i}y_{j}\Big\rangle_\cH=0\quad i\not=j
\end{equation}
\begin{equation}\label{quadractic_mode}
 \Big\langle v^{\alpha, \beta, \gamma, R}_{\cC}(\tau_0) +\frac{\sqrt{2(n-k)}(|\mathbf{y}|^2-2k)}{4|\tau_0|},|\mathbf{y}|^2-2k\Big\rangle_\cH=\frac{\eta\kappa}{|\tau_0|},
\end{equation}
\begin{equation}\label{positive_mode}
\fp_+ \big(v^{\alpha, \beta, \gamma, R}_\cC(\tau_0)-\sqrt{2(n-k)}\big)=0.
\end{equation}
\end{proposition}

Then we discuss the uniqueness of such zeros as in \cite[Proposition 5.1]{CDDHS}. Recall that thanks to the $\mathbb{Z}_2^k$-symmetry the profile function of any $\mathcal{M}\in \mathcal{A}^{\circ}$ or of any $k$-oval $\mathcal{M}$ is orthogonal to the eigenfunctions $y_i, y_jy_\ell\quad 1\leq j<\ell\leq k$. Our first goal concerns orthogonality relations with respect to the eigenfunctions
\begin{equation}
\psi_{k+1}=1,\qquad \psi_{k+2}=|{\bf y}|^2-2k,
\end{equation}
where  they satisfy $\langle \psi_{k+2},\psi_{k+2}^2\rangle_\cH = 8 \|\psi_{k+2}\|^2=64\|\psi_{k+1}\|^2$. To this end, given $\mathcal{M}=\{M_t\}$, and parameters ${\tilde{\beta}}$ and ${{\gamma}}$, we consider the transformed flow
\begin{equation}
\mathcal{M}^{{\tilde{\beta}},{{\gamma}}}=\{ e^{{{\gamma}}/2} M_{e^{-{{\gamma}}}(t-{\tilde{\beta}})}\}.
\end{equation}
For convenience we set
\begin{equation}\label{bgamma}
   \tilde{b}=\sqrt{1+{\tilde{\beta}} e^{\tau}}-1, \qquad {\tilde{\Gamma}}=\frac{{\gamma}-\ln (1+{\tilde{\beta}} e^{\tau})}{\tau}.
\end{equation}
Then, the renormalized profile functions of $\mathcal{M}^{{\tilde{\beta}},{{\gamma}}}$ and $\mathcal{M}$ are related by
\begin{equation}\label{bgammatransform}
    v^{\tilde{b} {\tilde{\Gamma}}}({\bf y},\tau)=({1+\tilde{b}})\, v\left(\frac{{\bf y}}{{1+\tilde{b}}}, (1+{\tilde{\Gamma}})\tau\right).
\end{equation}
Our first goal is to find a canonical zero of the map 
\begin{equation}\label{Psi=00}
\Psi_\tau(\tilde{b}, {\tilde{\Gamma}})=\left(  \Big\langle  \psi_{k+1} ,v^{{\tilde{b}\tilde{\Gamma}}}_\cC-\sqrt{2(n-k)} \Big\rangle_\cH , 
\Big\langle  \psi_{k+2}, v^{{\tilde{b}\tilde{\Gamma}}}_\cC +\frac{\sqrt{2(n-k)}\psi_{k+2}}{4|\tau|}\Big\rangle_\cH \right).
\end{equation}
The key towards finding such a canonical zero is the Jacobian estimate, which we restate here for convenience of the reader:

\begin{proposition}[Jacobian estimate]\label{JPhiestimates} There exist $\kappa>0$ and $\tau_\ast>-\infty$ with the following significance. If $\mathcal{M}$ is $\kappa$-quadratic at time $\tau_0\leq \tau_\ast$, then 
\begin{equation}\label{1b}
\left\langle \psi_{k+1},\partial_{\tilde{b}} v_{\cC}^{{\tilde{b}\tilde{\Gamma}}}\right\rangle_\cH = \sqrt{2(n-k)}\|\psi_{k+1}\|^2+\frac{\sqrt{2(n-k)}}{8|\tau|}\|\psi_{k+2}\|^2+O(\frac{\kappa}{|\tau|}),
\end{equation}
\begin{equation}\label{2b}
    \left\langle \psi_{k+2},\partial_{\tilde{b}} v_{\cC}^{{\tilde{b}\tilde{\Gamma}}}\right\rangle_\cH = \frac{\sqrt{2(n-k)}\|\psi_{k+2}\|^2}{4|\tau|}+O(\frac{\kappa}{|\tau|}),
\end{equation}
\begin{equation}\label{1g}
     \left \langle \psi_{k+1},  \partial_{{\tilde{\Gamma}}} v_{\cC}^{\tilde{b} {\tilde{\Gamma}}} \right\rangle_\cH =O(\kappa)
 \end{equation}
 \begin{equation}\label{2g}
      \left \langle \psi_{k+2},  \partial_{{\tilde{\Gamma}}} v_{\cC}^{\tilde{b} {\tilde{\Gamma}}} \right\rangle_\cH =\frac{\|\psi_{k+2}\|^2_{\mathcal{H}}}{{2\sqrt{2(n-k)}}|\tau|}+O\left(\frac{\kappa}{|\tau|}\right).
 \end{equation}
 and in particular
\begin{equation}
    \mathrm{det}(J\Psi_\tau(\tilde{b}, {\tilde{\Gamma}}))=\frac{\|\psi_{k+1}\|^2\|\psi_{k+2}\|^2}{2|\tau|}+O\left(\frac{\kappa}{|\tau|}\right)>0
\end{equation}
holds for all $\tau\leq\tau_0$ and all $(\tilde b,{\tilde{\Gamma}})$ with $|\tau|^2{\tilde{b}}^2+{\tilde{\Gamma}}^2\leq 100 \kappa^2$, where we write $f=O(g)$ if there exists some constant $C<\infty$ independent of $\tau\leq\tau_0$ and $(b,{\tilde{\Gamma}})$, with $|\tau|^2{\tilde{b}}^2+{\tilde{\Gamma}}^2\leq 100 \kappa^2$, such that $|f|\leq Cg$.
\end{proposition}

\begin{proof} 
Throughout this proof, we write $f=O(g)$ if there exists some constant $C<\infty$ independent of $\tau\leq\tau_0$ and $(\tilde b,{\tilde{\Gamma}})$, with $|\tau|^2{\tilde{b}}^2+{\tilde{\Gamma}}^2\leq 100 \kappa^2$, such that $|f|\leq Cg$.\footnote{The constraint $|\tau|^2{\tilde{b}}^2+{\tilde{\Gamma}}^2\leq 100 \kappa^2$ is motivated by the degree argument below.} 
Let us also recall that our cutoff function $\chi_\cC$ satisfies
\be
 \chi_\mathcal{C}(v)=0 \,\,\, \textrm{for}\, \,\, v\leq \tfrac58 \theta  \qquad \mbox{and} \qquad    \chi_\mathcal{C}(v)=1 \,\,\, \textrm{for}\, \, \, v\geq \tfrac78 \theta,
 \ee
and that thanks to \cite[Proposition 2.2]{CDZ_ovals} (uniform sharp asymptotics) and convexity for $\theta$ small enough we have the Gaussian tail estimate
\begin{equation}\label{gaussian_tail}
e^{-\frac{|{\bf{y}}|^2}{4}} 1_{\{ v^{{\tilde{b}\tilde{\Gamma}}}\leq \theta\}}\leq e^{-\frac{|\tau|}{3}},
\end{equation}
and the gradient estimate
\begin{equation}\label{grad_est_Sec5}
| D v^{{\tilde{b}\tilde{\Gamma}}}|  1_{\{v^{{\tilde{b}\tilde{\Gamma}}}\geq \theta/2\}}=O(|\tau|^{-1/2}).
\end{equation}
Now, to prove the proposition, we have to estimate the derivatives with respect to the parameters $b$ and ${\tilde{\Gamma}}$. We start with the former:

\begin{claim}[$\tilde b$-derivatives]\label{claim_b_der}
For the derivatives with respect to $\tilde b$ we have
\begin{equation}\label{b-k+1}
\left\langle \psi_{k+1},\partial_{\tilde{b}} v_{\cC}^{{\tilde{b}\tilde{\Gamma}}}\right\rangle_\cH = \sqrt{2(n-k)}\|\psi_{k+1}\|^2+\frac{\sqrt{2(n-k)}}{8|\tau|}\|\psi_{k+2}\|^2+O(\frac{\kappa}{|\tau|}),
\end{equation}
\begin{equation}\label{b-k+2}
    \left\langle \psi_{k+2},\partial_{\tilde{b}} v_{\cC}^{{\tilde{b}\tilde{\Gamma}}}\right\rangle_\cH = \frac{\sqrt{2(n-k)}\|\psi_{k+2}\|^2}{4|\tau|}+O(\frac{\kappa}{|\tau|}),
\end{equation}
\end{claim}

\begin{proof}[{Proof of Claim \ref{claim_b_der}}]
First, by the transformation formula  \eqref{bgammatransform} we have
\begin{equation}\label{partialb}
   \partial_{\tilde{b}} v^{\tilde{b} {\tilde{\Gamma}}}({\bf y},\tau)= v\left(\frac{{\bf y}}{{1+\tilde{b}}}, (1+{\tilde{\Gamma}})\tau\right)-\frac{{\bf y}}{{1+\tilde{b}}}\cdot D v\left(\frac{{\bf y}}{{1+\tilde{b}}}, (1+{\tilde{\Gamma}})\tau\right).
\end{equation}
Via integration by parts and noticing the Gaussian weighed in the integral, this implies for $i\in \{k+1, k+2\}$ we have
\begin{multline}
   \left\langle \psi_i,\chi_{\cC}(v^{{\tilde{b}\tilde{\Gamma}}})\partial_{\tilde{b}} v^{{\tilde{b}\tilde{\Gamma}}}\right\rangle_\cH=\left\langle 
  \left(\psi_i+ \frac{2{\bf y}\cdot D \psi_i+(2k-|{\bf{y}}|^2)\psi_i}{2({1+\tilde{b}})}\right)\chi_{\cC}(v^{{\tilde{b}\tilde{\Gamma}}}),v\right\rangle_\cH\\
  +\left\langle \psi_i\frac{{\bf y}\cdot D v^{{\tilde{b}\tilde{\Gamma}}}}{{1+\tilde{b}}}\chi_{\cC}'(v^{{\tilde{b}\tilde{\Gamma}}}), v\right\rangle_\cH,
\end{multline}
where the function $v$ is evaluated at $\big(\frac{{\bf y}}{{1+\tilde{b}}}, (1+{\tilde{\Gamma}})\tau\big)$.

Using the basic facts recalled above we see that
\begin{equation}
\left\langle 
  \left(\psi_i+ \frac{2{\bf y}\cdot D \psi_i+(2k-|{\bf{y}}|^2)\psi_i}{2({1+\tilde{b}})}\right)\left(1-\chi_{\cC}(v^{{\tilde{b}\tilde{\Gamma}}})\right), v\right\rangle_\cH= O(|\tau|^{-10}),
\end{equation}
and
\begin{equation}
\left\langle \psi_i \frac{{\bf y}\cdot D v^{{\tilde{b}\tilde{\Gamma}}}}{{1+\tilde{b}}}\chi_{\cC}'(v^{{\tilde{b}\tilde{\Gamma}}}), v\right\rangle_\cH=O(|\tau|^{-10}).
\end{equation} 

Next, by the $\kappa$-quadraticity assumption we have the $\mathcal{H}$-norm expansion
\begin{align}\label{ubg_expansion}
    v\left(\frac{{\bf y}}{{1+\tilde{b}}}, (1+{\tilde{\Gamma}})\tau\right)
   =\sqrt{2(n-k)}-\frac{\sqrt{2(n-k)}(|{\bf{y}}|^2-2k)}{4(1+{\tilde{\Gamma}})|\tau|}+O\left(\frac{\kappa}{|\tau|}\right).
\end{align}
This yields
\begin{align}
  &\quad \left\langle\left(\psi_{k+1}+ \frac{2{\bf y}\cdot D \psi_{k+1}+(2k-|{\bf{y}}|^2)\psi_{k+1}}{2({1+\tilde{b}})}\right),v\right\rangle_\cH\\
  &= \sqrt{2(n-k)}\|\psi_{k+1}\|^2+\frac{\sqrt{2(n-k)}}{8|\tau|}\|\psi_{k+2}\|^2+O(\frac{\kappa}{|\tau|}),
\end{align}
and, taking into account the identities $\langle  \psi_{k+2}^2 -8k,1\rangle_\cH = 0$ and $\langle \psi_{k+2},\psi_{k+2}^2\rangle_\cH = 8 \|\psi_{k+2}\|^2_{\mathcal{H}}$ also yields
\begin{align}
&\quad\left\langle 
  \left(\psi_{k+2}+ \frac{2{\bf y}\cdot D \psi_{k+2}+(2k-|{\bf{y}}|^2)\psi_{k+2}}{2({1+\tilde{b}})}\right),v\right\rangle_\cH\\
&=\left\langle 
  \left(\psi_{k+2}+2(\psi_{k+2}+2k)-\frac{\psi_{k+2}^2}{2}\right),v\right\rangle_\cH+O(\frac{\kappa}{|\tau|})\\
&=-\frac{\sqrt{2(n-k)}\|\psi_{k+2}\|^2_{\mathcal{H}}}{4|\tau|}+O(\frac{\kappa}{|\tau|})
\end{align}
Finally, observe that
\begin{equation}
\left\langle \psi_i,\partial_{\tilde{b}} v_{\cC}^{{\tilde{b}\tilde{\Gamma}}}\right\rangle_\cH=\left\langle \psi_i,\chi_{\cC}(v^{{\tilde{b}\tilde{\Gamma}}})\partial_{\tilde{b}} v^{{\tilde{b}\tilde{\Gamma}}}\right\rangle_\cH+\left\langle \psi_i,v^{{\tilde{b}\tilde{\Gamma}}}\chi_{\cC}'(v^{{\tilde{b}\tilde{\Gamma}}})\partial_{\tilde{b}} v^{{\tilde{b}\tilde{\Gamma}}}\right\rangle_\cH,
\end{equation}
and
\begin{equation}
\left\langle \psi_i,v^{{\tilde{b}\tilde{\Gamma}}}\chi_{\cC}'(v^{{\tilde{b}\tilde{\Gamma}}})\partial_{\tilde{b}} v^{{\tilde{b}\tilde{\Gamma}}}\right\rangle_\cH=O(|\tau|^{-10}).
\end{equation} 
Combining the above estimates the Claim \ref{claim_b_der} ($\tilde b$-derivatives) follows.
\end{proof}

\begin{claim}[${\tilde{\Gamma}}$-derivatives]\label{claim_gamma_der} For the derivatives with respect to ${\tilde{\Gamma}}$ we have
\begin{equation}\label{gamma-k+1}
     \left \langle \psi_{k+1},  \partial_{{\tilde{\Gamma}}} v_{\cC}^{\tilde{b} {\tilde{\Gamma}}} \right\rangle_\cH =O(\kappa)
 \end{equation}
 \begin{equation}\label{gamma-k+2}
      \left \langle \psi_{k+2},  \partial_{{\tilde{\Gamma}}} v_{\cC}^{\tilde{b} {\tilde{\Gamma}}} \right\rangle_\cH =\frac{\|\psi_{k+2}\|^2_{\mathcal{H}}}{{2\sqrt{2(n-k)}}|\tau|}+O\left(\frac{\kappa}{|\tau|}\right).
 \end{equation}
\end{claim}

\begin{proof}[{Proof of Claim \ref{claim_gamma_der}}]
By the transformation formula  \eqref{bgammatransform} we have
\begin{equation}\label{partialtau}
    \partial_{\tilde{\Gamma}} v^{\tilde{b} {\tilde{\Gamma}}}({\bf y},\tau)=({1+\tilde{b}})\tau v_\tau \left(\frac{{\bf y}}{{1+\tilde{b}}}, (1+{\tilde{\Gamma}})\tau\right).
\end{equation}
Recall that the renormalized profile function evolves by
\begin{equation}
v_\tau=\mathcal{L}v - \frac{ D^2v\!:\! (D v\otimes D v)}{1+|D v|^2} -\frac{v}{2}-\frac{n-k}{v}.
\end{equation}
Together with integration by parts, this implies for $i\in \{k+1, k+2\}$ we have
\begin{multline}\label{dividing 1btau}
\frac{\left\langle \psi_i,  \chi_{\cC}(v^{\tilde{b} {\tilde{\Gamma}}}) \partial_{{\tilde{\Gamma}}} v^{\tilde{b} {\tilde{\Gamma}}} \right\rangle_\cH}{({1+\tilde{b}})\tau}
     =  \left\langle \mathcal{L}\left(  \chi_{\cC}(v^{\tilde{b} {\tilde{\Gamma}}}) \psi_i\right), v\right\rangle_\cH\\
     -\left\langle \psi_i ,\chi_{\cC}(v^{\tilde{b} {\tilde{\Gamma}}})\left(\frac{v}{2}+\frac{n-k}{v}\right)\right\rangle_\cH 
      - \left\langle \psi_i,\chi_{\cC}(v^{\tilde{b} {\tilde{\Gamma}}}) \frac{  D^2v\!:\! (D v\otimes D v)}{1+| D v|^2}  \right\rangle_\cH ,
\end{multline}
where the function $v$ is evaluated at $\big(\frac{{\bf y}}{{1+\tilde{b}}}, (1+{\tilde{\Gamma}})\tau\big)$ as usual.

We  can rewrite the first term using the product rule in the form
\begin{equation}
\mathcal{L}\left(  \chi_{\cC}(v^{\tilde{b} {\tilde{\Gamma}}}) \psi_i\right)=\chi_{\cC}(v^{\tilde{b} {\tilde{\Gamma}}}) \mathcal{L}   \psi_i+\psi_i(\mathcal{L}-1)\chi_{\cC}(v^{\tilde{b} {\tilde{\Gamma}}})+2 D \chi_{\cC}(v^{\tilde{b} {\tilde{\Gamma}}}) D \psi_i.
\end{equation}
Arguing as above, and using also that $\mathcal{L}\psi_i=\delta_{1i}$, this yields
\begin{equation}
 \left\langle \mathcal{L}\left( \chi_{\cC}(v^{\tilde{b} {\tilde{\Gamma}}}) \psi_i\right) , v\right\rangle_\cH= \left\langle  \chi_{\cC}(v^{\tilde{b} {\tilde{\Gamma}}})\delta_{1i} , v\right\rangle_\cH+O(|\tau|^{-10}).
\end{equation}

\bigskip
Next, we have to deal with the term
\begin{equation}
\left\langle \chi_{\cC}(v^{\tilde{b} {\tilde{\Gamma}}}), v\right\rangle_\cH-\left\langle \psi_{k+1} ,\chi_{\cC}(v^{\tilde{b} {\tilde{\Gamma}}})\left(\frac{v}{2}+\frac{n-k}{v}\right)\right\rangle_\cH=\left\langle \chi_{\cC}(v^{\tilde{b} {\tilde{\Gamma}}}),\frac{v}{2}-\frac{n-k}{v} \right\rangle_\cH.
\end{equation}
To estimate this, we start with the algebraic identity 
\begin{equation}
\frac{v}{2}-\frac{n-k}{v}=v-\sqrt{2(n-k)}-\frac{(v-\sqrt{2(n-k)})^2}{2v}.
\end{equation}
Using the tail bound \eqref{gaussian_tail} we see that
\begin{equation}
\left\langle \chi_{\cC}(v^{\tilde{b} {\tilde{\Gamma}}})-1,v-\sqrt{2(n-k)}\right\rangle_\cH=O(|\tau|^{-10}),
\end{equation}
and using $\langle 1,\psi_{k+2}\rangle_\cH = 0$ and \eqref{ubg_expansion} we can estimate
\begin{equation}
\left\langle 1,v-\sqrt{2(n-k)}\right\rangle_\cH=\left\langle 1,v-\sqrt{2(n-k)}+\frac{\sqrt{2(n-k)}\psi_{k+2}}{4|\tau|}\right\rangle_\cH\!\!=\!O\left(\frac{\kappa}{|\tau|}\right).
\end{equation}
Moreover, remembering that $\chi_\cC(v)=0$ for $v\leq \tfrac58\theta$ we infer that
\begin{equation}
\left|\left\langle \chi_{\cC}(v^{\tilde{b} {\tilde{\Gamma}}}),\frac{(v-\sqrt{2(n-k)})^2}{2v} \right\rangle_\cH\right|\leq \frac{2}{\theta}\left\| v-\sqrt{2(n-k)}\right\|^2_{\mathcal{H}}=O(|\tau|^{-2}).
\end{equation}
Combining the above observations shows that
\begin{equation}
\left\langle \chi_{\cC}(v^{\tilde{b} {\tilde{\Gamma}}}) , v\right\rangle_\cH-\left\langle \psi_{k+1} ,\chi_{\cC}(v^{\tilde{b} {\tilde{\Gamma}}})\left(\frac{v}{2}+\frac{n-k}{v}\right)\right\rangle_\cH=O\left(\frac{\kappa}{|\tau|}\right).
\end{equation}

\medskip

Next, we have to deal with the term
\begin{equation}
-\Big\langle \psi_{k+2},\chi_{\cC}(v^{\tilde{b} {\tilde{\Gamma}}})\left(\frac{v}{2}+\frac{n-k}{v}\right)\Big\rangle_\cH.
\end{equation}
To estimate this, we start with the algebraic identity 
\begin{equation}
\frac{v}{2}+\frac{n-k}{v}=\sqrt{2(n-k)}+\frac{u^2}{2\sqrt{2(n-k)}}-\frac{u^3}{2\sqrt{2(n-k)}v},
\end{equation}
where  $u=v-\sqrt{2(n-k)}$.
Using in particular the tail bound \eqref{gaussian_tail} we see that
\begin{equation}
\Big\langle \psi_{k+2} \chi_{\cC}(v^{\tilde{b} {\tilde{\Gamma}}}),\sqrt{2(n-k)}\Big\rangle_\cH=O(|\tau|^{-10}),
\end{equation}
and
\begin{equation}
\Big\langle \psi_{k+2} \left(\chi_{\cC}(v^{\tilde{b} {\tilde{\Gamma}}})-1\right),u^2\Big\rangle_\cH=O(|\tau|^{-10}).
\end{equation}
Also observe that 
\begin{equation}\label{blabla1}
\left\langle  1_{\{  y\leq 4\textrm{ or } y\geq |\tau|^{1/100} \}} ,\psi_{k+2}\left( u^2-\frac{(n-k)\psi_{k+2}^2}{8|\tau|^2}\right)\right\rangle_\cH=O\left(\frac{\kappa}{|\tau|^2}\right).
\end{equation}
On the other hand, abbreviating: $\chi= 1_{\{ 4\leq y\leq |\tau|^{1/100} \}}$, we can estimate
\begin{multline}\label{blabla2}
\left|\left\langle \chi,\psi_{k+2}\left( u^2-\frac{(n-k)\psi_{k+2}^2}{8|\tau|^2}\right)\right\rangle_\cH\right|\\
\leq \left\| \chi\psi_{k+2}^{1/2}\left(u-\frac{\sqrt{2(n-k)}\psi_{k+2}}{4|\tau|}\right)\right\|_\cH \cdot  \left\| \psi_{k+2}^{1/2}\left(u+\frac{\sqrt{2(n-k)}\psi_{k+2}}{4|\tau|}\right)\right\|_\cH .
\end{multline}
To proceed, we set
\begin{equation}\label{w_def}
w=\hat{u}+\frac{\sqrt{2(n-k)}\psi_{k+2}}{4|\tau|},
\end{equation}
where $\hat{u}$ is truncated via the graphical radius $\rho(\tau)=|\tau|^{1/100}$.
Note that
\begin{equation}
\|  w \|_\cH =O\left(\frac{\kappa}{|\tau|}\right) \quad \textrm{and}\quad
\| (\partial_\tau-\mathcal{L}) w \|_\cH =O(|\tau|^{-101/100}).
\end{equation}
Hence, arguing similarly as in the proof of \cite[Proposition 2.8]{CDDHS} or \cite[Lemma 3.2]{DZ_spectral_quantization}
 we infer that
\begin{equation}\label{der_bound_eq}
\|Dw \|_\cH =O\left(\frac{\kappa}{|\tau|}\right).
\end{equation}
Using also the weighted Poincar\'e inequality as in \cite[(2.4)]{CDZ_ovals}
 for function $f$
\begin{equation}\label{easy_Poincare}
\big\|  (1+|{\bf{y}}|) f \big\|_{\mathcal{H}} \leq C \big( \| f \|_{\mathcal{H}}+\| Df \|_{\mathcal{H}}\big).
\end{equation}
this implies
\begin{equation}
\left\| \chi \psi_{k+2}^{1/2}\left(u+\frac{\sqrt{2(n-k)}\psi_{k+2}}{4|\tau|}\right)\right\|_\cH=O\left(\frac{\kappa}{|\tau|}\right),
\end{equation}
and  in particular via the triangle inequality this also implies
\begin{equation}
\left\|\chi \psi_{k+2}^{1/2}\left(u-\frac{\sqrt{2(n-k)}\psi_{k+2}}{4|\tau|}\right)\right\|_\cH=O(|\tau|^{-1}).
\end{equation}
Together with \eqref{blabla1} and \eqref{blabla2}, and with $\langle \psi_{k+2},\psi_{k+2}^2\rangle_\cH = 8 \|\psi_{k+2}\|^2_{\mathcal{H}}$, this yields
\begin{equation}
\Big\langle \psi_{k+2} , u^2\Big\rangle_\cH=\frac{\|\psi_{k+2}\|^2_{\mathcal{H}}}{|\tau|^2}+O\left(\frac{\kappa}{|\tau|^2}\right).
\end{equation}
Moreover, using the graphical radius condition from \eqref{condition3} we see that
\begin{equation}
\left| \Big\langle \psi_{k+2}1_{\{y\leq |\tau|^{1/100}\}},\frac{u^3}{v} \Big\rangle_\cH \right|\leq C \langle |\psi_{k+2}|,u^2\rangle_\cH \max_{y\leq |\tau|^{1/100}} |u|=O\left(|\tau|^{-\frac{101}{50}}\right).
\end{equation}
Note also that we have the tail estimate
\begin{equation}
\Big\langle \psi_{k+2}\chi_{\cC}(v^{\tilde{b} {\tilde{\Gamma}}})1_{\{y\geq |\tau|^{1/100}\}},\frac{u^3}{v} \Big\rangle_\cH =O(|\tau|^{-10}).
\end{equation}
Combining the above observations shows that
\be
-\left\langle \psi_{k+2} ,\chi_{\cC}(v^{\tilde{b} {\tilde{\Gamma}}})\left(\frac{v}{2}+\frac{n-k}{v}\right)\right\rangle_\cH
=-\frac{\|\psi_{k+2}\|^2_{\mathcal{H}}}{{2\sqrt{2(n-k)}}|\tau|^2}+O\left(\frac{\kappa}{|\tau|^2}\right).
\ee

\medskip

Moreover, by the derivative estimates from \cite[Lemma 4.16]{CDDHS} for estimating $|D^2v|$ and using \eqref{der_bound_eq} and definition of $w$ in \eqref{w_def} and Gaussian tail estimates for estimates $\|1_{|y|\leq |\tau|^{1/100}}Dv\|^{-1}_{\mathcal{H}}$,  we have the estimate 
\begin{equation}
 \left\langle 1_{\{y\leq |\tau|^{1/100}\}}, |\psi_i| |D^2v| |Dv|^2 \right\rangle_\cH =O\left(|\tau|^{\frac{1}{50}-\frac{5}{2}}\right).
\end{equation}
and
\begin{equation}
 \left\langle 1_{\{y\geq |\tau|^{1/100}\}},\psi_i \chi_{\cC}(v^{\tilde{b} {\tilde{\Gamma}}}) \frac{  D^2v\!:\! (D v\otimes D v)}{1+| D v|^2}  \right\rangle_\cH  = O(|\tau|^{-10}).
\end{equation}
This shows that
\be
\left\langle \psi_i,\chi_{\cC}(v^{\tilde{b} {\tilde{\Gamma}}}) \frac{  D^2v\!:\! (D v\otimes D v)}{1+| D v|^2}  \right\rangle_\cH=O\left(|\tau|^{-\frac{49}{20}}\right).
\ee

Furthermore, observe that
\begin{equation}
\left\langle \psi_i,\partial_{\tilde{\Gamma}} v_{\cC}^{{\tilde{b}\tilde{\Gamma}}}\right\rangle_\cH=\left\langle \psi_i,\chi_{\cC}(v^{{\tilde{b}\tilde{\Gamma}}})\partial_{\tilde{\Gamma}} v^{{\tilde{b}\tilde{\Gamma}}}\right\rangle_\cH+\left\langle \psi_i,v^{{\tilde{b}\tilde{\Gamma}}}\chi_{\cC}'(v^{{\tilde{b}\tilde{\Gamma}}})\partial_{\tilde{\Gamma}} v^{{\tilde{b}\tilde{\Gamma}}}\right\rangle_\cH,
\end{equation}
and
\begin{equation}
\left\langle \psi_i,v^{{\tilde{b}\tilde{\Gamma}}}\chi_{\cC}'(v^{{\tilde{b}\tilde{\Gamma}}})\partial_{\tilde{\Gamma}} v^{{\tilde{b}\tilde{\Gamma}}}\right\rangle_\cH=O(|\tau|^{-10}).
\end{equation} 
Recalling \eqref{dividing 1btau}, combining the above estimates and multiplying $({1+\tilde{b}})\tau$ with $|\tilde{b}\tau|\leq 10\kappa$, the Claim \ref{claim_gamma_der} (${\tilde{\Gamma}}$-derivatives) follows.
\end{proof}
 
Finally, combining Claim \ref{claim_b_der} ($\tilde b$-derivatives) and Claim \ref{claim_gamma_der} (${\tilde{\Gamma}}$-derivatives) we conclude that
\begin{equation}
    \mathrm{det}(J\Psi_\tau)=\frac{\|\psi_{k+1}\|^2_{\mathcal{H}}\|\psi_{k+2}\|^2_{\mathcal{H}}}{2|\tau|}+O\left(\frac{\kappa}{|\tau|}\right).
\end{equation}
This proves the proposition.
\end{proof}

Given $\kappa>0$ and $\tau_0>-\infty$, we let $
 \mathcal{A}'_{\kappa}(\tau_0)$ be the set of ancient ovals whose element $\mathcal{M}$ is $\kappa$-quadratic at time $\tau_0$,  and $\exists \mathcal{N}\in\mathcal{A}^{k, \circ}$, $\exists\beta,\gamma$: $\mathcal{M}=\mathcal{N}^{\beta,\gamma}$ and satisfies 
\begin{equation}\label{quadractic_mode1}
 \Big\langle v^{\mathcal{M}}_{\cC}(\tau_0) +\frac{\sqrt{2(n-k)}(|\mathbf{y}|^2-2k)}{4|\tau_0|},|\mathbf{y}|^2-2k\Big\rangle_\cH=0,
\end{equation}
\begin{equation}\label{positive_mode1}
\fp_+ \big(v^{\mathcal{M}}_\cC(\tau_0)-\sqrt{2(n-k)}\big)=0.
\end{equation}
\begin{corollary}[local continuous orthogonality]\label{prop_continuous_orthogonality}
For any $k$-oval $\mathcal{M}_0$ in $\mathbb{R}^{n+1}$ and any $\kappa>0$, there exists  constants $\tau_{\ast}={\tau}_\ast(\mathcal{M}_0,\kappa)>-\infty$,  with the following significance. For every $\tau_0\leq {\tau}_{\ast}$, there is an open neighborhood $U(\mathcal{M}_0, \kappa, \tau_0)$ of  $\mathcal{M}_0$ depending on $\mathcal{M}_0, \kappa, \tau_0$ such that for every $\mathcal{M}\in U(\mathcal{M}_0, \kappa, \tau_0)$,  there exist  $\beta\in \mathbb{R}, \gamma\in \mathbb{R}$  depending on $\kappa, \tau_0$ and continuously depending $\mathcal{M}\in U(\mathcal{M}_0, \kappa, \tau_0)$, such that the truncated renormalized profile function $v^{\beta, \gamma}_{\cC}$ of the time-shifted and parabolically rescaled flow $\mathcal{M}^{\beta, \gamma}$  is $\kappa$-quadratic at time $\tau_0$ and satisfies
\begin{equation}\label{quadractic_mode2}
 \Big\langle v^{\beta, \gamma}_{\cC}(\tau_0) +\frac{\sqrt{2(n-k)}(|\mathbf{y}|^2-2k)}{4|\tau_0|},|\mathbf{y}|^2-2k\Big\rangle_\cH=0,
\end{equation}
\begin{equation}\label{positive_mode2}
\fp_+ \big(v^{\beta, \gamma}_\cC(\tau_0)-\sqrt{2(n-k)}\big)=0.
\end{equation}
Moreover, the transformation $\mathcal{M}^{\beta, \gamma}$ gives a homeomorphism from $U(\mathcal{M}_0, \kappa, \tau_0)$ to its image  $U'(\mathcal{M}_0, \kappa, \tau_0)$.
\end{corollary}
\begin{proof}
    By \cite[Theorem 5.1]{DZ_spectral_quantization}, for any $\kappa'=\frac{\kappa}{C}>0$ where the large positive constant $C<+\infty$ will be determined later, there exists $\tau_{\ast}>-\infty$, such that $\mathcal{M}_0$ is strongly $\kappa'$-quadratic from time $\tau_\ast$. In particular, for all $\tau\leq\tau_\ast$ we have its profile function satisfies
\begin{equation}
    \left\| v_{0, \cC}(\mathbf{y},  \tau)-\sqrt{2(n-k)}+\frac{\sqrt{2(n-k)}(|\mathbf{y}|^2-2k)}{4|\tau|}  \right\|_{\mathcal{H}}\leq \frac{\kappa'}{|\tau|} .
\end{equation}
  By this and Gaussian tail estimates, for $0<\kappa'=\frac{\kappa}{C}\ll \kappa$ and $\tau_0\leq 2\tau_*$  negative enough we can set $U(\mathcal{M}_0, \kappa, \tau_0)$ to be a open neighborhood in $k$-oval space endowed with the local smooth convergence topology such that  the truncated profile function $v_{\cC}$ of every $\mathcal{M}\in U(\mathcal{M}_0, \kappa, \tau_0)$    satisfies
  \begin{equation}
    \left\| v_{\cC}(\mathbf{y},  \tau)-\sqrt{2(n-k)}+\frac{\sqrt{2(n-k)}(|\mathbf{y}|^2-2k)}{4|\tau|}  \right\|_{\mathcal{H}}< \frac{2\kappa'}{|\tau|} \quad \tau\in [2\tau_0, \tau_0/2].
\end{equation}

Thus, using the transformation formula  for the profile function $v$ of $\mathcal{M}\in U(\mathcal{M}_0, \kappa, \tau_0)$
\begin{equation}
  b= \sqrt{1+\beta e^{\tau}}-1,\quad \Gamma=\frac{\gamma-\ln (1+\beta e^{\tau})}{\tau}.
\end{equation}
\begin{align}
     v^{b\Gamma}(y,\tau)=(1+b)v\left(\frac{y}{1+b},(1+\Gamma)\tau\right)
\end{align}
and standard Gaussian tail estimates, for all $(b, \Gamma)\in [-1/|\tau|,1/|\tau|]\times [-1/2,1/2]$ and $\tau\in [2\tau_0, \tau_0/2]$ we get
\begin{equation}\label{transformed profile expansion}
 \left\| v^{b\Gamma}_{\cC}(y,\tau)-\sqrt{2(n-k)} - \sqrt{2(n-k)}b
+\frac{\sqrt{2(n-k)}(|y|^2-2k)}{4|\tau|(1+\Gamma)}\right\|_\cH \leq \frac{10\kappa'}{|\tau|}.
\end{equation}
Hence, for $\kappa'>0$ small enough and $\tau_0\leq\tau_\ast$ negative enough, the map $\Psi({b}, {{\Gamma}})$
\begin{equation}
\left(  \Big\langle \psi_{k+1} ,v^{{{b},{\Gamma}}}_\cC(\tau)-\sqrt{2(n-k)} \Big\rangle_\cH , 
\Big\langle  \psi_{k+2}, v^{{{b}, {\Gamma}}}_\cC(\tau) +\frac{\sqrt{2(n-k)}\psi_{k+2}}{4|\tau|}\Big\rangle_\cH \right)
\end{equation}
is homotopic to 
\begin{equation}
(b,\Gamma)\mapsto \left(
     \sqrt{2(n-k)}\|\psi_{k+1}\|^2\,  b,
     \frac{\sqrt{2(n-k)}\|\psi_{k+2}\|^2}{4|\tau|}\frac{\Gamma}{(1+\Gamma)}
    \right),
\end{equation}
when restricted to the boundary of the disc
\begin{equation}\label{Dk2}
D=D_{\kappa', \tau}:=\left\{(b, \Gamma)\in\mathbb{R}^2\;:\; |\tau|^2b^2+\Gamma^2\leq 100 (\kappa')^2  \right\},
\end{equation}
where the homotopy can be chosen through maps avoiding the origin. Hence,
\begin{equation}\label{degree=1}
    \mathrm{deg}(\Psi|_D)=1.
\end{equation}
On the other hand, by Proposition \ref{JPhiestimates} (Jacobian estimate), as long as we choose $\kappa>0$ small enough and $\tau_0\leq\tau_\ast$ negative enough, we have
\begin{equation}
\mathrm{det}(J\Psi|_D)>0.
\end{equation}
In particular, $(0, 0)$ is a regular value. Recalling also that degree is given by counting the inverse images according to the sign of their Jacobian, we thus infer that for any $\mathcal{M}\in U(\mathcal{M}_0, \kappa, \tau_*)$, there exists a 
 unique $(b, \Gamma)\in D$ such that $\Psi(b, \Gamma, \mathcal{M})=0$. By uniqueness, $(b,\Gamma)$ and thus the corresponding $(\beta,\gamma)$, depends continuously on $\mathcal{M}\in U(\mathcal{M}_0, \kappa, \tau_0)$. Because the transformation is invertible and the inverse is also continuous in term of local smooth convergence topology, therefore the transformation $\mathcal{M}^{\beta, \gamma}$ gives a homeomophism from $U(\mathcal{M}_0, \kappa, \tau_0)$
 to its image $U'(\mathcal{M}_0, \kappa, \tau_0)$. By \eqref{transformed profile expansion}, if we make an expansion and let $C<\infty$ large enough in the definition of $\kappa'=\frac{\kappa}{C}>0$, we can make $\mathcal{M}^{\beta, \gamma}$ be $\kappa$-quadratic at $\tau_0$. This finishes the proof of the corollary.
\end{proof}

The following proposition asserts we can homeomophically shift the ancient solutions in $\mathscr{A}$ such that they have vanishing positive mode.
\begin{proposition}[shift map]\label{time-shift}
Given any $\tau_0<0$, for any ancient oval $\mathcal{M}\in \mathscr{A}$, we can find a unique $\beta=\beta(\mathcal{M}, \tau_0)$ such that the profile function of time shifted oval $M_{t-\beta}$ satisfies 
\begin{equation}\label{shift vanish +}
\mathfrak{p}_{+}\big(v^{\beta}_{\cC}(\cdot , \tau_{0})-\sqrt{2(n-k)}\big)=0,
\end{equation}
Let $\mathcal{S}$ be time shift map on $\mathcal{A}$ which is a homeomorphism from 
\begin{equation}
    \mathscr{A} = \cup_{\mu}\mathcal{A}_\mu
\end{equation}
to its image
\begin{equation}\label{shifted oval class}
    \mathscr{A}'=\mathcal{S}(\mathscr{A})
\end{equation}
and it induces restriction $\mathcal{S}_\mu$ a time shift map from $\mathcal{A}_\mu$, we have $\mathcal{S}_\mu$ is a homeomorphism $\mathcal{S}$ from $\mathcal{A}_\mu$ to its image
\begin{equation}
\mathcal{A}'_\mu=\mathcal{S}_\mu(\mathcal{A}_\mu)
\end{equation}
Moreover, for any $\kappa>0$, there is a $\kappa'>0$, if $\mathcal{M}\in \mathscr{A}$ is strongly $\kappa'$-quadratic from time $\tau_0+1$, then $\mathcal{S}(\mathcal{M}) = \mathcal M^\beta$ is $\kappa$-quadratic at time $\tau_0$.
\end{proposition}
\begin{proof}
By reflection symmetry, the projection on $y_i$ are vanishing. Recall that $M_{t-\beta}=\{({\bf{x}}, V({\bf{x}}, t-\beta))\}$, where $V$ is the profile function of $M_t$, and the renormalized flow $e^{\frac{\tau}{2}} {M}_{-e^{-\tau}-\beta}=\{(e^{\frac{\tau}{2}}{\bf{x}}, e^{\frac{\tau}{2}}V({\bf{x}}, -e^{-\tau}-\beta))\}$, let ${\bf{y}}=e^{\frac{\tau}{2}}{\bf{x}}$, the renormalized flow of shifted $k$-oval is the graph of
\begin{equation}\label{timeshiftedgraph}
    v^{\beta}({\bf{y}}, \tau)=e^{\frac{\tau}{2}}V(e^{-\frac{\tau}{2}}{\bf{y}}, -e^{-\tau}-\beta)
\end{equation}
over some compact domain depending on $\tau$.  By the backward expanding convex ancient $k$-oval $\mathcal M_t$ and the convexity of space-time track $\bigcup_{t<0} \{M_{t}\}$ from \cite[Theorem 2.1]{space-time-convex}, for fixed $\tau_0<0$, the function
\begin{equation}
\beta\mapsto v^{\beta}({\bf{y}} , \tau_{0})
\end{equation}
is convex nonnegative strictly increasing and has positive derivative whenever $\beta>\beta_0(\mathcal{M}, \tau_0)$ where $\beta_0(\mathcal{M}, \tau_0)$ is the largest  time shift parameter such that the above function achieves its minimum $0$ at $\beta_0(\mathcal{M}, \tau_0)$.  We also recall that
\begin{equation}
     v^{\beta}_{\mathcal{C}}({\bf y},\tau)= v^{\beta}({\bf y},\tau)\chi_{\mathcal{C}}( v^{\beta}({\bf y},\tau))
\end{equation}
where $\chi_{\cC}:\mathbb{R}_{+}\to [0,1]$ is a increasing cutoff function that satisfies $\chi_{\cC}(v)=1$ for $v\geq \tfrac 78  \theta$ and $\chi_{\cC}(v)=0$ for $v\leq  \tfrac 58  \theta$. Hence, we know that
\begin{equation}
\beta\mapsto \left\langle v^{\beta}_{\cC}({\bf{y}} , \tau_{0})-\sqrt{2(n-k)}, 1\right\rangle_\cH
\end{equation}
is monotonely increasing. Note that the monotonicity is strict as long as $ \left\langle v^{\beta}_{\cC}({\bf{y}} , \tau_{0})-\sqrt{2(n-k)}, 1\right\rangle_\cH\neq -\left\langle \sqrt{2(n-k)}, 1\right\rangle_\cH$ (in this case the flow vanishes as a point and $v^{\beta}({\bf{y}}, \tau)=e^{\frac{\tau}{2}}V(e^{-\frac{\tau}{2}}{\bf{y}}, -e^{-\tau}-\beta)=0$ for $\beta\leq \beta(\tau_0)$ for  some unique shifting parameter $\beta(\tau_0)$  depending on $\tau_0$), and  
\begin{align}
    \lim_{\beta\to \infty} \left\langle v^{\beta}_{\cC}({\bf{y}} , \tau_{0})-\sqrt{2(n-k)}, 1\right\rangle_\cH = -\left\langle \sqrt{2(n-k)}, 1\right\rangle_\cH<0,
\end{align}
\begin{align}
    \lim_{\beta\to -\infty} \left\langle v^{\beta}_{\cC}({\bf{y}} , \tau_{0})-\sqrt{2(n-k)}, 1\right\rangle_\cH = \infty.
\end{align}
Hence, by strict monotonicity and
continuity there exists a unique $\beta(\mathcal{M}, \tau_0)$ such that 
\begin{align}
    \left\langle v^{\beta}_{\cC}({\bf{y}} , \tau_{0})-\sqrt{2(n-k)}, 1\right\rangle_\cH  = 0.
\end{align}
This defines the shift map $\mathcal{S}$. Since all ancient ovals $\mathcal{M}\in \mathscr{A}$ are extinct in $t = 0$, neither element of $\mathcal{M}$ is a time shift of one another. Hence, $\mathcal{S}$ is injective. To establish the continuity of $\mathcal{S}$, note first that $(\mathcal{M}, \beta, {\bf y}) \mapsto v^{\beta}_{\cC}({\bf{y}} , \tau_{0})$ is continuous, and that
\begin{align}
    \lim _{\beta \rightarrow+\infty}v^{\beta}_{\cC}({\bf{y}} , \tau_{0}) = +\infty
\end{align}
uniformly on compact sets of $\mathscr{A} \times \mathbb{R}$. Therefore, if $\mathcal M_i \rightarrow \mathcal M$, then the sequence $\left\{\beta\left(\mathcal M_i, \tau_0\right)\right\}_{i=1}^{\infty}$ is bounded, so it subsequentially converges to some $\beta \in \mathbb{R}$. Then again by continuity, we infer $\mathfrak{p}_{+}\left(v^{\mathcal M, \beta}\left(\tau_0\right)\right)=\mathfrak{p}_{+}(\sqrt{2(n-k)})$, hence $\beta=\beta\left(\mathcal M, \tau_0\right)$ by uniqueness. As this is true for every converging subsequence, it follows that $\beta\left(\mathcal M_i, \tau_0\right) \rightarrow \beta\left(\mathcal M, \tau_0\right)$, proving the continuity of $\mathcal M \mapsto \beta\left(\mathcal M, \tau_0\right)$, and thus of $\mathcal{S}$. Then for any $\mathcal{M}^\beta\in \mathcal{S}(\mathscr{A)}$, $\mathcal{S}^{-1}(\mathcal{M}^\beta)$ denotes the time that the flow becomes extinct. Since the extinct time is continuous, $\mathcal{S}^{-1}$ is continuous as well. Moreover, if $\mathcal{M}\in \mathscr{A}$ is strongly $\kappa'$-quadratic from time $\tau_0+1$, then for $\tau\leq\tau_0$ we get
\begin{equation}
    \left\| v_{\cC}(\mathbf{y}, \tau_{0})-\sqrt{2(n-k)}+\frac{\sqrt{2(n-k)}}{4|\tau_{0}|}(|\mathbf{y}|^2-2k)\right\|_{\mathcal{H}}\leq C\frac{\kappa'}{|\tau|}.
\end{equation}
Since the time shifted oval $\mathcal{M}^\beta$ and $\mathcal{M}$ are related by \eqref{timeshiftedgraph}, which can be rewritten as 
\begin{align}
    v^\beta({{\bf y}}, \tau)=(1+b) v\left(\frac{{{\bf y}}}{1+b}, \tau-2 \log (1+b)\right),
\end{align}
where $b=\sqrt{1+\beta e^\tau}-1$. We can expand
\begin{align}
    &\quad v^\beta({\bf y}, \tau)-\sqrt{2(n-k)}\\
    &=\sqrt{2(n-k)} b-(1+b) \frac{\sqrt{2(n-k)}\left(\left(\frac{{\bf y}}{1+b}\right)^2-2\right)}{4|\tau-2 \log (1+b)|}+O\left(\kappa^{\prime} /|\tau|\right)\notag
\end{align}
in $\cH$-norm. It follows that for the unique solution of the orthogonality condition \eqref{positive_mode} we have
\begin{align}\label{bregion}
    |b| \leqslant C \kappa^{\prime} /|\tau|.
\end{align}
Thus, choosing $\kappa^{\prime}$ sufficiently small and $\tau_0 \leqslant \tau_*$ sufficiently negative, we conclude that $\mathcal{S}(\mathcal{M})$ is $\kappa$-quadratic at time $\tau_0$. This finishes the proof of the proposition.
\end{proof}
We can homeomophically shift ellipsoidal flows such that they have vanishing positive mode.
\begin{proposition}[transformation map]\label{S_property} 
Given any $\tau_0<0$, there is  a constant $H>e^{|\tau_0|}$ depending on $\tau_0$ with the following significance. For any $a\in \mathbb{R}_{+}^k$ and $\ell$ large  with the lifespan  $|T_\mu^{\ell, a}|=({\lambda_\mu^{\ell,a}})^2|t_\mu^{\ell,a}|\geq H$ of the ellipsoidal flow $M^{a, \ell}_{\mu, t}$ constructed in \eqref{m_ell_a}, where $\mu=\sum_{i=1}^k a^{i}$, there is a unique $\beta=\beta(a, \ell)\in [-\frac{H}{2}, \frac{H}{2}]$
depending continuously on $a\in \mathbb{R}^k_{+}$ such that the truncated renormalized profile function $v^{\beta}_\cC$ of the shifted flow ${M}^{a, \ell}_{t-\beta}$ satisfies
\begin{equation}\label{orth_cond_prop}
 \Big\langle  1 ,v^{\beta}_\cC(\tau_0)-\sqrt{2(n-k)} \Big\rangle_\cH=0.
\end{equation}
\end{proposition}
\begin{proof}
The argument is similarly as the above but we don't know the space-time track of ellipsoidal flows is convex or not. Firstly, let us observe that for any fixed $\tau_0<0$ the same argument as  in deriving \eqref{shift vanish +} of Proposition \ref{time-shift} with $\sqrt{2(n-k)}$ replaced by $\sqrt{n-k}$ and $2\sqrt{n-k}$, respectively, yields the existence of two continuous maps $\beta_1: \mathscr{A} \rightarrow \mathbb{R}$ and $\beta_2: \mathscr{A} \rightarrow \mathbb{R}$ such that the  profile functions of time shifted ovals ${M}_{t-\beta_1}$ and ${M}_{t-\beta_2}$ satisfy
\begin{align}
\mathfrak{p}_{+}\left(v_\cC^{\beta_1}\left(\cdot,\tau_0\right)-\sqrt{n-k}\right)=0 \quad \text { and } \quad \mathfrak{p}_{+}\left(v_\cC^{\beta_2}\left(\cdot,\tau_0\right)-2\sqrt{n-k}\right)=0.
\end{align}
Because $\mathscr{A}$ is compactified space of $\mathscr{A}^{\circ}$ and for fixed $\tau_0<0$, $\beta_1, \beta_2$ depends on elements in  $\mathscr{A}$ continuously, we can find a large enough constant $e^{|\tau_0|}<H<+\infty$ depending only on $\tau_0$  such that 
\begin{equation}
    \max\{\left|\beta_1\right|_{C^0},\left|\beta_2\right|_{C^0} \} \leq \frac{H}{4}.
\end{equation} 
Therefore, if

Now we show the existence of such $\beta=\beta(a, \ell)$, suppose towards a contradiction there is a sequence $\left(a_i, \ell_i\right)$ with $\ell_i \rightarrow \infty$ such that the asserted $\beta\left(a_i, \ell_i\right)$ in the interval $[-\frac{H}{2}, \frac{H}{2}]$ does not exist. After passing to a subsequence, the ellipsoidal flows $\mathcal M^{\ell_i, a_i}$ converge to some $\mathcal M \in \mathscr{A}$, and hence for $i$ large enough the corresponding truncated renormalized profile function of $M^{\ell_i,a_i}_{t-\beta_1}$ and ${M}^{\ell_i,a_i}_{t-\beta_2}$ satisfy
\begin{align}
\mathfrak{p}_{+}\left(v_{\cC,i}^{\beta_1}\left(\cdot,\tau_0\right)-\sqrt{2(n-k)}\right)<0 
\end{align}
and
\begin{align}
\mathfrak{p}_{+}\left(v_{\cC,i}^{\beta_2}\left(\cdot,\tau_0\right)-\sqrt{2(n-k)}\right)>0.
\end{align}
However, by the intermediate value theorem, we can find some $\beta\left(a_i, \ell_i\right)$ between $\beta_1(\mathcal{M})\in  [-\frac{H}{4}, \frac{H}{4}]$ and $\beta_2(\mathcal{M})\in [-\frac{H}{4}, \frac{H}{4}]$ such that the cylindrical profile function of $M^{\ell_i, a_i}_{t-\beta\left(a_i, \ell_i\right)}$ satisfies
\begin{align}\label{ai shift positive mode vanish}
\mathfrak{p}_{+}\left(v_{\cC,i}^{\beta(a_i,\ell_i)}\left(\cdot,\tau_0\right)-\sqrt{2(n-k)}\right)=0.  
\end{align}
which is a contradiction with our nonexistence assumption.

Next, to address uniqueness, observe that for every $\mathcal{M}\in\mathscr{A}$, according to the proof in Proposition \ref{time-shift}, the function 
\begin{equation}\label{beta to projection function}
\beta\mapsto \left\langle v^{\beta}_{\cC}({\bf{y}} , \tau_{0})-\sqrt{2(n-k)}, 1\right\rangle_\cH
\end{equation}
is convex strictly monotone increasing and has strictly positive derivative whenever $\beta>\beta_0(\mathcal{M}, \tau_0)$, where $\beta_0(\mathcal{M}, \tau_0)$ is the largest time shift parameter such that the above  function \eqref{beta to projection function} achieves its minimum at $\beta_0(\mathcal{M}, \tau_0)$:
\begin{align}
    \left\langle v^{\beta}_{\cC}({\bf{y}} , \tau_{0})-\sqrt{2(n-k)}, 1\right\rangle_\cH= -\left\langle \sqrt{2(n-k)}, 1\right\rangle_\cH.
\end{align}
Therefore, it follows again from smooth convergence and compactness that for $\ell_i$ large enough, the value  $\beta(a_i, \ell_i)$ is unique. Otherwise, by taking subsequential limit to some element $\mathcal{M}\in \mathscr{A}$ as $i\to \infty$ and noticing \eqref{ai shift positive mode vanish}, we obtain that the the above function \eqref{beta to projection function} has  vanishing derivative at the limiting time-shift parameter point $\beta_{\infty}\in (\beta_0(\mathcal{M}, \tau_0), +\infty)$ of the sequence of time shift parameters $\beta(a_i, \ell_i)$ as $i\to \infty$, which is a contradiction with nonvanishing derivative property at this point.  Finally, continuity of the map $a\mapsto \beta(a,\ell)$ follows from uniqueness and boundedness of $\beta(a,\ell)$ as discussed in Proposition \ref{time-shift}.

\end{proof}

\section{Existence with prescribed spectral value}
In this section, we discuss the existence with prescribed spectral value stated in Theorem
\ref{prescribed_eccentricity_restated} and  Corollary \ref{prescribed_eccentricity_restated'}. 

To this end, let us denote 
\begin{equation}
\mathcal{M}^{a}_{i}=\mathcal{M}_{\ell_i,a}^{\beta(a, \ell_i)}
\end{equation}
be the shifted ellipsoidal flow with ellipsoidal parameter $a$ and $\mathcal{M}^{a_i}_{i}=\mathcal{M}_{\ell_i,a_i}^{\beta(a_i, \ell_i)}$ be a sequence of the shifted ellipsoidal flows converging to some ancient oval $\mathcal{M}\in \mathscr{A}'=\mathcal{S}(\mathscr{A})$. In particular, we define $\mathscr{S}'\subset \mathcal{A}'$ containing all  $\mathrm{O}(k)\times\mathrm{O}(n+1-k)$-symmetric shifed ancient ovals, which are limits of converging sequences of  $\mathrm{O}(k)\times\mathrm{O}(n+1-k)$-symmetric shifted ellipsoidal flows.

Then, for fixed $\tau_0$ we consider the spectral  map $\mathscr{E}(\mathcal{M}):\mathscr{A}'\to \mathbb{R}^{k}$
\begin{align}\label{mathscr E}
\mathscr{E}(\mathcal{M})=\mathscr{E}(\mathcal{M})(\tau_0)=(\mathscr{E}^1(\mathcal{M})(\tau_0),\dots, \mathscr{E}^k(\mathcal{M})(\tau_0)),
\end{align}
where
\begin{equation}
   \mathscr{E}^j(\mathcal{M})(\tau_0)= {\Big\langle v^{\mathcal{M}}_{\cC}(\tau_{0}),  2-|{\bf{y}}|^2_{j} \Big\rangle_\cH}\quad j=1,\dots,k
\end{equation}
and we can also  naturally extend the spectral map to the associated spectral map for shifted ellipsoidal flows $E_i: \mathbb{R}^k_{+}\to \mathbb{R}^k$:
\begin{equation}
    {E}_i(a)=( {E}^1_i(a), \dots,  {E}^k_i(a))=(\mathscr{E}^1(\mathcal{M}^{a}_i),\dots, \mathscr{E}^k(\mathcal{M}^{a}_i))
\end{equation}
Let $\mathbf{S}_k$ be the discrete 
$k$-coordinates permutation group. We see $\mathscr{E}:\mathscr{A}'\to \mathbb{R}^{k}$ and $E_i: \mathbb{R}^{k}_{+}\to \mathbb{R}^{k}$ are ${\bf{S}}_{k}$-equivariant continuous map continuous map  and can be naturally continuously extended to their domain boundary.

Let $\Delta\subset\mathbb R^k$ denotes the standard $k-1$-dimensional simplex with $k$ vertices $(1,0,\cdots,0)$, $\cdots$, $(0,0,\cdots,1)$ and barycenter $(\frac{1}{k},\frac{1}{k},\cdots,\frac{1}{k})$. For $\delta>0$, let scaled simplex defined by $\Delta_\delta = \delta\Delta$. We also denote $\Delta_\delta(b) = \{{\bf\mu}-(\frac{b-1}{k}, \dots, \frac{b-1}{k}):{\bf\mu}\in \Delta_\delta\}$. Then we can define the shifted $\mathbf{S}_k$-symmetric $k$-dimensional simplicial cylinder with size $\eta$ and center at $(\frac{b}{k}, \dots, \frac{b}{k})$
\begin{equation}\label{defcylinder}
    \cylinders_\eta (b)= \bigcup_{d\in [b-\eta, b+\eta]} \Delta_\delta(d).
\end{equation}


Then we have the following theorem
\begin{theorem}[existence with prescribed spectral value]\label{prescribed_eccentricity_restated} 
There exist small constants $\delta>0$, $\kappa>0$ and very negative constant $\tau_\ast>-\infty$ with the following significance.  For every $\tau_{0}\leq \tau_{*}$ and every 
\begin{equation}\label{targetmu}
    \mu=(\mu_1, \dots, \mu_k)\in \cylinders_{\frac{\delta\kappa}{|\tau_0|}}(\bar{e})
\end{equation}
where 
\begin{equation}
\bar{e}=\frac{\sqrt{2(n-k)}\||\mathbf{y}|^2-2k\|^2_{\mathcal{H}}}{4|\tau_0|},
\end{equation}
and $\cylinders_{\frac{\delta\kappa}{|\tau_0|}}(\bar{e})$ denotes the $k$-dimensional simplicial cylinder centered at the point $(\frac{\bar{e}}{k}, \dots, \frac{\bar{e}}{k})$ with size $\frac{\delta\kappa}{|\tau_0|}$ defined in \eqref{defcylinder}, there exists a time-shifted ancient oval  $\mathcal{M}\in\mathscr{A}'$ that is $\kappa$-quadratic at $\tau_0$  and satisfies
\begin{equation}
    \mathscr{E}(\mathcal{M})= \mathscr{E}(\mathcal{M})(\tau_0)=\mu.
\end{equation}
\end{theorem}

\begin{proof}
\textbf{Spectral surjectivity along diagonal and compactness.} Let  $K=K(\kappa, \tau_0)$  be the  ${\bf S}_k$-symmetric $k$-dimensional set of all points $a=(a^1, \dots a^k)\in \mathbb{R}^{k}_{+}$ such that the shifted ellipsoidal flow $\mathcal{M}^{a}_i$  satisfies 
\begin{equation}\label{target compact}
    {E}_i(a)=\mathscr{E}(\mathcal{M}^{a}_i)\in \cylinders_{\frac{\delta\kappa}{|\tau_0|}}(\bar{e}).
\end{equation}
Firstly, we note that $K\not=\emptyset$. Indeed, possibly after decreasing $\tau_\ast$, given any $\tau_0\leq\tau_\ast$ negative enough, by asymptotics and Proposition \ref{prop_orthogonality} (orthogonality) on   the (up to parabolic rescaling) unique  $\mathrm{O}(k)\times\mathrm{O}(n+1-k)-$symmetric ancient oval\cite{DH_ovals} and Proposition \ref{time-shift}, for $\delta>0$ in \eqref{targetmu} small enough, after proper rescaling and unique time shift we can find three $\mathrm{O}(k)\times\mathrm{O}(n+1-k)$-symmetric ancient ovals $\mathcal{M}^{\textrm{sym}, \pm}\in \mathscr{S}', 
 \mathcal{M}^{\textrm{sym}}\in \mathscr{S}'$ such that
\begin{equation}\label{limit symmetry oval value}
    \mathscr{E}(\mathcal{M}^{\textrm{sym}, \pm})=(\frac{\bar{e}\pm\frac{2\delta\kappa}{|\tau_0|}}{k}, \dots, \frac{\bar{e}\pm\frac{2\delta\kappa}{|\tau_0|}}{k})\quad
\Big\langle v_{\cC}^{\mathcal{M}^{\textrm{sym}, \pm}}(\tau_0),  2k-|{\bf{y}}|^2\Big\rangle_\cH=\bar e\pm\frac{2\delta\kappa}{|\tau_0|}
\end{equation}
\begin{equation}\label{limit symmetry oval value 2}
    \mathscr{E}(\mathcal{M}^{\textrm{sym}})=(\frac{\bar{e}}{k}, \dots, \frac{\bar{e}}{k})\quad
\Big\langle v_{\cC}^{\mathcal{M}^{\textrm{sym}}}(\tau_0),  2k-|{\bf{y}}|^2\Big\rangle_\cH=\bar e.
\end{equation}
 Recall the construction of ancient ovals from \cite{DH_ovals} and using the homeomorphism and uniqueness of the restricted shifted map $\mathcal{S}: \mathscr{S}\to \mathscr{S}'$ in Proposition \ref{time-shift}, we can require that $\mathcal{S}^{-1}(\mathcal{M}^{\textrm{sym}, \pm})$ and $\mathcal{S}^{-1}(\mathcal{M}^{\textrm{sym}})$ are limits of unshifted ellipsoidal flows with ellipsoidal parameters $\bar{a}_{\pm}$ and $\bar{a}$ respectively.
 Using homeomorphism and uniqueness  of the restricted shifted map on ellipsoidal flow class in Proposition \ref{S_property}, there are three sequences of time shifted $\mathrm{O}(k)\times\mathrm{O}(n+1-k)-$symmetric  ellipsoidal flows $\mathcal{M}^{\bar{a}_{\pm}}_{i}$ and $\mathcal{M}^{\bar{a}}_{i}$ with ellipsoidal parameters $\bar{a}_{\pm}=(\frac{\bar{\mu}_{\pm}}{k}, \dots, \frac{\bar{\mu}_{\pm}}{k})$, $\bar{a}=(\frac{\bar{\mu}}{k}, \dots, \frac{\bar{\mu}}{k})$ for some $\bar{\mu}_{\pm}>0, \bar{\mu}>0$ such that they are $\kappa$-quadratic at $\tau_0$ and
\begin{equation}\label{sym oval ellipsoid limit}   
\mathcal{M}^{\textrm{sym}, \pm}=\lim_{i\to \infty}\mathcal{M}^{\bar{a}_{\pm}}_{i}\quad \mathcal{M}^{\textrm{sym}}=\lim_{i\to i}\mathcal{M}^{\bar{a}}_{i},
\end{equation}
and their profile functions $v^{\bar a_{\pm}}_i$, $v^{\bar a}_i$ satisfy
\begin{equation}
\mathfrak{p}_{+}(v^{\bar{a}_{\pm}}_{\cC, i})(\tau_{0})=0\quad \mathfrak{p}_{+}(v^{\bar{a}}_{\cC, i})(\tau_{0})=0
    \end{equation}
    Then we define
\begin{equation}
\bar{e}_{i, \pm}=\sum_{j=1}^{k}\bar{e}_{i, \pm}^{j}=\sum_{j=1}^{k}{E}^{j}_i(\bar{a}_{\pm})=\sum_{j=1}\mathscr{E}^{j}(\mathcal{M}^{\bar{a}_{\pm}}_i)
\end{equation}
\begin{equation}\label{barei}
\bar{e}_{i}=\sum_{i=1}^{k}\bar{e}_{i}^{j}=\sum_{j=1}{E}^{j}_i(\bar{a})=\sum_{j=1}\mathscr{E}^{j}(\mathcal{M}^{\bar{a}}_i)
\end{equation}
and note 
\begin{equation}
\bar{e}_{i,\pm}^{1}=\dots=\bar{e}_{i,\pm}^{k}=\frac{\bar{e}_{i,\pm}}{k}\quad \bar{e}_{i}^{1}=\dots=\bar{e}_{i}^{k}=\frac{\bar{e}_{i}}{k}
\end{equation}
and 
\begin{equation}\label{eibar to ebar}
   \lim_{i\to +\infty}\bar{e}_{i, \pm}=\bar{e}_{\pm}=\bar e\pm \frac{2\delta\kappa}{|\tau_0|}\quad  \lim_{i\to +\infty}\bar{e}_{i}=\bar{e}
\end{equation}
Therefore, for $i$ large enough
\begin{equation}\label{end symmetry points value}
  \bar{e}_{i, \pm}\in  (\bar{e}_{\pm}-\frac{\delta\kappa}{10|\tau_0|}, \bar{e}_{\pm}+\frac{\delta\kappa}{10|\tau_0|})\quad \bar{e}_{i}\in  (\bar{e}-\frac{\delta\kappa}{10|\tau_0|}, \bar{e}+\frac{\delta\kappa}{10|\tau_0|}).
\end{equation}
and in particular for $\delta>0, \kappa>0$ small
\begin{equation}
  \bar{e}_{i, -}<  \bar{e}_i<\bar{e}_{i, +}\quad  \bar{e}_{i, -}<  \bar{e}<\bar{e}_{i, +}.
\end{equation}
Then we define the diagonal set in $\mathbb R^k$ 
\begin{equation}
    \textrm{Diag}=\{\mu \in \mathbb{R}^{k}: \mu^{1}=\dots=\mu^k\}.
\end{equation}
For restriction $E_i|_{\text{Diag}}: \text{Diag}\subset \mathbb{R}^{k}\to \text{Diag}\subset \mathbb{R}^{k}$, by target points value range \eqref{end symmetry points value} and limiting $O(k)\times O(n+1-k)$ symmetric oval target values in \eqref{limit symmetry oval value}
and intermediate value theorem, we know that the restriction
\begin{equation}\label{ends oval surjectivity}
    E_i|_{\text{Diag}}: \text{Diag}\subset \mathbb{R}^{k}\to \{\mu\in\mathbb{R}^{k}_{+}:   \mu_1=\dots=\mu_k\in [\frac{\bar{e}}{k}-\frac{\delta\kappa}{k|\tau_0|}, \frac{\bar{e}}{k}+\frac{\delta\kappa}{k|\tau_0|}]\}
\end{equation}
is surjective and $K\not=\emptyset$. In particular, 
\begin{align}\label{diagsurjective}
   \cylinders_{\frac{\delta\kappa}{|\tau_0|}}(\bar{e})\cap {\rm Diag} \subset \mathscr{E}(\mathscr{S}').
\end{align}
For $\bar{e}_i$ given in \eqref{barei}, we note that $ \cylinders_{\frac{\delta\kappa}{|\tau_0|}}(\bar{e}_i)$ converges to $\cylinders_{\frac{\delta\kappa}{|\tau_0|}}(\bar{e})$ as $i\to \infty$.

Then,  for later applying techniques about differential topology for smooth manifolds and ease of notation, let $\cylinders^k_{\frac{\delta\kappa}{|\tau_0|}, i}$ be the ${\bf S}_k$-symmetric preserved, smooth isotopic $\eta_i$-close approximation to the simplicial cylinder $\cylinders^k_{\frac{\delta\kappa}{|\tau_0|}}(\bar{e}_i)$  (one can obtain it by small boundary distance level sets or by short time mean curvature flow) and the simplicial cylinders $\cylinders^k_{\frac{\delta\kappa}{|\tau_0|}, i}$ converge to $\cylinders_{\frac{\delta\kappa}{|\tau_0|}}(\bar{e})$ as small numbers $\eta_i\to 0$,. We also denote top and bottom $(k-1)$-dimensional simplicial faces of simplicial cylinder .
\begin{align}\label{Top}
    \simplex^{+}_{\frac{\delta\kappa}{|\tau_0|}, i}  =\cylinders_{\frac{\delta\kappa}{|\tau_0|}, i}\bigcap \{\mu\in\mathbb{R}^{k}:   \mu^1+\dots+\mu^k = \bar{e}_i+\frac{\delta\kappa}{|\tau_0|}]\}
\end{align}
and
\begin{align}\label{Bottom}
    \simplex^{-}_{\frac{\delta\kappa}{|\tau_0|}, i} =\cylinders_{\frac{\delta\kappa}{|\tau_0|}, i}\bigcap \{\mu\in\mathbb{R}^{k}:   \mu^1+\dots+\mu^k = \bar{e}_i-\frac{\delta\kappa}{|\tau_0|}]\}.
\end{align}
where $\bar{e}_i$ is given in \eqref{barei}.  Moreover, the skeleton structure, top and bottom simplicial faces denoted by $\simplex^{\pm, k}_{\frac{\delta\kappa}{|\tau_0|}, i}$ and their  vertices and barycenter  $\mu$ are naturally labeled in the smoothly approximated  simplicial cylinder $\cylinders^k_{\frac{\delta\kappa}{|\tau_0|}, i}$ by natural transfer. We also claim that
\begin{claim}[compactness]\label{claim_boundedness}
$K$ is compact and 
\begin{equation}
    d(K, \partial\mathbb{R}^{k}_{+})\geq c=c(\kappa, \tau_0)>0\quad \textrm{Diam}(K)\leq C=C(\kappa, \tau_0)<\infty.
\end{equation}
\end{claim}

\begin{proof}[Proof of Claim]
Suppose towards a contradiction that for all $\kappa>0$ and all $\tau_0$ negative, there is a sequence of $a_{i}\in K$ such that $a_{i}$ converges to $\partial \mathbb{R}^k_{+}\setminus \{(0,\dots, 0)\}$ as $i\to \infty$ or 
$|a_{i}|_1$  converges to $0$ or $+\infty$ tangentially along $\partial \mathbb{R}^k_{+}\setminus \{(0,\dots, 0)\}$ as $i\to \infty$, then by \cite{HaslhoferKleiner_meanconvex, Brendle_inscribed} and the construction of these ovals in \cite{DH_ovals}, we have $\mathcal{M}^{a_i}_i$ which are $\kappa$-quadratic at $\tau_0$ subsequentially converge to a limit $\mathcal{M}^{a_\infty}_i$ with line splitting along some coordinates directions as $i\to +\infty$, which will cause $\lim_{i\to +\infty} E_i(a_{i})$ having vanishing coordinates along these directions and this is a contradiction with \eqref{target compact}  for  small $\kappa>0$ and very negative $\tau_0$. Therefore, we only need to exclude the cases where 
either  $|a_{i}|_1$ converges to $0$ or $+\infty$ inside a cone of $\mathbb{R}^k_{+}$ without intersection with $\partial\mathbb{R}^k_{+}\setminus \{(0,\dots, 0)\}$.  

By computation,  the renormalized flow ${M}_{\gamma, \tau}$  of $e^{\frac{\gamma}{2}}$-parabolic rescaled flow $\mathcal{M}_{\gamma}$ of ancient oval $\mathcal{M}\in \mathscr{A}'$
has the profile function 
\begin{equation}
    \bar{v}^{\gamma}({\bf{y}}, \tau)=e^{\frac{\gamma+\tau}{2}}V(e^{-\frac{\gamma+\tau}{2}}{\bf{y}}, -e^{-(\gamma+\tau)}-\beta).
\end{equation} 
When $\tau=\tau_0$, we have the following blowup limit
\begin{equation}
    \lim_{\gamma\to -\infty}\bar{v}^{\gamma}({\bf{y}}, \tau_0)=\sqrt{2(n-k)}
\end{equation}
and blowdown limit
\begin{equation}
    \lim_{\gamma\to +\infty}\bar{v}^{\gamma}({\bf{y}}, \tau_0)=\sqrt{2n-|{\bf{y}}|^2}
\end{equation}
which are independent on $\tau_0$. Hence, fixed $\tau_0\leq \tau_*$ is  negative enough
and $|a_{i}|_1$ converges to $0$ or $+\infty$ inside a cone of $\mathbb{R}^k_{+}$ without intersection with $\partial\mathbb{R}^k_{+}\setminus \{(0,\dots, 0)\}$, we have the shifted ellipsoidal flows $\mathcal{M}^{a_i}_i$ 
would be very close to a  spherical blowup limit or cylindrical blowdown limit, which would again violate \eqref{target compact} or $\kappa$-quadraticity at $\tau_0$. 
\end{proof}

\textbf{Spectral injectivity along diagonal}. Fix constants $\tau_\ast>-\infty$ negative enough and $\kappa>0$ small enough such that Proposition \ref{S_property} (transformation map) and Theorem \ref{point_strong} (strong $\kappa$-quadraticity) apply. Our map $\mathscr{E}: \mathscr{A}'\to \mathbb{R}^k$ is a continuous map. For any time-shifted  $\mathcal{M}=\lim_{i\to \infty}\mathcal{M}^{a_i}_i\in \mathscr{A}'$ which is $\kappa$-quadratic at $\tau_0$ and satisifes vanishing positive mode condition\eqref{shift vanish +}, where $\sum_{j=1}^k a^j_i = \mu_0 = e^{\frac{\gamma_0}{2}}$. Then we can rewrite the ellipsoidal parameter $a_i$ as $(e^\frac{\gamma_0}{2}, {\bf \theta}_i)$, where ${\bf \theta}_i$ is a unit vector in $S^{k-1}$, which converges to a ${\bf \theta}\in S^{k-1}$. Let $V$ be the profile function of $\mathcal{M}$ and  ${\bf{y}}=e^{\frac{\tau}{2}}{\bf{x}}$, we consider the change of profile function with respect to the rescaling parameter $\gamma$. Then the renormalized flow of shifted ancient oval is the graph of 
\begin{equation}
    v^{\gamma}({\bf{y}}, \tau)=e^{\frac{\gamma + \tau}{2}}V(e^{-\frac{\gamma+\tau}{2}}{\bf{y}}, -e^{-\gamma-\tau}-\beta(\gamma,{\bf \theta})),
\end{equation}
Then the renormalized profile functions $v^{\gamma}$ and $v$ are related by 
\begin{align}\label{vgamma}
    v^{\gamma}({\bf y}, \tau) = (1+\tilde b)v(\frac{{\bf y}}{1+\tilde b}, (1+\tilde \Gamma)\tau),
\end{align}
where 
\begin{align}\label{tildebgamma}
    \tilde b = \sqrt{1+\tilde\beta e^\tau}-1,\quad \tilde\Gamma = \frac{\gamma-\log(1+\tilde\beta e^\tau)}{\tau},\quad \text{and} \quad \tilde\beta(\gamma) = \beta(\gamma) e^\gamma.
\end{align}
Since by Theorem \ref{point_strong} (strong $\kappa$-quadraticity) $v$ is strongly $5\kappa$-quadratic from $\tau_0$ which is very negative, then for $\tau\leq\tau_0$ considering Gaussian tail estimates we get
\begin{equation}
    \left\| v(y, \tau)-\sqrt{2(n-k)}+\frac{\sqrt{2(n-k)}}{4|\tau|}(|y|^2-2k)\right\|_{\mathcal{H}}\leq \frac{10\kappa}{|\tau|}.
\end{equation}
Since the renormalized profile functions $v^{\gamma}$ and $v$ are related by \eqref{vgamma}, we can expand
\begin{align}
    &\quad v^\gamma(y, \tau) - \sqrt{2(n-k)}\\
    &= \sqrt{2(n-k)}\tilde b - (1+\tilde b)\frac{\sqrt{2(n-k)}\left(\left(\frac{{y}}{1+\tilde{b}}\right)^2-2k\right)}{4(1+\tilde \Gamma)|\tau|}+O\left(\kappa/|\tau|\right)
\end{align}
in $\cH$-norm. It follows that for the unique solution of the orthogonality condition
\begin{equation}
\mathfrak{p}_{+}\big(v^{\gamma}_{\cC}(\cdot , \tau_{0})-\sqrt{2(n-k)}\big)=0,
\end{equation}
we have
\begin{align}
    |\tilde b| \leqslant 10\kappa/|\tau_0|.
\end{align}
Thus we have 
\begin{align}\label{tildebeta}
    |\tilde \beta e^\tau|\leq C\frac{\kappa}{|\tau|}.
\end{align}

\begin{claim}[rescaling monotonicity]\label{Rescaling monotonicity}
There exist $\kappa>0$ and $\tau_*>-\infty$ with the following significance. For any given reference ancient oval $\mathcal{M}$ which is $\kappa$-quadratic at time $\tau_0 \leq \tau_*$, the profile function of $\mathcal{M}^{\beta(\gamma), \gamma}$ satisfies
\begin{align}\label{negative derivative}
\langle \frac{\partial}{\partial \gamma} v_\cC^{\gamma}({\bf y}, \tau), |{\bf y}|^2 - 2k\rangle_\cH= -\frac{\sqrt{2(n-k)}||\psi_{k+2}||^2_\cH}{4|\tau|^2} +O(\frac{\kappa}{|\tau|^2})<0
\end{align}
holds for all $\tau \leq \tau_0$ and all $|\gamma| \leq 10\kappa|\tau|$.
\end{claim}
\begin{proof}[Proof of Claim]
For any fixed reference ancient oval $\mathcal{M}$ which is $\kappa$-quadratic at $\tau_0$, since  $\beta=\beta(\gamma)$ ($\beta$ also depends on $\mathcal{M}$) is uniquely determined by the equation
\begin{align}
    F(\beta,\gamma) = \left\langle v^{\beta}_{\cC}({\bf{y}} , \tau_{0})-\sqrt{2(n-k)}, 1\right\rangle_\cH  = 0.
\end{align}
We first show that $\tilde\beta$ is a smooth function of $\gamma$. For all $|\gamma|\leq 10\kappa|\tau|$, by \eqref{b-k+1} and \eqref{b-k+2} in Claim \ref{JPhiestimates}, we have 
\begin{align}\label{Fbeta}
    \frac{\partial F}{\partial \tilde{\beta}} &= \langle \frac{\partial}{\partial \tilde \beta} v_\cC^{\gamma}({\bf y}, \tau), 1\rangle_\cH 
    = \frac{\partial F}{\partial \tilde b}\frac{\partial \tilde{b}}{\partial \tilde \beta} + \frac{\partial F}{\partial \tilde \Gamma}\frac{\partial \tilde{\Gamma}}{\partial \tilde \beta}\\
    &= \sqrt{2(n-k)}||\psi_{k+1}||^2_\cH\frac{e^\tau}{2\sqrt{1+\tilde{\beta}e^\tau}} + O(\frac{e^\tau}{2\tau\sqrt{1+\tilde{\beta}e^\tau}}) - O(\kappa)\frac{e^\tau}{\tau(1+\tilde{\beta}e^\tau)}\notag\\
    &>0 \notag
\end{align}
for all $\tau\leq\tau_0$. Thus, $\tilde \beta$ is a smooth function of $\gamma$ according to the implicit function theorem, thus $\beta$ is also a smooth function of $\gamma$ by \eqref{tildebgamma}.
Moreover, by \eqref{Fbeta}, and using implicit function theorem, we get 
\begin{align}\label{betagamma}
    \frac{\partial \tilde \beta}{\partial\gamma} &= -(\frac{\partial F}{\partial \gamma})/(\frac{\partial F}{\partial \tilde \beta})\notag\\
    &= -(\frac{\partial F}{\partial \tilde\Gamma}\frac{\partial \tilde{\Gamma}}{\partial \gamma})/(\frac{\partial F}{\partial \tilde \beta})\notag\\
    &= (O(\kappa)\frac{1}{\tau})/(\frac{\partial F}{\partial \tilde \beta})\notag\\
    &= O(\frac{\kappa}{\tau e^\tau})+O(\frac{\kappa}{\tau^2 e^\tau}). 
\end{align}
Thus, by direct calculation and \eqref{betagamma}, we have 
\begin{align}
    \frac{\partial \tilde b}{\partial \gamma} = \frac{e^\tau}{2\sqrt{1+\tilde\beta e^\tau}}\frac{\partial\tilde{\beta}}{\partial \gamma} = O(\frac{\kappa}{\tau}),
\end{align}
\begin{align}
    \frac{\partial \tilde \Gamma}{\partial \gamma} = \frac{1}{\tau} + \frac{e^\tau}{(1+\tilde\beta e^\tau)|\tau|}\frac{\partial \tilde\beta}{\partial \gamma} = \frac{1}{\tau} + O(\frac{\kappa}{\tau^2}).
\end{align}
for all $\tau\leq\tau_0$. 
Then using \eqref{b-k+2} and \eqref{gamma-k+2} in Proposition \ref{JPhiestimates} and chain rule of derivative, we have 
\begin{align}
    \langle \frac{\partial}{\partial \gamma} v_\cC^{\gamma}({\bf y}, \tau), |{\bf y}|^2 - 2k\rangle_\cH &= \langle \frac{\partial v_\cC^\gamma}{\partial \tilde b}\frac{\partial \tilde b}{\partial \gamma} + \frac{\partial v_\cC^\gamma}{\partial \tilde \Gamma}\frac{\partial \tilde \Gamma}{\partial \gamma}, |{\bf y}|^2 - 2k\rangle_\cH \notag\\
    &= -\frac{\sqrt{2(n-k)}||\psi_{k+2}||^2_\cH}{4|\tau|^2} +O(\frac{\kappa}{|\tau|^2})<0.
\end{align}
for all $\tau\leq \tau_0$.
\end{proof}
For ellipsoidal flows, we have the following monotonicity.
\begin{claim}[rescaling monotonicity for ellipsoidal flow]\label{Rescaling monotonicity for ellipsoidal flow}
There exist $\kappa>0$ and $\tau_*>-\infty$ with the following significance. Given any reference point $(a^1, \dots, a^k)\in \mathbb{R}^k_{+}$ with $\sum_{j=1}^{k}{a^{j}}=e^{\frac{\gamma_0}{2}}$ such that the sequence of shifted reference ellipsoidal flows $\mathcal{M}^{a}_i$ is $\kappa$-quadratic at time $\tau_0 \leq \tau_*$, then for $i\geq i_0$, where $i_0$ only depends on $\kappa, \tau_0$,  the corresponding truncated profile functions $v_{\cC,i}^{\gamma}$ satisfy
\begin{align}
\langle \frac{\partial}{\partial \gamma} v_{\cC,i}^{\gamma}({\bf y}, \tau_0), |{\bf y}|^2 - 2k\rangle_\cH<0
\end{align}
for  $a\in K$ in Claim \ref{claim_boundedness} and all $|\gamma| \leq 10\kappa|\tau_0|$, where 
\begin{equation}
    d(K, \partial\mathbb{R}^{k}_{+})\geq c=c(\kappa, \tau_0)>0\quad \textrm{Diam}(K)\leq C=C(\kappa, \tau_0)<\infty.
\end{equation}
\end{claim}
\begin{proof}[Proof of Claim]
    Suppose towards a contradiction for all $\kappa>0$,  there  is $\tau_0>0$ very negative and a sequence of shifted ellipsoidal flows $\mathcal{M}^{a_i}_i$ is $\kappa$-quadratic at time $\tau_0$ but with 
    \begin{align}
\langle \frac{\partial}{\partial \gamma}|_{\gamma=\gamma_i} v_{\cC,i}^{\gamma}({\bf y}, \tau_0), |{\bf y}|^2 - 2k\rangle_\cH\geq 0
\end{align}
holds for $|\gamma_i|\leq 10\kappa |\tau_0|$
By Claim \ref{claim_boundedness} (compactness), after passing through a subsequence, we obtain a limit flow $\mathcal{M}^{\infty}$ which is $\kappa$-quadratic at time $\tau_0$  with truncated profile function $v_{\cC,\infty}^{\gamma}$ in local smooth convergence sense and $\gamma_i \to \gamma_{\infty}\in [-10\kappa |\tau_0|, 10\kappa |\tau_0|]$ as $i\to \infty$ but
 \begin{align}
\langle \frac{\partial}{\partial \gamma}|_{\gamma=\gamma_\infty} v_{\cC,\infty}^{\gamma}({\bf y}, \tau_0), |{\bf y}|^2 - 2k\rangle_\cH\geq 0,
\end{align} 
which is a contradiction with \eqref{negative derivative} in Claim \ref{Rescaling monotonicity} (rescaling monotonicity).
\end{proof}

{\bf Equivariant transversal  homotopic spectral approximation.} We recall that by \eqref{sym oval ellipsoid limit} and \eqref{end symmetry points value},  $\mathcal{M}_i^{\bar{a}}$ is the sequence of reference shifted $\mathrm{O}(k)\times\mathrm{O}(n+1-k)-$symmetric  ellipsoidal flows with ellipsoidal parameter $\bar{a}\in \textrm{Diag}\cap\mathbb{R}^{k}_{+}$, vanishing positive mode projection 
\begin{equation}
\mathfrak{p}_{+}v^{\bar{a}}_{\cC, i}(\tau_{0})=0
    \end{equation}
and is $\frac{\delta\kappa}{100}$-quadratic at $\tau_0$ for some small $\delta>0$ and converging to the reference $\mathrm{O}(k)\times\mathrm{O}(n+1-k)-$symmetric 
ancient oval $\mathcal{M}^{\textrm{sym}}$ and
\begin{equation}
 E_i(\bar{a})=(\frac{\bar{e}_i}{k}, \dots, \frac{\bar{e}_i}{k}).
\end{equation}
Let $K$ be the ${\bf S}_k$-symmetric compact set of $\mathbb{R}^{k}_{+}$ from Claim \ref{claim_boundedness}, and $D\subset \mathbb{R}^k_{+}$ be a compact smooth manifold with smooth manifold boundary satisfying $K\subset \textrm{Int}(D)\subset \mathbb{R}^k_{+}$  Now, we consider the continuous map $E_i: D \to \mathbb{R}^k_{+}\subset \mathbb{R}^k_{+}\cup\{\infty\}$, where $\approx$ means diffemorphism relation. By equivariant Tietze extension theorem in \cite{feragen2010equivariant} \cite{gleason1950spaces} \cite[Theorem 1.4.4]{palais1960classification}, we can extend it as a continuous map ${E}^{c}_i: S^{k} \to S^{k}$ without influence its value in the set $D$, where $\mathbb{R}^k_{+}\cup\{\infty\}$ is diffeomorphically identified with $S^k$. Then, by the equivariant smooth approximation theory in \cite[Corollary 1.12]{wasserman1969equivariant}\footnote{Discrete finite group is $0$-dimensional compact Lie group.}, for any $\epsilon>0$, we can find a ${\bf S}_k$-symmetric smooth map $E^{s}_i:  S^{k} \to S^{k}$ homotopic to the continuous map ${E}^{c}_i: S^{k} \to S^{k}$ and 
\begin{equation}
    |{E}^{s}_i-{E}^c_i|_{C^0}<\epsilon/2.
\end{equation}

By the theory of equivariant transversality in \cite[Proposition 2.2]{wasserman1975equivariant}, \cite[Theorem 6.14.4]{field2007dynamics} \cite[Lemma 3.2]{wasserman1969equivariant} which uses \cite[Theorem 1.35]{milnor1958lectures}, we can find a ${\bf S}_k$-symmetric smooth map $\tilde{E}_{i}$ homotopic to $E^{s}_i$ such that $\tilde{E}_{i}$ is transversal to the $k-1$-dimensional  smoothly $\eta_i$-approximated simplicial cylinder boundary $\partial \cylinders^k_{\frac{\delta\kappa}{|\tau_0|}, i}(\bar{e})$, which is denoted by 
\begin{equation}
    \tilde{E}_{i} \pitchfork\partial\cylinders^k_{\frac{\delta\kappa}{|\tau_0|}, i}(\bar{e}_i)
\end{equation}
and due to compactness of $S^{k-1}$ or $D$, it satisfies
\begin{equation}
     |\tilde{E}_i-{E}^{s}_i|_{C^0}<\epsilon/2
\end{equation}
In particular,  the ${\bf S}_k$-symmetric map $\tilde{E}_i|_{D}\pitchfork  \partial \cylinders^k_{\frac{\delta\kappa}{|\tau_0|}, i}(\bar{e}_i)$ are homotopic in $D$ and
\begin{equation}\label{trans-smooth approximation}
     |\tilde{E}_i|_{D}-E_i|_{D}|_{C^0}<\epsilon.
\end{equation}
By $K\subset \textrm{Int}(D)$ implying $\tilde{E}_i|_{D}^{-1}(\partial \cylinders^k_{\frac{\delta\kappa}{|\tau_0|}, i}(\bar{e}_i))\cap \partial D=\emptyset$, ${\bf S}_k$-symmetry as well as diagonal surjectivity of $E_{i}|_{\textrm{Diag}}$ in \eqref{ends oval surjectivity}, we know that 
\begin{equation}
  \tilde{E}_i|_{D}^{-1}(\partial\cylinders^k_{\frac{\delta\kappa}{|\tau_0|}, i}(\bar{e}_i))\cap \textrm{Diag}\not=\emptyset.
\end{equation}
By ${\bf{S}}_k$-equivariant transversality property 
\begin{equation}
  \tilde{E}_i|_{D}^{-1}(\partial\cylinders^k_{\frac{\delta\kappa}{|\tau_0|}, i}(\bar{e}_i))
\end{equation}
is disjoint union of (at least one)  connected components of ${\bf{S}}_k$-symmetric $(k-1)$-dimensional orientable  closed smooth hypersurfaces  in $D\subset \mathbb{R}^k_{+}$.

 By the Jordan-Brouwer separation theorem, each of them is an orientable manifold and dividing $\mathbb{R}^{k}$ into two  smooth connected regions of $\mathbb{R}^{k}$ with the same smooth submanifolds boundary and only one of the two regions enclosed by the connected component of inverse image is compact. According to the surjectivity along the diagonal and  continuity of $\tilde{E}_i$ at $\bar{a}$, we can find the innermost connected component in $\tilde{E}_i|_{D}^{-1}(\partial\cylinders^k_{\frac{\delta\kappa}{|\tau_0|}, i}(\bar{e}_i))$ with positive distance to the point $\bar{a}$.  We denote the compact smooth region  $\tilde{\Omega}^{k}_i$ enclosed by this component  and its boundary mainifold $\partial \tilde{\Omega}^{k}_i$ as the innermost component of $ \tilde{E}_i|_{D}^{-1}(\partial\cylinders^k_{\frac{\delta\kappa}{|\tau_0|}, i}(\bar{e}_i))$, and in particular by construction
 \begin{equation}\label{boundary to boundary map}
    \tilde{E}_i(\partial \tilde{\Omega}^{k}_i)\subset \partial\cylinders^k_{\frac{\delta\kappa}{|\tau_0|}, i}(\bar{e}_i).
 \end{equation}
 We denote the different points in $\partial \tilde{\Omega}^{k}_i\cap \textrm{Diag}$ by $\tilde{a}^{\pm}_{i}$.
 By ${\bf{S}}_k$-symmetry and that $\partial \tilde{\Omega}^{k}_i$ is an embedded submanifold in  $\mathbb{R}^k$, the interior of the  diagonal segment $\tilde{\Omega}^{k}_i\cap \textrm{Diag}$ connecting $\tilde{a}^{\pm}_{i}$ is in the interior of the compact smooth region $\tilde{\Omega}^{k}_i$. Namely
 \begin{equation}\label{segment inclusion}
     \bar{a}\in\textrm{Int}(\tilde{\Omega}^{k}_i\cap \textrm{Diag})\subset   \textrm{Int}(\tilde{\Omega}^{k}_i).
 \end{equation}

 We first illustrate that for all ellipsoidal parameter $a\in  \tilde{\Omega}^{k}_i$, the associated ellipsoidal flow $\mathcal{M}^{a}_{i}$ is $\kappa$-quadratic at time $\tau_0$. To this end, given $\kappa>0$ and $\tau_0<0$, we define $Q_i(\kappa, \tau_0)$ to be the set of all elements $a\in \mathbb{R}^k$ such that $\mathcal{M}^{a}_i$ is $\kappa$-quadratic at $\tau_0$.
\begin{claim}\label{saturationomega}
There are $\kappa>0$ small enough and $\tau_0\leq \tau_*$ negative enough  and $i$ large enough such that
\begin{equation}
     \tilde{\Omega}^{k}_i \cap Q_i(\kappa, \tau_0) = \tilde{\Omega}^{k}_i.
\end{equation}
In particular, for any $a\in \tilde{\Omega}^{k}_i$, the associated shifted ellipsoidal flow $\mathcal{M}^{a}_i$ with ellisoidal parameter $a$ is $\kappa$-quadratic at $\tau_0$.
\end{claim}
\begin{proof}[proof of Claim]
    By construction and \eqref{segment inclusion}, we know that for $i$ large enough, the point $\bar{a}$ in \eqref{sym oval ellipsoid limit}   satisfies $\bar{a}\in   \tilde{\Omega}^{k}_i\cap Q_i(\kappa, \tau_0) $. We only need to show $\tilde{\Omega}^{k}_i \cap Q_i(\kappa, \tau_0)$ in the sense of subspace topology is a open and closed subset of the connected (actually also path-connected) $k$-dimensional manifold $\tilde{\Omega}^{k}_i$ with boundary. Closedness follows from the definition of $\kappa$-quadraticity and the construction of $\tilde{\Omega}^{k}_i$ as the compact region enclosed by $\partial \tilde{\Omega}^{k}_i$ the innermost component of $ \tilde{E}_i|_{D}^{-1}(\partial\cylinders^k_{\frac{\delta\kappa}{|\tau_0|}, i}(\bar{e}_i))$. Then we show that any boundary point ${a_i}\in\tilde{\Omega}^{k}_i\cap Q_i(\kappa, \tau_0) $ is a interior point of $\tilde{\Omega}^{k}_i \cap Q_i(\kappa, \tau_0)$ in the subspace topology sense, which implies openness. Suppose towards a contradiction that the above $a_i$ is not interior point of  $\tilde{\Omega}^{k}_i \cap Q_i(\kappa, \tau_0)$, then the property that  $\mathcal{M}^{a_i}_{i}$ is $\kappa'$-quadratic at time $\tau_0$ must be saturated. i.e.,  at least one of the weak inequalities for its associated truncated profile function $v^{a_i}_{\cC}$
\begin{equation}\label{kappacond1}
    \left\| v^{a_i}_{\cC}(\mathbf{y}, \tau_{0})-\sqrt{2(n-k)}+\frac{\sqrt{2(n-k)}}{4|\tau_{0}|}(|\mathbf{y}|^2-2k)\right\|_{\mathcal{H}}\leq \frac{\kappa}{|\tau_{0}|},
\end{equation}
or
\begin{equation}\label{kappacond2}
        \sup_{\tau\in [2\tau_{0}, \tau_{0}]} |\tau|^{\frac{1}{50}} \|v^{a_i}(\cdot ,\tau)-\sqrt{2(n-k)}\|_{C^{4}(B(0, 2|\tau|^{{1}/{100}}))}\leq 1.
\end{equation}
must be an equality. By Theorem \ref{point_strong} (strong $\kappa$-quadraticity), it is strongly $\kappa$-quadratic from time $\tau_{0}$.
After passing to a subsequence the $\mathcal M_i^{a_i}$ converge to a limit oval $\mathcal M \in \mathscr{A}'$, which by \eqref{kappacond1} and \eqref{kappacond2} is $\kappa$-quadratic at time $\tau_0$. Thus, by Theorem ($\kappa$-quadraticity implies strong $\kappa$-quadraticity), the oval $\mathcal M$ is strongly $5 \kappa$-quadratic from time $\tau_0$. In particular, $\rho^{\mathcal M}(\tau)=|\tau|^{1 / 10}$ is an admissible graphical radius function for $\tau \leqslant \tau_0$, so inequality \eqref{kappacond2} is a strict inequality for $i$ large enough. Therefore, it must be the case that
\begin{align}\label{saturationcond}
\left\| v^{a_i}_{\cC}(\mathbf{y}, \tau_{0})-\sqrt{2(n-k)}+\frac{\sqrt{2(n-k)}}{4|\tau_{0}|}(|\mathbf{y}|^2-2k)\right\|_{\mathcal{H}}= \frac{\kappa}{|\tau_{0}|}.
\end{align}
On the other hand, by the centering condition we have
\begin{equation}
\fp_{+}(v_{\cC}^{a_i}(\tau_0)-\sqrt{2(n-k)})=0,
\end{equation}
and  as in \cite[Lemma 2.7]{CDDHS} or \cite[Lemma 2.6]{CDZ_ovals} (quantitative Merle-Zaag type estimate) we have
\begin{equation}
\big\|\fp_{-}(v_{\cC}^{{a_i}}(\tau_{0}))\big\|_\cH \\ \leq \frac{\kappa}{100|\tau_0|},
\end{equation}
and the fact that $\mathscr{E}(\mathcal{M}^{a_i}_i) \in \cylinders^k_{\frac{\delta\kappa}{|\tau_0|}, i}(\bar{e}_i)$ with $\delta>0$ and transversal smooth approximation in \eqref{trans-smooth approximation} with $\epsilon>0$  small enough  yields
\begin{align}
\left\|\mathfrak{p}_0\left(v_C^{a_i}\left(\tau_0\right)\right)-\frac{\sqrt{2(n-k)}(|\mathbf{y}|^2-2k)}{4|\tau_{0}|}\right\|_{\mathcal{H}} \leqslant \frac{6 \kappa}{10\left|\tau_0\right|}.
\end{align}
Adding these estimates implies that
\begin{equation}\label{still_in_int}
\left\|v_{\cC}^{a_i}(\tau_{0})-\sqrt{2(n-k)}+\frac{\sqrt{2(n-k)}(|\mathbf{y}|^2-2k)}{4|\tau_{0}|}\right\|_\cH \leq \frac{4\kappa}{5|\tau_0|}.
\end{equation}
This contradicts \eqref{saturationcond} and implies openness. Therefore this nonempty open and closed subset of $ \tilde{\Omega}^{k}_i$ satisfies
\begin{equation}
     \tilde{\Omega}^{k}_i \cap Q_i(\kappa, \tau_0) = \tilde{\Omega}^{k}_i.
\end{equation}
\end{proof}

By Claim \ref{saturationomega}, for all ellipsoidal parameter in the segment $\tilde{\Omega}^{k}_i\cap \textrm{Diag}$ the associated shifted $O(k)\times O(n+1-k)$-symmetric  ellipsoidal flow is $\kappa$-quadratic at $\tau_0$, so we can apply Claim \ref {Rescaling monotonicity for ellipsoidal flow} ({rescaling monotonicity for ellipsoidal flow}). We also recall that in our oval construction the ellipsoidal parameter $a$ and rescaling parameter satisfy 
\begin{equation}
    |a|_1=a_1+\dots +a_{k}=e^{\frac{\gamma}{2}}\quad a\in \mathbb{R}^k_+
\end{equation}
Then we consider the  $O(k)\times O(n+1-k)$-symmetric ellipsoidal flows spectral values. By  Claim \ref {Rescaling monotonicity for ellipsoidal flow}, the single variable smooth function 
\begin{equation}
f_i(s)=\sum_{j=1}^{k}\tilde{E}^{j}_i(\frac{s}{k}, \dots, \frac{s}{k})
\end{equation}
 has positive derivative at every point for $s\in [|\tilde{a}^{-}_{i}|_1, |\tilde{a}^{+}_{i}|_1]$, where $\tilde{E}^{j}_i$ is the $j$-th component of map $\tilde{E}_i$ and the two diagonal points   $\tilde{a}^{\pm}_{i}\in \partial \tilde{\Omega}^{k}_i\cap \textrm{Diag}$ satisfy $|\tilde{a}^{-}_{i}|_1< |\bar{a}_{i}|_1< |\tilde{a}^{+}_{i}|_1$.  Therefore, $f_i$ strictly decreases monotonically in the interval $[|\tilde{a}^{-}_{i}|_1, |\tilde{a}^{+}_{i}|_1]$. In particular,   $i$ large enough the image points
\begin{equation}\label{to barycenter}
\tilde{E}_i({\tilde{a}^{\pm}_{i}})=\bar\mu_{i,\mp} = (\bar{e}_{i, \mp}/k,\dots, \bar{e}_{i, \mp}/k)\in \partial\cylinders^k_{\frac{\delta\kappa}{|\tau_0|}, i}(\bar{e}_i)\cap \textrm{Diag};\quad \bar{e}_{i, +}>\bar{e}_{i, -}
\end{equation}
are two different points in $\partial\cylinders^k_{\frac{\delta\kappa}{|\tau_0|}, i}(\bar{e}_i)\cap \textrm{Diag}$.

\textbf{Inductive degree argument}

Our next aim is to show
\begin{claim}\label{kboundary surjective} 
We have
    \begin{equation}
       \tilde{E}_i( \partial \tilde{\Omega}^{k}_i)=\partial\cylinders^k_{\frac{\delta\kappa}{|\tau_0|}, i}(\bar{e}_i)
    \end{equation}
\end{claim}
Before the proof, we need to recall the notation convention discussion before Claim \ref{claim_boundedness}, where we have designed our smoothed simplicial cylinder $\cylinders^k_{\frac{\delta\kappa}{|\tau_0|}, i}(\bar{e}_i)$, which is smooth, 
isotopically approximates in a $\eta_i$-distance close way to the simplicial cylinder $\cylinders_{\frac{\delta\kappa}{|\tau_0|}, i}(\bar{e}_i)$ with corners and preserves ${\bf{S}}_k$-symmetry.
\begin{proof}[proof of Claim \ref{kboundary surjective}] 
 We start the proof by some examples to 
illustrate the ideas when $k=1, 2, 3$ and then we inductively  prove the claim. 

If $k=1$, this corresponds to the $\mathbb{Z}_2\times O(n)$-symmetric ancient oval case and we are done by \eqref{to barycenter} and intermediate value theorem. If $k=2$,  this corresponds to the bubble-sheet oval case and it again follows from \eqref{to barycenter} and intermediate value theorem as well as ${\bf{S}}_2=\mathbb{Z}^2$-symmetry (see Figure 1 below).
\begin{figure}[ht]
    \centering
\includegraphics[width=6.3cm,height=3.3cm]{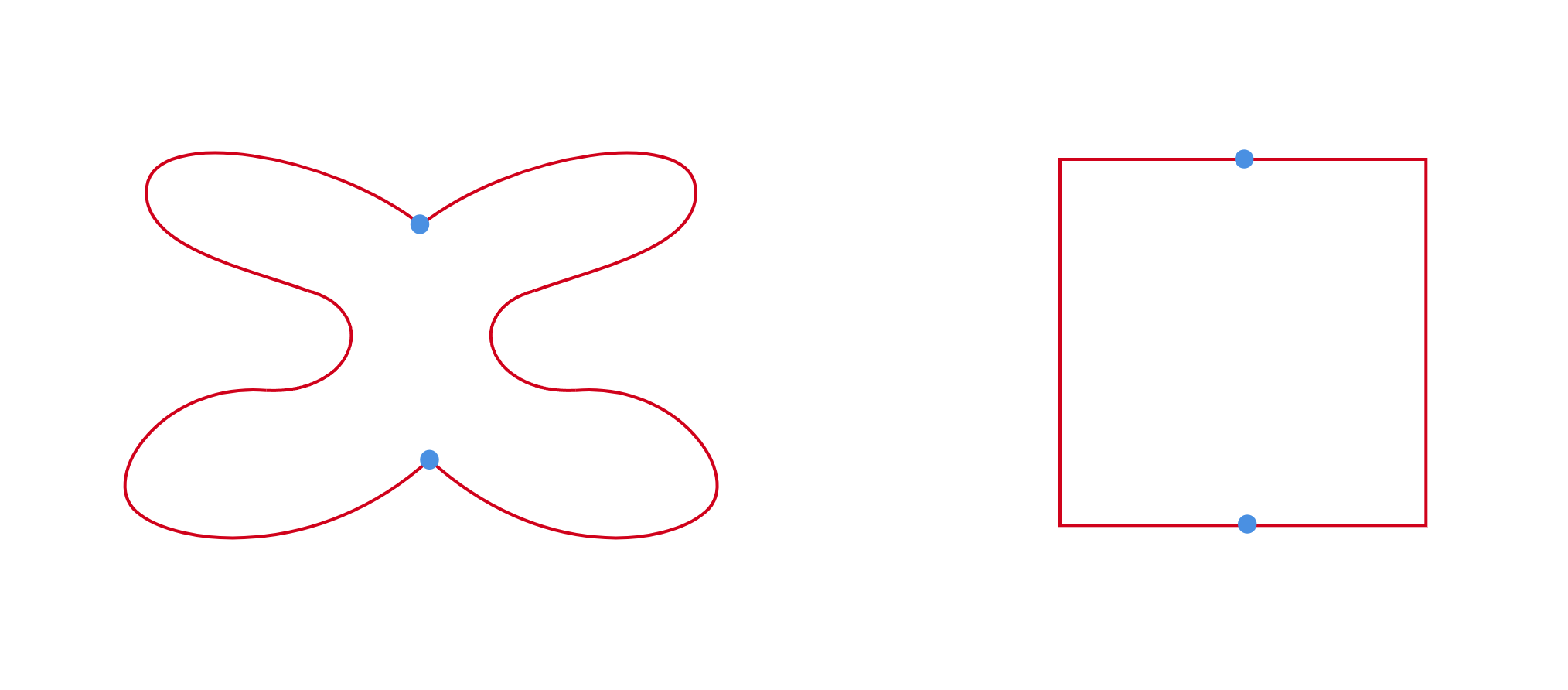} 
    \label{fig1}
    \caption{$k=2$ case, left is mapped to right.}
\end{figure}

For $k=3$ we illustrate ideas in Figure 2 below and we define the two dimensional planes \begin{equation}     \Sigma_{j}=\{\mu\in \mathbb{R}^3: \textrm{all coordinates are same except the } j\textrm{-th component} \}. \end{equation} 
Then the one dimensional smoothed rectangle  $\Gamma^{i}_{j}$ for $j=1, 2, 3$ can be decomposed into two curves only intersecting at the two barycenter points on the simplex $\simplex^{2}_{\frac{\delta\kappa}{|\tau_0|}, i}$.
\begin{equation}
    \Gamma^{i}_{j}=\Sigma^{1}_{j}\cap \partial\cylinders^3_{\frac{\delta\kappa}{|\tau_0|}, i}= \Gamma^{i, 1}_{j}\cup \Gamma^{i, 2}_{j} \quad j=1,\dots 3
\end{equation}
and
where $\Gamma^{i, 1}_{j}$ passes the  barycenter and $j$-th vertex of  $\simplex^{\pm, 2}_{\frac{\delta\kappa}{|\tau_0|}, i}$ and  $\Gamma^{i, 2}_{j}$ passes the the  barycenter and $j$-th $1$-dimensional edge (note the $j$-th vertex is not on it) of $\simplex^{\pm, 2}_{\frac{\delta\kappa}{|\tau_0|}, i}$   for $1\leq j\leq 3$. By \eqref{boundary to boundary map}, \eqref{to barycenter} and applying intermediate value theorem for map $ \tilde{E}_i|_{\partial \tilde \Omega_{i}^3}:\partial \tilde \Omega_{i}^3\to \partial\cylinders^3_{\frac{\delta\kappa}{|\tau_0|}, i}$ on $ \Sigma^{1}_{j}\cap \partial \tilde \Omega_{i}^3$ as well as using ${\bf{S}}_3$-symmetry, we have that it is at least  surjective to one of $\bigcup_{j=1}^{3}\Gamma^{i, 1}_{j}$, $\bigcup_{j=1}^{3}\Gamma^{i, 2}_{j}$. Suppose it is surjective to one of $\bigcup_{j=1}^{3}\Gamma^{i, 1}_{j}$. Then again by smooth approximation and the ${\bf{S}}_3$-symmetry, $ \tilde{E}_i$ restricted on a closed curve diffemorphic to $S^1$  on $\partial \tilde \Omega_{i}^3$, which smoothly approximates piece of closed curve ${(\cup_{j=1}^{2}\Sigma^{1}_{j})\cap \partial \tilde \Omega_{i}^3}$ has nonvanishing winding number\footnote{During the proof, we can compose a homeomorphic map from a domain with boundary  on simplicial cylinder to make the target of the map is Euclidean space and then we can apply the classical theory of degree winding number for continuous map from smooth manifolds with boundary to Euclidean space, see \cite{outerelo2009mapping, hirsch2012differential, guillemin2010differential}.} with respect some interior point on $\Gamma^{i, 1}_{3}$. Therefore, again by theory of degree and winding number and taking limit in the smooth curve approximation and again ${\bf{S}}_3$-symmetry, we know that  $  \tilde{E}_i|_{\partial \tilde \Omega_{i}^3}:\partial \tilde \Omega_{i}^3\to \partial\cylinders^3_{\frac{\delta\kappa}{|\tau_0|}, i}$ is surjective. In the second case if it is surjective to one of $\bigcup_{j=1}^{3}\Gamma^{i, 2}_{j}$, the argument is the same and the required surjectivity still holds. 
\begin{figure}[ht] \label{fig23}
    \centering
\includegraphics[width=4.2cm,height=3cm]{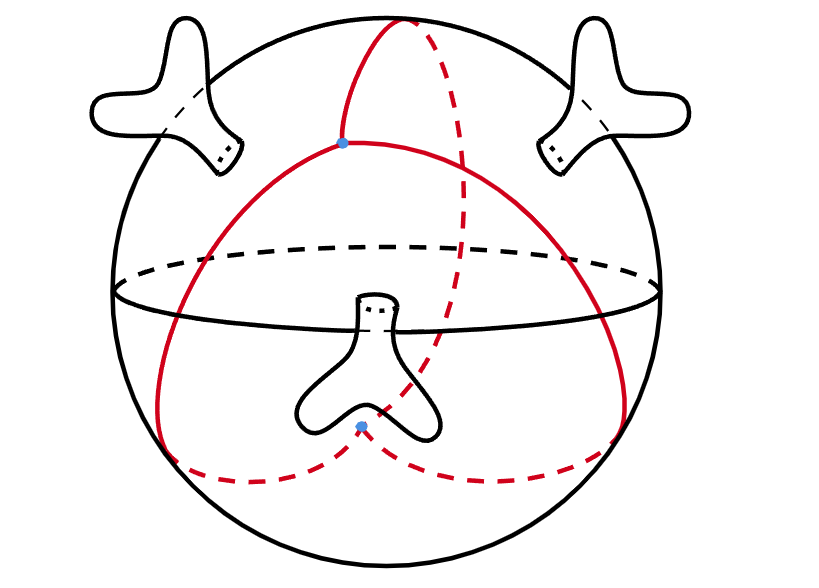}\quad\quad
\includegraphics[width=3.2cm,height=3cm]{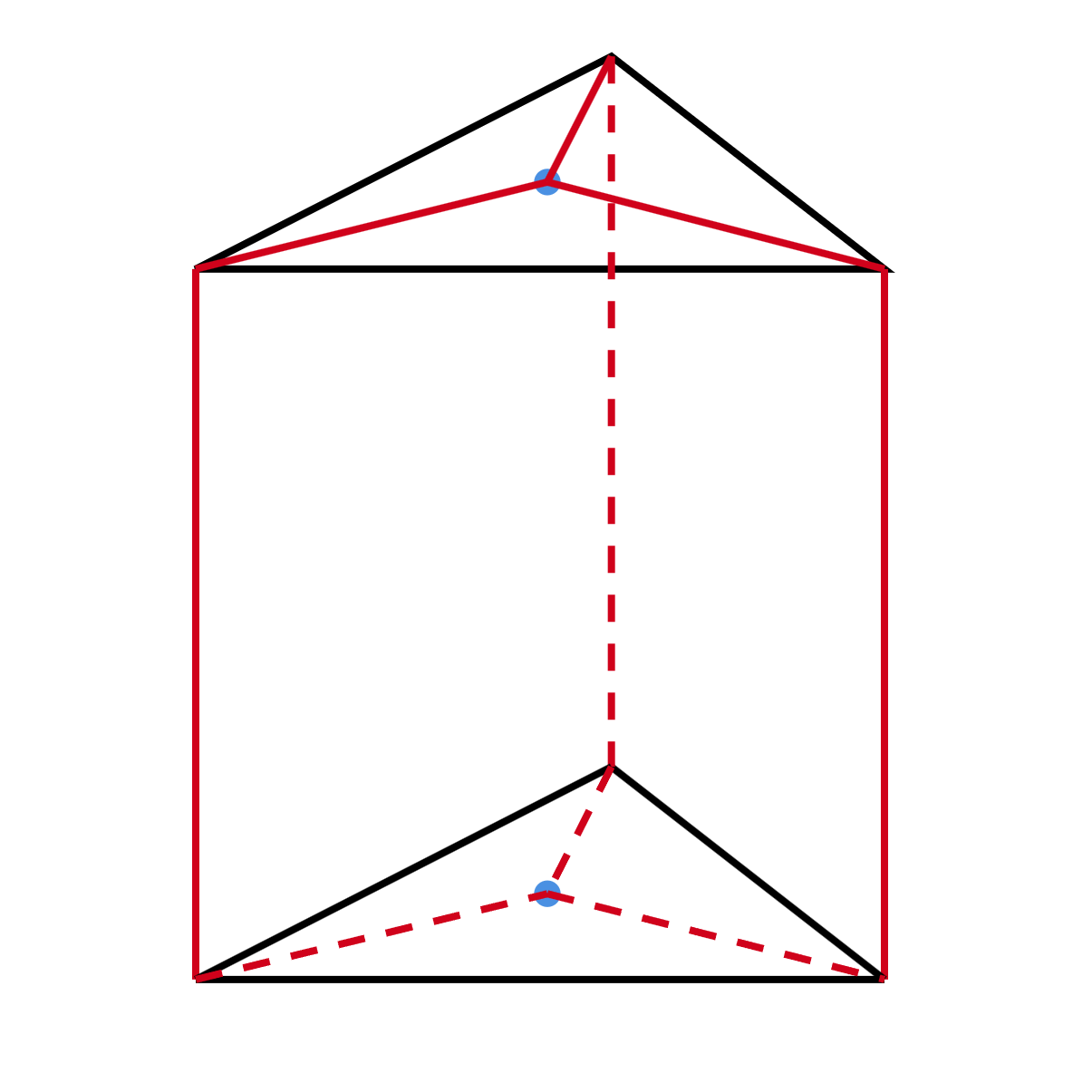}
    \caption{$k=3$ case: left is mapped to the right.}
\end{figure}

Now, we adapt the inductive argument to prove the surjectivity claim holds for all $k$ without the need of discussing different cases as when $k=3$. Suppose that the inductive assumption that
   \begin{equation}\label{inductive assumption}
        \tilde{E}_i|_{\partial \tilde \Omega_{i}^m}:\partial \tilde \Omega_{i}^m\to \partial\cylinders^m_{\frac{\delta\kappa}{|\tau_0|}, i}\,\,\textrm{is surjective for}\,  m\leq k-1
    \end{equation}
holds, we aim to show that the claim holds when $m=k$. We note that  $\tilde \Omega_{i}^{k}$ and the smoothly $\eta_i$-approximated $\cylinders^{k}_{\frac{\delta\kappa}{|\tau_0|}, i}$ with induced vertices and skeleton structure as well as the map 
$\tilde{E}_i|_{\partial \tilde \Omega_{i}^k}$ are both  ${\bf{S}}_{k}$-symmetric. We also define the two dimensional planes
\begin{equation}
    \Sigma_{j}=\{\mu\in \mathbb{R}^3: \textrm{all coordinates are same except the } j\textrm{-th component} \}.
\end{equation}
for $j=1,\dots k,$
\begin{equation}
    \Gamma^{i}_{j}=\Sigma^{1}_{j}\cap \partial\cylinders^k_{\frac{\delta\kappa}{|\tau_0|}, i}= \Gamma^{i, 1}_{j}\cup \Gamma^{i, 2}_{j} \quad j=1,\dots k
\end{equation}
where $\Gamma^{i, 1}_{j}$ passes the  barycenter and $j$-th vertex of  top/bottom simplicial faces of the simplicial cylinder $\simplex^{\pm, k-1}_{\frac{\delta\kappa}{|\tau_0|}, i}$ defined in \eqref{Top}-\eqref{Bottom} and  $\Gamma^{i, 2}_{j}$ passes the  barycenter and $j$-th $k-2$-dimensional surface (note the $j$-th vertex is not on it) of $\simplex^{\pm, k-1}_{\frac{\delta\kappa}{|\tau_0|}, i}$   for $1\leq j\leq k$.
Let us denote ${\bf{v}}^{j, \pm}_{i}\in \Gamma^{i, 1}_{j}$ be the $k$ vertices of $k-1$-dimensional $\simplex^{\pm, k-1}_{\frac{\delta\kappa}{|\tau_0|}, i} $ respectively  for $j=1, \dots, k$. We recall by \eqref{to barycenter} $\bar\mu_{i,\pm} = (\bar{e}_{i, \pm}/k,\dots, \bar{e}_{i, \pm}/k)=  \tilde{E}_i|_{\partial \tilde \Omega_{i}^k}(\tilde{a}^{\pm}_{i})$ are the barycenters of $\simplex^{\pm, k-1}_{\frac{\delta\kappa}{|\tau_0|}, i} $.  Connecting the first $k-1$ vertices ${\bf{v}}^{j, +}_{i}$ (${\bf{v}}^{j, -}_{i}$) for $j=1, \dots, k-1$ and barycenter $\bar\mu_{i,+}$ (barycenter $\bar\mu_{i,-}$) respectively, we obtain a $k-1$-dimensional simplex denoted by $\tetra({\bf{v}}^{1, +}_{i},\dots, {\bf{v}}^{k-1, +}_{i}, \bar\mu_{i,+})\subset \simplex^{+, k-1}_{\frac{\delta\kappa}{|\tau_0|}, i}$ (and $\tetra({\bf{v}}^{1, -}_{i},\dots, {\bf{v}}^{k-1, -}_{i}, \bar\mu_{i,-})\subset \simplex^{-, k-1}_{\frac{\delta\kappa}{|\tau_0|}, i}$), and an open simplex $\simplex({\bf{v}}^{1, +}_{i},\dots, {\bf{v}}^{k-1, +}_{i})$  generated by ${\bf{v}}^{1, +}_{i},\dots, {\bf{v}}^{k-1, +}_{i}$   (and the open simplex $\simplex({\bf{v}}^{1, -}_{i},\dots, {\bf{v}}^{k-1, -}_{i})$ generated by ${\bf{v}}^{1, -}_{i},\dots, {\bf{v}}^{k-1, -}_{i}$ ) are their $k-2$ dimensional faces respectively.   We note that $\simplex^{+, k-1}_{\frac{\delta\kappa}{|\tau_0|}, i}$  is the interior disjoint  union of $k$ number of simplicial images of ${\bf{S}}_k$ action on $\tetra({\bf{v}}^{1, +}_{i},\dots, {\bf{v}}^{k-1, +}_{i}, \bar\mu_{i,+})$ (and $\simplex^{-, k-1}_{\frac{\delta\kappa}{|\tau_0|}, i}$  is the interior disjoint  union of $k$ number of simplicial images of ${\bf{S}}_k$ action on  $\tetra({\bf{v}}^{1, -}_{i},\dots, {\bf{v}}^{k-1, -}_{i}, \bar\mu_{i,-})$). Let us also define
\begin{equation}
    \partial^{\circ}{\tetra}({\bf{v}}^{1, +}_{i}\!,\dots\!, {\bf{v}}^{k-1, +}_{i}, \bar\mu_{i,+})=\partial[{\tetra}({\bf{v}}^{1, +}_{i}\!,\dots\!, {\bf{v}}^{k-1, +}_{i}, \bar\mu_{i,+})]\setminus\simplex({\bf{v}}^{1, +}_{i}\!,\dots\!, {\bf{v}}^{k-1, +}_{i})
\end{equation}
and
\begin{equation}
    \partial^{\circ}{\tetra}({\bf{v}}^{1, -}_{i}\!,\dots\!, {\bf{v}}^{k-1, -}_{i}, \bar\mu_{i,-})=\partial[{\tetra}({\bf{v}}^{1, -}_{i}\!,\dots\!, {\bf{v}}^{k-1, -}_{i}, \bar\mu_{i,-})]\setminus\simplex({\bf{v}}^{1, -}_{i}\!,\dots\!, {\bf{v}}^{k-1, -}_{i})
\end{equation}
There are nature $\bar\mu_{i,+}-\frac{1}{k-1}\sum_{j=1}^{k-1}{\bf{v}}^{j, +}_{i}$-direction and  $\bar\mu_{i,-}-\frac{1}{k-1}\sum_{j=1}^{k-1}{\bf{v}}^{j, -}_{i}$-direction orthogonal projection homeomorphisms 
\begin{equation}
    P^{+}_{k-1, i}:  \partial^{\circ}{\tetra}({\bf{v}}^{1, +}_{i}\!,\dots\!, {\bf{v}}^{k-1, +}_{i}, \bar\mu_{i,+})\to \textrm{cl}[\simplex({\bf{v}}^{1, +}_{i}\!,\dots\!, {\bf{v}}^{k-1, -}_{i})]
\end{equation}
and
\begin{equation}
    P^{-}_{k-1, i}:  \partial^{\circ}{\tetra}({\bf{v}}^{1, -}_{i}\!,\dots\!, {\bf{v}}^{k-1, -}_{i}, \bar\mu_{i,-})\to \textrm{cl}[\simplex({\bf{v}}^{1, -}_{i}\!,\dots\!, {\bf{v}}^{k-1, -}_{i})]
\end{equation}
where $\textrm{cl}$ means closure operation. In addition, they preserve the discrete ${\bf S}_{k-1}$-symmetry and map barycenter to barycenter. Let $\mathfrak{C}^{+}_{k-1, i}$ be the $k-1$-dimensional cylinder  obtained by crossing $ \partial^{\circ}{\tetra}({\bf{v}}^{1, -}_{i}\!,\dots\!, {\bf{v}}^{k-1, -}_{i}, \bar\mu_{i,-})$ along the diagonal segment inside $\cylinders^{k}_{\frac{\delta\kappa}{|\tau_0|}, i}$,  and let $\mathfrak{c}^{+}_{k-1, i}$ be the $k-1$-dimensional  ${\bf{S}}_{k-1}$-symmetric simplicial cylinder  obtained by crossing $\textrm{cl}[\simplex({\bf{v}}^{1, -}_{i}\!,\dots\!, {\bf{v}}^{k-1, -}_{i})]$ along the diagonal segment inside $\cylinders^{k}_{\frac{\delta\kappa}{|\tau_0|}, i}$.
Similarly we can define $\mathfrak{C}^{-}_{k-1, i}$ and ${\bf{S}}_{k-1}$-symmetric simplicial cylinder $\mathfrak{c}^{-}_{k-1, i}$. Then, we have naturally induced projection homeomorphisms
\begin{equation}
     \mathbf{P}^{+}_{k-1, i}:\mathfrak{C}^{+}_{k-1, i}\to \mathfrak{c}^{+}_{k-1, i} \quad  \mathbf{P}^{-}_{k-1, i}:\mathfrak{C}^{-}_{k-1, i}\to \mathfrak{c}^{-}_{k-1, i}. 
\end{equation}
\begin{figure}[ht] \label{fig24}
    \centering
\includegraphics[width=3cm,height=3cm]{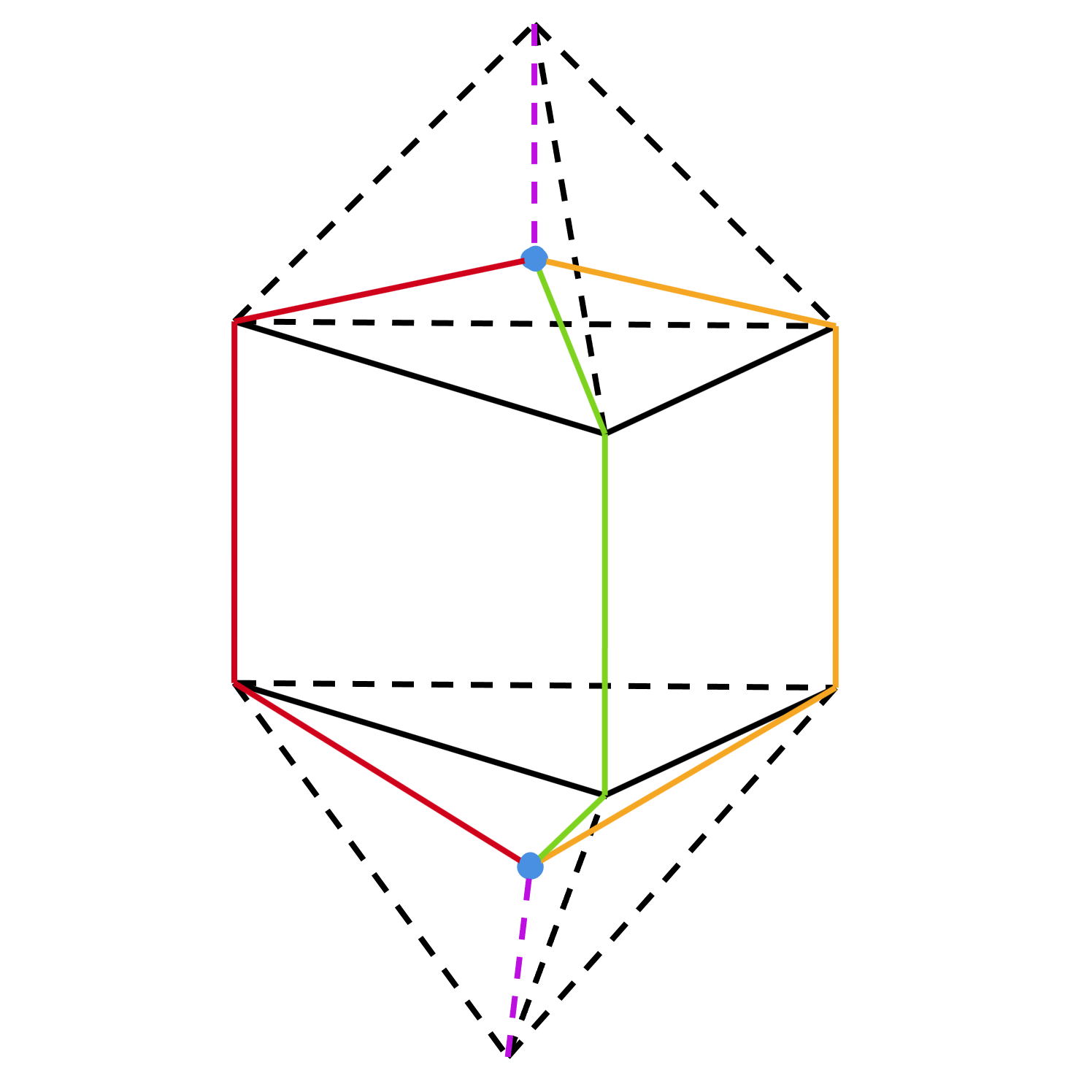}\quad\quad\includegraphics[width=3.2cm,height=2.5cm]{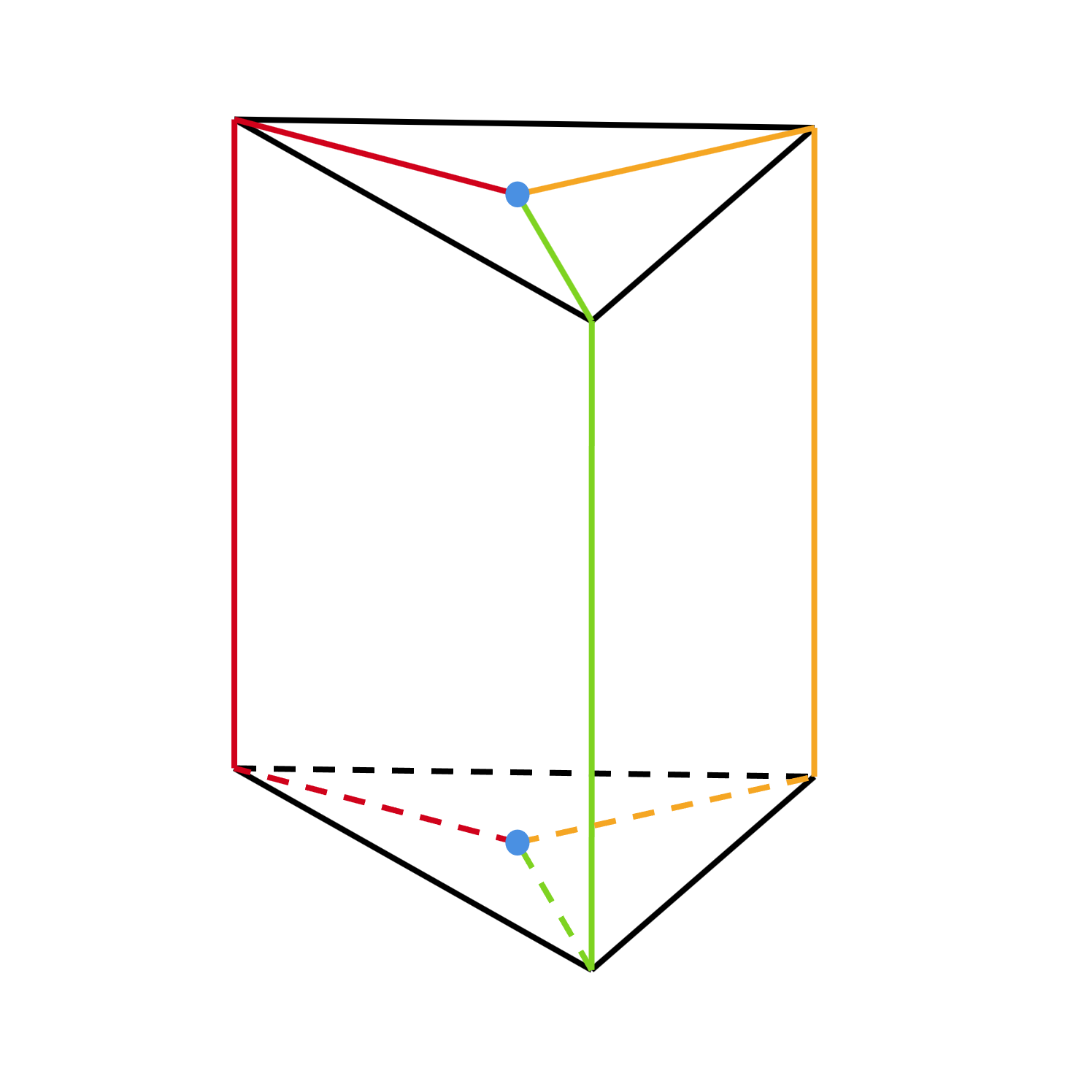}
    \caption{$k=4$ case: left is projected to the right.}
\end{figure}

We note that all of our constructions and operations above preserve the discrete permutation symmetry. Now, we can decompose the $(k-1)$-dimensional manifold $\partial \tilde{\Omega}^{k}_{i}$ into $k$ equal pieces according to the ${\bf{S}}_k$-symmetry. Let $\check{\Omega}_{i}^{k-1}\subset \partial \tilde{\Omega}^{k}_{i}$ be one of the piece whose boundary $\partial  \check{\Omega}_{i}^{k-1}$ has identical symmetry as $\mathfrak{c}^{\pm}_{k-1, i}$
(meaning they are obtained by the same way of intersecting with some equal coordinates subspaces). Moreover $\check{\Omega}_{i}^{k-1}\subset \partial \tilde{\Omega}^{k}_{i}$ is smooth manifolds with corners\footnote{One can also apply differential topology techniques for manifolds with corners, c.f. \cite{nielsen1981transversality, joyce2009manifolds, margalef1992differential}.}. Together with a 
${\bf S}_k$ equivariant smooth approximation and limiting argument, we can apply the inductive assumption \eqref{inductive assumption} when $m=k-1$ or can adapt similar argument as in case $m=k-1$  to imply the following ${\bf{S}}_{k-1}$-symmetric composed continuous maps with ${\bf{S}}_{k-1}$-symmetric domains an images
\begin{equation}
    \mathbf{P}^{\pm}_{k-1, i}\circ  \tilde{E}_i|_{\partial \check{\Omega}_{i}^{k-1}}:\partial  \check{\Omega}_{i}^{k-1}\to \mathfrak{c}^{\pm}_{k-1, i} 
\end{equation}
are surjective. Therefore 
\begin{equation}
 \tilde{E}_i|_{\partial \check \Omega_{i}^{k-1}}:\partial \tilde \Omega_{i}^{k-1}\to \mathfrak{C}^{\pm}_{k-1, i} 
\end{equation}
are surjective.

Moreover, by construction $\mathfrak{C}^{\pm}_{k-1, i}$ is a $k-2$-dimensional topological sphere dividing the $k-1$-dimensional  topological sphere $\partial\cylinders^{k}_{\frac{\delta\kappa}{|\tau_0|}, i}$ into two parts. By the ${\bf{S}}_k$-symmetry of the domain, range and map, we know that the winding number of $\tilde{E}_i|_{\partial  \check\Omega_{i}^{k-1}}$ with respect to points in at least one part of them is nonvanishing. By theory of degree and 
winding number and again ${\bf{S}}_k$-symmetry as well as a smooth approximation argument as  before,  the map
\begin{equation}
        \tilde{E}_i|_{\partial \tilde \Omega_{i}^k}:\partial \tilde \Omega_{i}^k\to \partial\cylinders^{k}_{\frac{\delta\kappa}{|\tau_0|}, i}
    \end{equation}
is surjective. This completes the proof of Claim \ref{kboundary surjective}.
\end{proof}

By \eqref{to barycenter}, Claim \ref{kboundary surjective} and  recalling \eqref{boundary to boundary map} and using theory of degree and 
winding number, we have that the continuous map  
    \begin{equation}
        \tilde{E}_i:\tilde \Omega_{i}^k\to \textrm{Int}(\mathbb{R}^{k}_{+}) 
    \end{equation}
surjectively covers $\cylinders^{k}_{\frac{\delta\kappa}{|\tau_0|}, i}$ for $\kappa>0, \delta>0$ small enough, $\tau\leq \tau_*$ negative enough and $i<\infty$ large enough. Sending $\epsilon\to 0$ in \eqref{trans-smooth approximation},  we obtain that
\begin{equation}
 \cylinders_{\frac{\delta\kappa}{|\tau_0|}, i} \subset   E_i(\Omega^k_i)
\end{equation}
holds for $\kappa>0$ small enough and $\tau\leq \tau_*$ negative enough and $i$ large enough, $\delta>0$ small enough.  Namely, for any given $\mu \in \cylinders^{k}_{\frac{\delta \kappa}{|\tau_0|}}$  for $i$ large enough we can find a sequence of ellipsoidal flows $\mathcal{M}^{a_{i}}_{i}$ with ellipsoidal parameters ${a}_{i}$ such that they are $\kappa$-quadratic at $\tau_0$ and
and their profile functions $v^{a_{i}}$ satisfy
\begin{equation}
\mathfrak{p}_{+}(v^{{a}_{i}}_{\cC, i})(\tau_{0})=0
    \end{equation}
By Claim \ref{claim_boundedness} (compactness) after passing through subsequential limit as $i\to \infty$,  we can find a time-shifted ancient oval  
\begin{equation}
\lim_{i\to \infty}\mathcal{M}^{{a}_{i}}_{i}=\mathcal{M}^{\infty}\in\mathscr{A}'
\end{equation}
that is $\kappa$-quadratic at $\tau_0$  and whose profile function $v^{{\infty}}$ satisfies
\begin{equation}
\mathfrak{p}_{+}(v^{{\infty}}_{\cC, i})(\tau_{0})=0
    \end{equation}
    and
\begin{equation}
    \mathscr{E}(\mathcal{M}^{\infty})(\tau_0)=\mu.
\end{equation}

This concludes the proof of Theorem \ref{prescribed_eccentricity_restated}.

\end{proof}

Before we state the following existence results for the moduli space, we recall the definition of $\mathcal{E}$ in \cite{CDZ_ovals}.
Given $\kappa>0$ and $\tau_0>-\infty, \overline{\mathcal{A}}_\kappa\left(\tau_0\right)$ is the set of all $k$-ovals which are $\kappa$-quadratic at time $\tau_0$, and the corresponding truncated profile function $v_{\cC}$ satisfies orthogonality conditions by Proposition \ref{prop_orthogonality}.
\begin{align*}
\begin{gathered}
\mathfrak{p}_{+}\left(v_{\cC}\left(\mathbf{y}, \tau_0\right)-\sqrt{2(n-k)}\right)=0, \\
\left\langle v_{\cC}^{\mathcal{M}}\left(\mathbf{y}, \tau_0\right), y_i y_j\right\rangle_{\mathcal{H}}=0, \quad 1 \leq i<j \leq k
\end{gathered}
\end{align*}
and
$$
\left\langle v_{\cC}^{\mathcal{M}}\left(\mathbf{y}, \tau_0\right)+\frac{\sqrt{2(n-k)}\left(|\mathbf{y}|^2-2 k\right)}{4\left|\tau_0\right|},| \mathbf{y}|^2-2 k\right\rangle_{\mathcal{H}}=0.
$$
Then we define the spectral ratio map $\mathcal{E}=\mathcal{E}\left(\tau_0\right): \bar{\mathcal{A}}_\kappa\left(\tau_0\right) \rightarrow \Delta^{k-1}$ at $\tau_0$ :
$$
\mathcal{E}(\mathcal{M})=\mathcal{E}\left(\tau_0\right)(\mathcal{M})=\left(\frac{\left\langle v_{\cC}^{\mathcal{M}}\left(\tau_0\right), y_1^2-2\right\rangle_{\mathcal{H}}}{\left\langle v_{\cC}^{\mathcal{M}}\left(\tau_0\right),|\mathbf{y}|^2-2 k\right\rangle_{\mathcal{H}}}, \ldots, \frac{\left\langle v_{\cC}^{\mathcal{M}}\left(\tau_0\right), y_k^2-2\right\rangle_{\mathcal{H}}}{\left\langle v_{\cC}^{\mathcal{M}}\left(\tau_0\right),|\mathbf{y}|^2-2 k\right\rangle_{\mathcal{H}}}\right)
$$

\begin{corollary}[existence with prescribed spectral value]\label{prescribed_eccentricity_restated'} 
There exist small constants $\delta>0$,  $\kappa>0$, $\kappa'>0$  and very negative constant $\tau_\ast>-\infty$ with the following significance.  For every $\tau_{0}\leq \tau_{*}$ and every 
\begin{equation}\label{targetmu'}
    \mu'=(\mu_1', \dots, \mu_k')\in \simplex_{\frac{\delta\kappa}{c}}(1)
\end{equation}
where 
\begin{equation}
\sum_{i=1}^{k}\mu_i'=1,\quad c = \frac{\sqrt{2(n-k)}\||\mathbf{y}|^2-2k\|^2_{\mathcal{H}}}{4}
\end{equation}
and $\simplex_{\frac{\delta\kappa}{c}}(1)$ denotes the $k-1$-dimensional simplex centered at the point $(\frac{1}{k}, \dots, \frac{1}{k})$ with size $\frac{\delta\kappa}{c}$, there exists an ancient oval  $\mathcal{M}\in\mathcal{A}^{\circ}$ which is $\kappa'$-quadratic at $\tau_0$ after suitable time shifting by parameter $\beta(\mathcal{M})$ and parabolic rescaling by parameter $\gamma(\mathcal{M})$, the transformed flow $\mathcal{M}^{\beta(\mathcal{M}), \gamma(\mathcal{M})}$ is $\kappa$-quadratic at $\tau_0$  and satisfies
\begin{equation}
    \mathcal{E}(\mathcal{M}^{\beta(\mathcal{M}), \gamma(\mathcal{M})})=   \mathcal{E}(\mathcal{M}^{\beta(\mathcal{M}), \gamma(\mathcal{M})})(\tau_0)=\mu'.
\end{equation}
Moreover, such transformation parameter $(\beta(\mathcal{M}), \gamma(\mathcal{M}))$ depending on the $\mathcal{M}$ continuously. 
\end{corollary}

\begin{proof}
We note that for any $\mu'\in \simplex_{\frac{\delta\kappa}{c}}(1)$, we can naturally lift to an element $\mu$ in $\cylinders_{\frac{\delta\kappa}{|\tau_0|}}(\bar{e})$ through the map $p: \cylinders_{\frac{\delta\kappa}{|\tau_0|}}(\bar{e}) \mapsto \simplex_{\frac{\delta\kappa}{c}}(1)$ defined by $p(\mu) = \frac{1}{|\mu|}\mu$. By Theorem \ref{prescribed_eccentricity_restated}, there exists a time-shifted ancient oval  $\mathcal{M}'\in\mathscr{A}'$ that is $\kappa$-quadratic at $\tau_0$  and satisfies
\begin{equation}
    \mathscr{E}(\mathcal{M}')= \mathscr{E}(\mathcal{M}')(\tau_0)=\mu.
\end{equation} 
By the definition of $\mathscr{A}'$, we have 
\begin{align}
\left\langle v^{\mathcal M'}_{\cC}({\bf{y}} , \tau_{0})-\sqrt{2(n-k)}, 1\right\rangle_\cH = 0.
\end{align}
Moreover, by definition of $\bar e$, we also have
\begin{align}
\left\langle v^{\mathcal M'}_{\cC}({\bf{y}} , \tau_{0})+\frac{\sqrt{2(n-k)}\||\mathbf{y}|^2-2k\|^2_{\mathcal{H}}}{4|\tau_0|}, |\mathbf{y}|^2-2k\right\rangle_\cH = 0.
\end{align}
Thus, $\mathcal{M}' \in \mathcal{A}'$ by definition. This implies following diagram commutes
$$
\begin{tikzcd}
\mathscr{A}' \arrow[r, "\mathscr{E}"] \arrow[d, "\pi"'] 
& \cylinders_{\frac{\delta\kappa}{|\tau_0|}}(\bar{e}) \arrow[d, "p"] \\
\mathcal{A}' \arrow[r, "\mathcal{E}"'] 
& \simplex_{\frac{\delta\kappa}{c}}(1)
\end{tikzcd}
$$
Then by Proposition \ref{time-shift}, we may find a $\mathcal M\in \mathscr{A}^{\circ}$ and a unique $\beta(\mathcal M)$ depending on $\mathcal M$ continuously to satisfy $\mathcal M^{\beta(\mathcal M)} = \mathcal M'$. Moreover, we also have $\mathcal M\in \mathcal{A}^{\circ}$ and $\mathcal M^{\beta(\mathcal M),1} = \mathcal M'$.
\end{proof}

\section{Classification and moduli space} \label{sec-classification and moduli}
In this section, we give the proof of our main results in Theorem \ref{classification_theorem} and Theorem \ref{cor_moduli}. 

\begin{proof}[Proof of Theorem \ref{classification_theorem}]
Let $\mathcal{M}^1$ be any $k$-oval in $\mathbb{R}^{n+1}$ which always has $\mathbb{Z}_2^k\times O(n+1-k)$-symmetry by \cite[Corollary 1.4]{CDZ_ovals}. Then, we work in coordinates such that the tangent flow at $-\infty$ is given by $\mathbb{R}^k\times S^{n-k}(\sqrt{2(n-k)|t|})$ and such that the $\mathrm{O}(n+1-k)$-symmetry is in the $x_{k+1}\dots x_{n+1}$-plane centered at the origin. By  Proposition \ref{prop_orthogonality} (orthogonality), given any $\kappa'>0$ and $\tau_0\leq\tau_\ast(\mathcal{M}^1,\kappa')$ after a suitable space-time transformation we can assume that the truncated renormalized profile function $v_{\cC}^{\mathcal{M}^1}$ of $\mathcal{M}^1$  is $\kappa'$-quadratic at time $\tau_0$ and satisfies
\begin{equation}
\fp_+ \big(v_{\cC}^{\mathcal{M}^1}(\tau_0)-\sqrt{2(n-k)}\big)=0,
\end{equation}
\begin{equation}
\Big\langle v_{\cC}^{\mathcal{M}^1}(\tau_0) +\frac{\sqrt{2(n-k)}(|\mathbf{y}|^2-2k)}{4|\tau_0|},|\mathbf{y}|^2-2k\Big\rangle_\cH=0,
\end{equation}
\begin{equation}
   \Big\langle v_{\cC}^{\mathcal{M}^1}(\tau_0), y_{i}y_{j}\Big\rangle_\cH=0 \quad i\not=j.
\end{equation}
Let $\delta>0$, $\kappa>0$, $\tau_\ast>-\infty$ be constants such that Theorem    \ref{prescribed_eccentricity_restated} (existence with prescribed spectral value) and Theorem \ref{spectral uniqueness restated} (spectral uniqueness) apply. Possibly after decreasing $\tau_\ast$ and choosing $\kappa'\ll \delta$ we can arrange that
\begin{equation}
\mathscr{E}(\mathcal{M}^1)
\in \cylinders_{\frac{\delta\kappa}{|\tau_0|}}(\bar{e}).
\end{equation}
Thus, by Theorem \ref{prescribed_eccentricity_restated}  (existence with prescribed spectral value) there exists a $k$-oval $\mathcal{M}^2\in \mathscr{A}'$ such that it is $\kappa$-quadratic at time $\tau_0$ (up to transformation it also belongs to the class $\mathcal{A}^{k, \circ}$ by our construction of $\mathscr{A}{'}$) and it satisfies
\begin{equation}
 \mathscr{E}(\mathcal{M}^2)=\mathscr{E}(\mathcal{M}^1),
 \end{equation}
and in particular by time-shifted oval class $\mathscr{A}'$ definition in \eqref{shifted oval class} it also sastisfies
 \begin{equation}
\fp_+ \big(v_{\cC}^{\mathcal{M}^2}(\tau_0)-\sqrt{2(n-k)}\big)=0.
\end{equation}
Recall also that by construction of ancient ovals in  $\mathscr{A}'$ we have
\begin{equation}
\Big\langle v_{\cC}^{\mathcal{M}^2}(\tau_0),  y_{i}y_{j}\Big\rangle_\cH=0\quad i\not=j.
\end{equation}
Hence, we can apply Theorem \ref{spectral uniqueness restated} (spectral uniqueness) to conclude that the $k$-ovals $\mathcal{M}^1$ and $\mathcal{M}^2$ coincide. We have thus shown that any $k$-oval in $\mathbb{R}^{n+1}$ belongs, up to parabolic rescaling and space-time rigid motion, to the oval class  $\mathcal{A}^{k, \circ}$ constructed in \cite{DH_ovals}. This completes the proof of  Theorem \ref{classification_theorem} (classification of $k$-ovals).
\end{proof}

Then we discuss the proof of Theorem \ref{cor_moduli} (classification of ancient ovals and moduli space). We recall that $\mathcal{X}$ is the quotient space of ancient ovals whose tangent flow at $-\infty$ coincides with $\mathbb{R}^k\times S^{n-k}(\sqrt{2(n-k)|t|})$, modulo spacetime rigid motions and parabolic dilations. This space is endowed with the topology induced by local smooth convergence of solutions. 

\begin{proof}[Proof of Theorem \ref{cor_moduli}]
The classification of ancient ovals follows from Theorem \ref{classification_theorem} (classification of $k$-ovals) and Conjecture \ref{oval-koval} verified by Bamler-Lai in \cite{BL1,BL2}. Then we discuss the moduli space $\mathcal{X}$.
\begin{claim}\label{claim5.1}
We have the homeomorphism between Hausdorff spaces. 
\begin{equation}
     \mathcal{X}\cong \mathcal{A}^{k, \circ} / \mathbf{S}_k    
\end{equation}
\end{claim}

\begin{proof}[proof of Claim \ref{claim5.3}]
By   Conjecture \ref{oval-koval} and Theorem \ref{prescribed_eccentricity_restated} (existence with prescribed spectral value), we know that  all  the ancient ovals  with $\mathbb{R}^{k}\times S^{n-k}$ as tangent flow at $-\infty$ are inside $\mathcal{A}^{k, \circ}$.

Consider the canonical map $q: \mathcal{A}^{k, \circ}\to \mathcal{X}$ that sends any $\mathcal{M}\in\mathcal{A}^{k, \circ}$ to its equivalence class $[\mathcal{M}]\in\mathcal{X}$.  By Theorem \ref{classification_theorem} (classification of $k$-ovals) and the above, the map $q$ is well defined and surjective. Suppose now that $q(\mathcal{M}^1)=q(\mathcal{M}^2)$. Then, by definition of our equivalence relation, $\mathcal{M}^2$ is obtained from $\mathcal{M}^1$ by a space-time rigid motion and parabolic dilation. Since all elements at of the class $\mathcal{A}^{k, \circ}$ become extinct at the origin at time zero, there cannot be any nontrivial space-time translation, and thanks to the condition about the Huisken density at time $-1$ there cannot be any nontrivial parabolic dilation either. So $\mathcal{M}^2$ is obtained from $\mathcal{M}^1$ by a rotation $R\in \mathrm{SO}(n+1)$. Since the rotation fixes the tangent flow at $-\infty$ and the solutions are $SO(n+1-k)$-symmetric about last $n+1-k$ coordinates, the rotation  $R$ must be  rotating  the $\mathbb{R}^{k}$ factor.  To proceed, let us consider the spectral  matrix
\be
\mathcal{W}(\mathcal{M})=(\big\langle v^{\mathcal{M}^{\beta,\gamma}}_\cC(\tau_0),  \psi_{ij}\big\rangle_\cH )_{k\times k}
\ee
where $\psi_{\ell\ell}=y_\ell^2-2$  for $\ell=1,\dots, k$ and $\psi_{ij}=2 y_iy_j$ for $1\leq i\not=j\leq k$ and $\beta,\gamma$ are time-shift parameter and parabolic dilation parameter depending on $\mathcal{M}$ are chosen according to Proposition \ref{prop_orthogonality} such that $\mathcal{M}^{\beta,\gamma}$ is $\kappa$-quadratic  at $\tau_0$ and its profile function satisfies the following orthogonality conditions,
\begin{equation}\label{bg ON+}
    \mathfrak{p}_{+}\big(v^{\beta, \gamma}_{\cC}({\bf{y}}, \tau_{0})-\sqrt{2(n-k)}\big)=0,
\end{equation}
\begin{equation}\label{bg ON y2-2k}
         \Big\langle v^{\beta, \gamma}_{\cC}({\bf{y}}, \tau_0) +\frac{\sqrt{2(n-k)}(|\mathbf{y}|^2-2k)}{4|\tau_0|},|\mathbf{y}|^2-2k\Big\rangle_\cH=0,
    \end{equation}
if we choose $\tau_{0}$ negative enough and $\kappa>0$ small enough depending on $\mathcal{M}$.
Note that $\mathcal{W}(\mathcal{M})$ is diagonal for any $\mathcal{M}\in \mathcal{A}^{k, \circ}$ thanks to the $\mathbb{Z}_2^k$-symmetry and in particular  $\mathcal{W}(\mathcal{M}^1)$ and $\mathcal{W}(\mathcal{M}^2)$ are diagonal matrices and they have the same diagonal entries up to some subgroup of ${\bf S}_k$-permutation. If $\mathcal{W}(\mathcal{M}^1)$ is a multiple of the identity matrix, then applying Theorem \ref{spectral uniqueness restated} (spectral uniqueness) we see that $\mathcal{M}^1=\mathcal{M}^2$ is the unique $\mathrm{O}(k)\times \mathrm{O}(n+1-k)$-symmetric $k$-oval, and if $\mathcal{W}(\mathcal{M}^1)$ has  distinct eigenvalues then, taking also into account again the $\mathbb{Z}_2^k$-symmetry, we see that either $\mathcal{M}^2=\mathcal{M}^1$ or $\mathcal{M}^2$ is obtained from $\mathcal{M}^1$ by the rotation induced by the action of some discrete subgroup  of  $\mathbf{S}_{k}$. This shows that the induced map $\bar{q}: \mathcal{A}^{k, \circ}/\mathbf{S}_{k}\to \mathcal{X}$ is bijective. Notice that $\mathbf{S}_{k}$ is properly discontinuous homeomorphism group action on the locally compact Hausdorff space $\mathcal{A}^{k, \circ}$, where the locally compact Hausdorff property of $\mathcal{A}^{k, \circ}$ follows from Theorem \ref{spectral uniqueness restated} (spectral uniqueness) and continuity of spectral map $\mathcal{E}$,  so $\mathcal{A}^{k, \circ}/\mathbf{S}_{k}$ is a Hausdorff space and the quotient map is open map. Hence, by definition of the quotient topology, the map $\bar{q}$ is also continuous and so is a homeomorphism. 

\end{proof}

\begin{claim}\label{metrizable}
  The space $\mathcal{A}^{k, \circ}$ of ancient ovals with $\mathbb{R}^k\times S^{n-k}$ as its tangent flow at $-\infty$  endowed with locally smooth convergence is a $k-1$-dimensional open manifold and is metrizable. 
\end{claim}

\begin{proof}[proof of Claim \ref{metrizable}]
 For every $\mathcal{M}_0\in \mathcal{A}^{k, \circ}$, by Conjecture \ref{oval-koval} it is a $k$-oval and by Corollary \ref{prop_continuous_orthogonality} we can make $\kappa>0$ small and $\tau_{0}\leq \tau_\ast$ negative enough such that there is  a open neighborhood $U(\mathcal{M}_0)=U(\mathcal{M}_0, \kappa, \tau_*)$ of $\mathcal{M}_0$ and for every element $\mathcal{M}\in U(\mathcal{M}_0)$, there are $\beta=\beta(\mathcal{M}), \gamma=\gamma(\mathcal{M})$  continuously depending on $\mathcal{M}$ and $\mathcal{M}^{\beta,\gamma}$ is $\kappa$-quadratic at $\tau_0$  and its profile function satisfies the following orthogonality conditions,
 \begin{equation}
\mathfrak{p}_{+}\big(v^{\beta,\gamma}_{\cC}({\bf{y}}, \tau_{0})-\sqrt{2(n-k)}\big)=0,
\end{equation}
\begin{equation}\label{ON betagamma y2-2k}
         \Big\langle v^{\beta,\gamma}_{\cC}({\bf{y}}, \tau_0) +\frac{\sqrt{2(n-k)}(|\mathbf{y}|^2-2k)}{4|\tau_0|},|\mathbf{y}|^2-2k\Big\rangle_\cH=0,
    \end{equation}
and the transformation gives homeomorphism from $U(\mathcal{M}_0)$ to its image $U'(\mathcal{M}_0)$.
  Moreover, possibly after decreasing $\tau_\ast$ and choosing $\kappa$ small we can arrange that
\begin{equation}
\mathscr{E}(\mathcal{M}^{\beta,\gamma}_0)(\tau_0)\in \cylinders_{\frac{\delta\kappa}{|\tau_0|}}(\bar{e}).
\end{equation}
where 
\begin{equation}
\bar{e}=\frac{\sqrt{2(n-k)}\||\mathbf{y}|^2-2k\|^2_{\mathcal{H}}}{4|\tau_0|},
\end{equation}
and if we consider the spectral ratio map
\begin{equation}
\mathcal{E}(\mathcal{M}^{\beta,\gamma})(\tau_0)=  \left(\frac{\langle v^{\beta,\gamma}_{\cC}(\tau_{0}), y^2_{1}-2 \rangle_\cH}{\langle v^{\beta,\gamma}_{\cC}(\tau_{0}), |{\bf{y}}|^2-2k \rangle_\cH}, \dots,\frac{\langle v^{\beta,\gamma}_{\cC}(\tau_{0}),  y^2_{k}-2 \rangle_\cH}{\langle v^{\beta,\gamma}_{\cC}(\tau_{0}), |{\bf{y}}|^2-2k \rangle_\cH}\,\right),
\end{equation}
where the denominator of each coordinates is a constant depending on $\tau_0$ by \eqref{ON betagamma y2-2k}, we have
\begin{equation}
\bar{q}=\mathcal{E}(\mathcal{M}^{\beta,\gamma}_0)\in \Delta_{c_k\delta\kappa}.
\end{equation}
where $\Delta_{c_k\delta\kappa}$ is a $k-1$ dimensional $c_k\delta\kappa$ size scaling of the standard simplex $\{x\in \mathbb{R}^k_+: x_1+\dots+x_k=1\}$ with center at $p_k=(1/k,\dots, 1/k)$ and  the constant $c_k$ is given by $c_k=[\sqrt{2(n-k)}\||\mathbf{y}|^2-2k\|^2_{\mathcal{H}}]^{-1}$.

Then,  we consider the continuous map $\mathcal{E}(\tau_0): U'(\mathcal{M}_0)\to \Delta_1$. By
Theorem \ref{spectral uniqueness restated} (spectral uniqueness) it is a injective map. Moreover for its inverse $\mathcal{E}^{-1}(\tau_0): V= \mathcal{E}(\tau_0)(U'(\mathcal{M}_0)) \to U'(\mathcal{M}_0)$, it is also a continuous map. If there are $q_i\in V$ converges to $q_{\infty}\in V$, but the sequence of flows $\mathcal{M}^{i}=\mathcal{E}^{-1}(\tau_0)(q_i)$  converges to $\bar{\mathcal{M}}^{\infty}\not={\mathcal{M}}^{\infty}=\mathcal{E}^{-1}(\tau_0)(q_\infty)$. By the definition of $\kappa$-quadraticity at $\tau_0$,  the limit ${\mathcal{M}}^{\infty}$ of  $\mathcal{M}^{i}$ which are $\kappa$-quadratic at $\tau_0$ is still $\kappa$-quadratic at $\tau_0$. By continuity of $\mathcal{E}(\tau_0)$, we have 
\begin{equation}
    \mathcal{E}(\tau_0)(\bar{\mathcal{M}}^{\infty})=\lim_{i\to \infty}\mathcal{E}(\tau_0)({\mathcal{M}}^{i})=\lim_{i\to \infty} q_i=q_{\infty}= \mathcal{E}(\tau_0)(\mathcal{M}^{\infty})
\end{equation}
 By 
Theorem \ref{spectral uniqueness restated} (spectral uniqueness), $\mathcal{M}^{\infty}=\bar{\mathcal{M}}^{\infty}$, which is a contradiction. Therefore  $\mathcal{E}^{-1}(\tau_0): V= \mathcal{E}(\tau_0)(U'(\mathcal{M}_0)) \to U'(\mathcal{M}_0)$ is also a continuous map. By Corollary \ref{prescribed_eccentricity_restated'} (existence with prescribed spectral value) and spectral ratio map $\mathcal{E}$ being a composition of maps $p\circ\mathscr{E}\circ\pi^{-1}$, there is a small open set  $W$ containing $\bar{q}$ such that the open set $\mathcal{E}^{-1}(W)\subset U'(\mathcal{M}_0)$, otherwise a similar sequence convergence  argument with
Theorem \ref{spectral uniqueness restated} (spectral uniqueness) resulting in a contradiction. Hence, $\mathcal{E}^{-1}(\tau_0): W\to \mathcal{E}^{-1}(\tau_0)(W)\subset U'(\mathcal{M}_0)$ composing with the inverse transformation of time-shift and parabolic dilation gives a  local homeomorphism coordinates
chart in an open neighborhood of $\mathcal{M}_0\in \mathcal{A}^{k, \circ}$. Moreover, by local homeomorphism of spectral map $\mathcal{E}$ given above and Theorem \ref{spectral uniqueness restated} (spectral uniqueness), $\mathcal{A}^{k, \circ}$ is Hausdorff. Because the ancient ovals in $\mathcal{A}^{k, \circ}$ are convex, we can use radial graph representation to regard 
$\mathcal{A}^{k, \circ}$ as a subspace of the mapping spaces of $C^{\infty}(S^{n}\times (-\infty, 0], \mathbb{R})$ endowed with  compact topology for locally smooth convergence. Note that $C^{\infty}(S^{n}\times (-\infty, 0], \mathbb{R})$ is second countable due to the second countable Hausdorff manifolds and local compactness properties of $S^{n}\times (-\infty, 0]$ and $\mathbb{R}$ (see \cite[Section 4]{kundu2007countability}, \cite[Ex 3.4H]{engelking1989general}, \cite{michael1961theorem}). Hence we know that the subspace $\mathcal{A}^{k, \circ}$ of $C^{\infty}(S^{n}\times (-\infty, 0], \mathbb{R})$ is second countable. Hence, $\mathcal{A}^{k, \circ}$ is an $(k-1)$-dimensional open topological manifold, and the local coordinate map is given by spectral ratio map $\mathcal{E}$ and its inverse. Moreover, $\mathcal{X}$ obtained from $\mathcal{A}^{k, \circ}$ by modulo the  proper discrete group ${\bf S}_k$ action is also an $(k-1)$-dimensional open topological manifold. Moreover, compactified space $\mathcal{A}$ of $\mathcal{A}^{\circ}$  is metrizable by convex radial representation again and noting that $C^{\infty}(S^{n}\times (-\infty, 0], \mathbb{R})$ with local smooth convergence topology is metrizable by hemi-compactness  of $S^{n}\times (-\infty, 0], \mathbb{R}$ (namely it has a sequence of compact subsets such that every compact subset of the space lies inside some compact set in the sequence, see \cite[Section 7]{arens1946topology}, \cite[Example IV.2.2]{conway1994course} also \cite[Section 1.1, 1.2]{treves2016topological} and \cite[Part I, Fréchet spaces, examples]{infusinotopological} and it holds for the locally smooth convergence compact open topology as well). The second countable  property of $\mathcal{A}^{k, \circ}$ follows from compact metric space property of the compactified space $\mathcal{A}^{k}$. 
\end{proof}

\begin{claim}\label{claim5.3}
 $\mathcal{A}^{k, \circ}$   is contractible and in particular is path connected and simply connected. 
\end{claim}
\begin{proof}[proof Claim \ref{claim5.3}]
By Corollary  \ref{prescribed_eccentricity_restated'}  (existence with prescribed spectral value), for any compact set $K\subset \mathcal{A}^{k, \circ, \prime}$ (the time shifted homeomorphism image of $\mathcal{A}^{k, \circ}$), we can find $\kappa>0$ small and $\tau_{0}$ negative such that $K$ is contained inside a closed topological ball
\begin{equation}
K\subset \mathcal{E}^{-1}(\tau_0)(\bar{\simplex}_{\frac{\delta\kappa}{2c}}(1))\subset \mathcal{E}^{-1}(\tau_0)(
\simplex_{\frac{\delta\kappa}{c}}(1)).
\end{equation} 
$\mathcal{A}^{k, \circ}$ and $\mathcal{A}^{k, \circ, \prime}$ (up to time shift homeomorphism) is exhaustion of increasing  sequence of closed topologically balls $\mathcal{E}^{-1}(\tau_j)(\bar{\simplex}_{\frac{\delta\kappa}{2c}}(1))$ and also  exhaustion of increasing open topologically balls 
 $\mathcal{E}^{-1}(\tau_j)(
\simplex_{\frac{\delta\kappa}{c}}(1))$ for some sequence of $\tau_j\to -\infty$ and every of them is contractible. In particular, $\mathcal{A}^{k, \circ}$  is $\sigma$-compact.  We also recall that a metrizable space $\mathcal{A}^{k, \circ}$ is normal space. Therefore by the result in \cite{ancel2016monotone}, we have
 $\mathcal{A}^{k, \circ}$  is contractible. In particular, the quotient space $\mathcal{A}^{k, \circ}$  is contractible so it is path connected and simply  connected.
  \end{proof}

\begin{claim}\label{claim5.4}
        When $k\geq 3$, $\mathcal{A}^{k, \circ}$ is simply connected at infinity.
    \end{claim}
    \begin{proof}[proof Claim \ref{claim5.4}]
Because $\mathcal{A}^{k, \prime}\cong \mathcal{A}^{k, \circ}$ by time shift map. We only need to show   $\mathcal{A}^{k, \prime}$ is simply connected at infinity.   For any compact subset $K$ of $\mathcal{A}^{k, \prime}$, by compactness, $\kappa$-quadraticity at $\tau_0$, Theorem \ref{spectral uniqueness restated} (spectral uniqueness), Corollary \ref{prescribed_eccentricity_restated'}  (existence with prescribed spectral value),  we can find $\kappa>0$ and $\tau_0$ negative enough such that $\mathcal{E}(\tau_0)(K) \subset 
\simplex_{\frac{\delta\kappa}{2c}}(1)$. Let $D=\mathcal{E}(\tau_0)^{-1}(\bar\simplex_{\frac{\delta\kappa}{c}}(1))$. For any loop $\gamma \in \mathcal{A}^{k, \prime}- D$, by compactness of loop $\gamma$  and possbily by decreasing $\tau_0$ to some $\tau_1$. We may assume there exists a $C>0$, such that the loop $\mathcal{E}(\tau_1)(\gamma)\subset \simplex_{\frac{C\delta\kappa}{c}}(1)-\bar\simplex_{\frac{\delta\kappa}{c}}(1)$. Therefore, by simply connectedness of difference of two scaled simplexes $\simplex_{\frac{C\delta\kappa}{c}}(1)-\bar\simplex_{\frac{\delta\kappa}{c}}(1)$, we know that there is a continuous homotopy $H$ deforming $\mathcal{E}(\tau_1)(\gamma)$ to a point in $\simplex_{\frac{C\delta\kappa}{c}}(1)-\bar\simplex_{\frac{\delta\kappa}{c}}(1)$. Pulling back this homotopy by $\mathcal{E}(\tau_1)^{-1}$, we know that $\gamma\in \mathcal{A}^{k, \prime}- D$ can be continuously deformed into a point in $\mathcal{A}^{k, \prime}- D$. In particular, $\pi_{1}(\mathcal{A}^{k, \prime}-D)=0$  and the inclusion induced fundamental group homeomorphism
    \begin{equation}
        i_*:\pi_{1}(\mathcal{A}^{k, \prime}-D)\rightarrow  \pi_{1}(\mathcal{A}^{k, \prime}-K).
    \end{equation}
is a zero map and $\mathcal{A}^{k, \circ}$ is simply connected at infinity.
\end{proof}

Hence, by using Claim \ref{metrizable}, Claim \ref{claim5.3},  Claim \ref{claim5.4}  and the following topology facts about the classification of $1$-dimensional open connected manifold in case $k=2$, Riemann mapping theorem in case $k=3$ and in the case $k\geq 4$, we apply the topological fact that  any open contractible $(k-1)$-manifold that is simply-connected at infinity is homeomorphic to  $\mathbb{R}^{k-1}$ (see  \cite{wall1965open, edwards1963open, husch1969finding}
for $k=4$ case, and  \cite{freedman1982topology} for $k=5$ case, and\cite{stallings1962piecewise} for $k\geq 6$ cases for the related results) and we obtain homeomorphism 
\begin{equation}\label{5.21}
\mathcal{A}^{k, \circ} \cong \mathrm{Int}(\Delta_{k-1}). 
\end{equation}
Hence 
\begin{equation}
\mathcal{X}\cong \mathcal{A}^{k, \circ} / \mathbf{S}_k\cong \mathrm{Int}(\Delta_{k-1}) / \mathbf{S}_k
\end{equation}
follows from Claim \ref{claim5.1}, \eqref{5.21} and modulo the discrete proper action of $k$ coordinates permutation induced rigid motion subgroup $\mathbf{S}_k\subset O(k)$. This completes the proof of the proof of Theorem \ref{cor_moduli} (classification of ancient ovals and moduli space).

\end{proof}

\bibliography{kovals}

\bibliographystyle{alpha-short}

\vspace{5mm}

{\sc Beomjun Choi, Department of Mathematical Sciences,KAIST, Daejeon, Korea}

{\sc Wenkui Du, Department of Mathematics, Massachusetts Institute of Technology,  Massachusetts, 02139, USA}

{\sc Ziyi Zhao,Institute for Theoretical Sciences, Westlake University, China}

\vspace{5mm}

\emph{E-mail:} bchoi@kaist.ac.kr, 
wenkui.du@mit.edu, zhaoziyi@westlake.edu.cn.

\end{document}